\documentclass[a4paper]{article}

\usepackage[utf8]{inputenc}

\usepackage[backend=biber, maxbibnames=10, bibencoding=utf8, giveninits=true]{biblatex}
\usepackage{graphicx, mathtools, amsmath, amsfonts, amsthm, amssymb, latexsym,
verbatim, pifont, xspace, mathrsfs, tabularx, xcolor}
\usepackage{calc}
\usepackage{booktabs}
\usepackage{multirow}
\usepackage[shortlabels]{enumitem}
\usepackage{subfigure}
\usepackage{yfonts}
\usepackage{adjustbox}
\usepackage{anyfontsize}
\usepackage{xspace}

\usepackage[colorlinks, hypertexnames=false]{hyperref}
\hypersetup{linkcolor=blue, urlcolor=blue, citecolor=red}
\usepackage[capitalise]{cleveref}

\usepackage{todonotes}
\setuptodonotes{inline}

\makeatletter
\define@key{todonotes}{FH}[]{%
    \setkeys{todonotes}{author=FH,color=green}}%
\define@key{todonotes}{JDM}[]{%
    \setkeys{todonotes}{author=JDM,color=cyan}}%
\define@key{todonotes}{AA}[]{%
    \setkeys{todonotes}{author=AA,color=blue!50,noinline,size=\scriptsize}}%
\define@key{todonotes}{JCN}[]{%
    \setkeys{todonotes}{author=JCN,color=yellow}}%
\makeatother

\usepackage{tikz}
\usepackage{tikz-cd}
\usetikzlibrary{decorations, arrows}
\usetikzlibrary{arrows,decorations.markings}
\usetikzlibrary{arrows.meta}
\usetikzlibrary{petri,positioning,decorations.pathmorphing}
\usetikzlibrary{graphs} 
\usepackage{genyoungtabtikz}
\newcommand\Ygr{\Yfillcolour{gray!20}}
\newcommand\Yw{\Yfillcolour{white}}

\usepackage{mahjong}
\usepackage{ytableau}
\usepackage[edges]{forest}
\usepackage{pdflscape}

\pgfdeclaredecoration{sl}{initial}{
  \state{initial}[width=\pgfdecoratedpathlength-1sp]{
    \pgfmoveto{\pgfpointorigin}
  }
  \state{final}{
    \pgflineto{\pgfpointorigin}
  }
}

\pgfarrowsdeclarecombine{aa}{aa}{angle 90}{angle 90}{angle 90}{angle 90}

\newcommand{\NN}{\mathbb{N}}    
\newcommand{\Pp}{\mathcal{P}}   
\newcommand{\R}{\mathcal{R}}    

\newcommand{\A}{\mathcal{A}}
\newcommand{\B}{\mathcal{B}}
\newcommand{\C}{\mathcal{C}}

\newcommand{\M}{\mathbf{M}}
\newcommand{\Nn}{\mathbf{N}}
\newcommand{\Com}{\mathbf{Com}}
\renewcommand{\L}{\mathscr{L}}  
\newcommand{\J}{\mathscr{J}}    
\newcommand{\G}{\mathcal{G}}    

\newcommand{\id}{\mbox{\rm id}}     
\newcommand{\set}[2]{\ensuremath{\{\: #1 \: :\: #2 \:\}}}   
\renewcommand{\to}{\longrightarrow} 
\newcommand{\im}{\operatorname{Im}}     
\newcommand{\dom}{\operatorname{Dom}}   
\newcommand{\cont}{\operatorname{Cont}} 
\newcommand{\ev}{\operatorname{ev}}     
\newcommand{\RW}{\operatorname{R}}      

\newcommand{\stal}{\mathbf{Stal}}

\newcommand{\plax}{\mathbf{Plax}}
\newcommand{\styl}{\mathbf{Styl}}
\renewcommand{\P}{\operatorname{P}}
\newcommand{\N}{\operatorname{N}}

\newcommand{\chinese}{\mathbf{Ch}}
\newcommand{\dyck}{\mathbf{Dyck}}       
\newcommand{\cdf}{\mathbf{CDF}}         

\newcommand\dyckpath[8]{
  \begin{scope}[local bounding box=#4]
    \fill[white]  (#1) rectangle +(#2,#2);
    \fill[white] (#1) foreach \dir in {#3}{-- ++(\dir*90:1)} |- (#1);
    \path[fill] (#1) foreach \num [count=\i from 0] in {#5}{ +(\i.5,-.6) node[anchor=south]{\num} \ifnum\i>#2 circle (1pt) \fi};
    \path[fill] (#1) foreach \num [count=\i from 0] in {#5}{ +(#2.1,\i.5) node[anchor=west]{\num} \ifnum\i>#2 circle (1pt) \fi};
    \draw[help lines] (#1) grid +(#2,#2);
    \draw[line width=2pt] (#1) foreach \dir in {#3}{ -- ++(\dir*90:1)};
  \end{scope}
}

\newcommand{\f}{f}

\newcommand{\fp}{fp}
\newcommand{\emptyf}{\tilde f}
\newcommand{\rd}{g}
\newcommand{\rdb}{g'}
\newcommand{\gd}{h}

\renewcommand{\gathered}{regrouped\xspace}

\newcommand{\HPl}{\mathbf{HP}}
\newcommand{\Q}{\operatorname{Q}}

\newcommand{\Hlleft}{\xleftarrow{\HPl_2}}

\newcommand{\Sylv}{\mathbf{Sylv}}
\newcommand{\T}{\operatorname{T}}

\newcommand{\HS}{\mathbf{HS}}
\renewcommand{\H}{\operatorname{H}}

\newcommand{\MS}{\mathbf{MS}}
\newcommand{\Mm}{\operatorname{M}}

\newcommand{\Met}{\mathcal{M}}

\NewEnviron{mosaic}[2][]{\begin{tikzpicture}[baseline]\def\tilesize{#2}\matrix[mosai,ampersand replacement=\&, #1]{\BODY};\end{tikzpicture}}

\newcommand{\tileedge}{.05pt} 
\newcommand{\knotthickness}{1.5pt} 

\newcommand{\tilevi}[1][{}]{\tikz {\draw[tilestyle] (0,0) rectangle (\tilesize,\tilesize); \draw[knotstyle={\tilesize*\knotthickness}] (.5*\tilesize,0) to (.5*\tilesize,\tilesize);}}

\tikzset{tilestyle/.style={use as bounding box, line width=\tileedge}} 
\tikzset{knotstyle/.style={line width=#1}} 
\tikzset{mosai/.style={execute at begin cell=\node\bgroup, execute at end cell=\egroup;, column sep={\tilesize cm,between origins}, row sep={\tilesize cm,between origins}, inner sep=0}}

\newtheorem{thm}{Theorem}[section]
\newtheorem{prop}[thm]{Proposition}
\newtheorem{cor}[thm]{Corollary}
\newtheorem{lemma}[thm]{Lemma}
\newtheorem{algo}[thm]{Algorithm}

\newtheorem{defi}[thm]{Definition}
\newenvironment{de}{\begin{defi} \rm}{\end{defi}}
\newtheorem{exam}[thm]{Example}

\newtheorem{remark}[thm]{Remark}

\newtheorem{note}[thm]{Note}
\numberwithin{equation}{section}

\title{Power quotients of plactic-like monoids}
\author{Antoine Abram
\thanks{ULiège, Liège, Belgium, \href{mailto:antoine.abram@uliege.be}{antoine.abram@uliege.be}}
\and
Florent Hivert
\thanks{LISN, Université Paris-Saclay, Orsay, France, 
\href{mailto:florent.hivert@lri.fr}{florent.hivert@lri.fr}}
\and
James D. Mitchell
\thanks{School of Mathematics and Statistics, University of St Andrews, St Andrews, Scotland, 
\href{mailto:jdm3@st-andrews.ac.uk}{jdm3@st-andrews.ac.uk}}
\and
Jean-Christophe Novelli
\thanks{LIGM, Université Gustave-Eiffel, Marne-la-Vallée, France, 
\href{mailto:novelli@univ-mlv.fr}{novelli@univ-mlv.fr}}
\and
Maria Tsalakou,}

\date{\today}

\begin{document}
\maketitle

\begin{abstract}
In this paper we describe the quotients of several plactic-like monoids by
the least congruences containing the relations $a^{\sigma(a)} = a$ for
$\sigma(a) \geq 2$ for every generator $a$.
The starting point of this paper is the recent paper of Reutenauer and
the first author about the stylic monoid. This monoid is the quotient of the
plactic monoid by the relations $a^2 = a$ for every generator $a$. The
plactic-like monoids considered here are the plactic monoid itself, the
Chinese monoid, the hypoplactic monoid, the sylvester monoid, the
metasylvester and hyposylvester monoids.
In each case we describe a set of normal forms and the idempotents. We obtain
formulae for their cardinalities and we describe their isomorphism types.
\end{abstract}

\tableofcontents


\section{Introduction}\label{section-introduction}

Given an alphabet $\A$, the \textit{plactic monoid} $\plax(\A)$ is the
quotient of the free monoid $\A ^*$ by the equivalence relation that
relates words having the same $\P$-tableau when the Robinson-Schensted
Algorithm~\cite{R, S} is applied. Knuth~\cite{K} found a presentation for
$\plax(\A)$ where $\A$ is a totally ordered finite alphabet $(\A, <)$:
\begin{equation*}
  acb  =  cab \quad \text{if } a \leq b < c \quad\text{and}\quad
  bac =  bca \quad \text{if } a < b \leq c, \quad \text{for all }a, b, c\in \A. 
\end{equation*}
The plactic monoid was studied by Lascoux and Sch\"utzenberger in
\cite{Lascoux1981aa}. Due to its connection with symmetric functions and
representation theory, the plactic monoid is a central object in algebraic
combinatorics that has been widely studied in the literature. It has a rich
combinatorial structure, and has revealed many connections between seemingly
unrelated concepts. A survey on the plactic monoid and its applications is
given by Lascoux, Leclerc and Thibon (Chapter 5 of Lothaire’s book
\cite{Lothaire2002aa}).

More recently, Reutenauer and the first author published a paper~\cite{AR}
where they define and study the \textit{stylic monoid}, the quotient of the
plactic monoid by the relations $a^2=a$ for every letter $a$.
They show that it is a finite monoid and that it has an interesting structure
and many intriguing properties; for example, if $|\A|=n$, its cardinality is
equal to the ($n+1$)-st Bell number.

From this, one can ask if it is combinatorially sound to quotient the plactic
monoid by relations $a^{\sigma(a)}\equiv a$ with $\sigma$
any function from the alphabet to $\NN_{\geq2}$ and if the original idea and
its above generalization are also sound as quotients of other monoids ?

Since the point was to study these questions from the combinatorial point of
view, we restricted ourselves to monoids sharing deep combinatorial or
algebraic connections with the plactic monoid. These monoids, refered as
\emph{plactic-like monoids}, have rich combinatorial structures and relations
to representation theory and geometry.

\subsection{The club of plactic-like monoids}

The \textit{Chinese monoid} $\chinese(\A)$ was first defined in \cite{DK}
where the authors study monoids having the same growth function as the
plactic monoid. It appears in the study of Hecke atoms and the Bruhat
order~\cite{HMP}.
Its insertion algorithm, defined in~\cite{CASSAIGNE2001}, involves tableau-like structure called Chinese
staircases.
A quotient of the plactic monoid called the \textit{hypoplactic monoid}
$\HPl(\A)$ was defined by Krob and Thibon~\cite{KT} in their study of the
representations of the degeneracy at~$q=0$ of the Iwahori-Hecke algebra and
the quantum linear group.
This monoid plays the same role for non-commutative and quasi-symmetric
functions as the plactic monoid does for symmetric functions. For a detailed
study of this monoid see~\cite{N}.
The \textit{sylvester monoid} $\Sylv(\A)$ is obtained using an insertion
algorithm into binary search trees~\cite{HNT05}.
The sylvester monoid is related, among other things, to the associahedra and
the Loday-Ronco algebra of trees~\cite{LR02}.
In the study of Hopf algebras of $(m+1)$-ary trees and $m$-permutations,
Thibon and the fourth author defined two extra monoids: the
\textit{metasylvester} and \textit{hyposylvester} monoids, denoted by
$\MS(\A)$ and $\HS(\A)$ respectively.
These monoids are used to generalize non-commutative symmetric and
quasi-symmetric functions and are respectively a coarsening and a refinement
of the sylvester monoid.

\subsection{Quotienting the plactic-like monoids}

Let $\M(\A)$ be any of the plactic-like monoids above.
We shall investigate the following general quotients: the quotients
$\M(\A, \sigma)$ of $\M(\A)$ by the relations $a^{\sigma(a)} =a$ for every
$a\in \A$ where $\sigma:\A\to \NN_{\geq 2}$ is any function.
We call such quotients \emph{$\sigma$-quotients of $\M(\A)$} or in a more
general way, \textit{power-quotients}.
Surprisingly, all of the $\sigma$-quotients of the plactic-like monoids given
above are finite and the function $\sigma$ uniquely determines the
isomorphism-type of the monoid.
Also, they can be defined in terms of insertion algorithms into combinatorial
objects that are, in each case, related to the corresponding objects of
$\M(\A)$.
Using these algorithms, we study the structure of each $\M(\A,\sigma)$ by
in particular computing their cardinalities and describing their idempotents;
see \cref{table-summary} for details of where to find the relevant results. 

In every case, the quotients $\M(\A, 2)$, where $\sigma(a)=2$ for all
$a\in\A$, play a key role in the structure of the general quotients
$\M(\A, \sigma)$: describing the structure of $\M(\A, 2)$ is indeed very
helpful to describe the structure of all the $\M(\A, \sigma)$.

During the first stages of this research, we found that our monoids satisfy
two kinds of very specific properties that notably simplify their studies and
decided to name these properties as defining monoids of the first or second
type.

The first type, that contains the plactic, the Chinese, and the
hypoplactic monoids, $\M(\A, \sigma)$ can be naturally embedded in the
cartesian product of $\M(\A, 2)$ with the commutative monoid
$\Com(\A, \sigma)$, the latter being the $\sigma$-quotient of the free
commutative monoid over $\A$.

The second type, that includes the hypoplactic, the sylvester, the
hyposylvester and the metasylvester monoids, the original monoid $\M(\A)$ has
a particular set of normal forms that makes it easy to apply the relations
$a^{\sigma(a)}=a$.

For both types of monoids, we provide a general framework for analysing these
in terms of $\M(\A, 2)$, and for monoids of the first type, we derive a
formula for their size.

Type 2 monoids are somewhat less tractable.
Using the particular normal form of $\M(\A)$, one can find normal forms for 
$\M(\A, \sigma)$ but this is not enough to fully understand these quotients,
even when $\sigma$ is a constant function.
For example, we do not know a formula for the cardinalities of
$\Sylv(\A,\sigma)$ and $\MS(\A,\sigma)$ when $\sigma$ is not constant and,
even though we can describe the idempotents in $\Sylv(\A,\sigma)$ and
$\MS(\A,\sigma)$ when $\sigma(a)\neq 2$, we do not know a formula for their
numbers. On the other hand, we expect that the techniques developed in
\cref{section-types} can be applied to other classes of plactic-like monoids
of type 1.

\begin{table}
\centering
  \begin{tabular}{l|l|l|l|l|l}
  Monoids & Comb. Obj. & Types & Cardinality & Idempotents & $\cong$ \\ \hline
  $\plax(\A)$
  & \ref{thm:N-tableau}
  & \ref{cor-plax-type1}
  & \ref{cor-plax-size}
  & \ref{prop-plax-idempotents}
  & \ref{plax-iso}\\
  $\chinese(\A)$
  & \ref{cor-main-chinese}, \ref{2chifunct}
  & \ref{cor:M-Type}
  & \ref{cor-sigma-chinese-size} 
  & \ref{prop-sigma-chinese-idempotents} 
  & \ref{chinese-iso}\\
  $\HPl(\A)$
  & \ref{sqrandhsm}
  & \ref{cor-hypo-gathered}, \ref{hypoplax-type1}
  & \ref{hypo-enum} 
  & \ref{HPsigmaIdem}
  & \ref{Hpl-iso} \\
  $\Sylv(\A)$
  & \ref{redtrees}
  & \ref{cor-sylv-gathered}
  & \ref{sylv-enum}
  & \ref{2sylv-idem} 
  & \ref{sylv-iso}\\
  $\HS(\A)$
  & \ref{hypotree-w}
  & \ref{cor-hyposylv-gathered}
  & \ref{hyposylv-enum} 
  & \ref{hyposylv-idem} 
  & \ref{hyposylv-iso} \\
  $\MS(\A)$
  & \ref{metasylv-s}
  & \ref{cor-metasylv-gathered}
  & \ref{meta-enum} 
  & \ref{prop:MS2idem} 
  & \ref{meta-iso}
\end{tabular}
\caption{Summary of the locations of the main results in the paper.}
\label{table-summary}
\end{table}

\subsection{Outline of the paper}

The paper is organised as follows. After recalling some fundamental
combinatorial objects and providing some notations in \cref{section-notation},
we define the general quotients $\M(\A,\sigma)$ and both types of monoids in
\cref{section-types}. The general strategy of the paper is detailed on a toy
example, the stalactic monoid defined in~\cite{HNT05-comm}, that is of both
types of monoids.
We then consider each of the plactic-like monoids in their own sections:
plactic in \cref{section-plax}; Chinese in \cref{section-chinese}; hypoplactic
in \cref{section-hypoplax}; sylvester in \cref{section-sylv}; hyposylvester in
\cref{section-hyposylv}; and metasylvester in \cref{section-metasylv}.

After proving that $\plax(\A)$ is of type 1, the structure of the
quotients $\plax(\A, \sigma)$ can be derived from the results in
\cref{ss:Type1Monoids} and from the description of the structure of the stylic
monoid $\styl(\A)=\plax(\A,2)$ from~\cite{AR}.

Concerning the Chinese monoid, the first essential prerequisite, a description
of $\chinese(\A,2)$, does not appear in the literature. Thus,
\cref{section-chinese} begins with a deep study of $\chinese(\A, 2)$ using two
different combinatorial objects: the $2$-Chinese staircases
(\cref{subsection-2-chinese-staircase}) and the $2$-Chinese functions
(\cref{subsection-2-chinese-functions}).
Using these objects, we proceed to a deeper study of $\chinese(\A, 2)$ by computing its cardinality,
which is directly related to the Catalan numbers, describing its idempotents and, showing that $\chinese(\A, 2)$ is
$\J$-trivial (\cref{ss-2-chinese-more}).
Using the description of $\chinese(\A,2)$, we prove that $\chinese(\A)$ is of
type 1 (\cref{ss-chinese-type}), and use this to exhibit the structure of
$\chinese(\A,\sigma)$ (\cref{ss-chinese-enum-idem-iso}).

The hypoplactic monoid $\HPl(\A)$, studied in \cref{section-hypoplax}, is both
of type 1 and type 2. Being a quotient of the plactic monoid, $\HPl(\A,2)$
inherits the $\J$-triviality of $\plax(\A,2)$ and we describe its $\J$-order
in \cref{2-Hypo-Jtriv}.
The structure of the other $\sigma$-quotients $\HPl(\A,\sigma)$ follows
naturally.

The first monoid we consider that is not of type 1 is the sylvester monoid,
$\Sylv(\A)$. In \cref{SylvWord}, we show that the right-to-left postfix
readings of binary search trees are a set of normal forms for $\Sylv(\A)$ and
this allows us to show that $\Sylv(\A)$ is of type 2
(\cref{cor-sylv-gathered}). 
In \cref{ss:2-sylv}, we exhibit some connections between enumerative properies
of $\Sylv(\A, 2)$ and the sequences of Borel numbers and of the large
Schr\"oder numbers.

In \cref{section-hyposylv}, we study the hyposylvester monoid, $\HS(\A)$. Like
its sibling the sylvester monoid, it is only of type 2. The cardinality of its
$\sigma$-quotients is related to the number of $k$-faces in an $n$-cube; see
\cref{hyposylv-reductree}.

Finally, we study the metasylvester monoid $\MS(\A)$ which is also only of
type 2 in \cref{section-metasylv}.

\section{Notations and basic definitions}
\label{section-notation}

In this paper, we denote finite totally ordered alphabets using calligraphic
letters such as $\A$, $\B$ and $\leq$ for the alphabet order.
The free monoid generated by $\A$ is denoted by $\A^*$ and the free
commutative monoid generated by $\A$ is denoted by $\Com(\A)$.
We denote letters by lower case Latin letters, such as $a$, $b$, $x$, and
words by capital Latin letters, such as $W$, $U$, $V$.
We denote monoids by boldface letters, such as $\M$, $\Nn$, or specific ones
like $\plax$.
The sets of consecutive integers ${a,a+1,\dots,b}$ (intervals) are written
$[a,b]$ and if $a$ is the smallest letter of $A$, the interval $[a,b]$ is
simplified as $[b]$ and called an \emph{initial} interval.

\paragraph{Words}

Let $W\in \A ^*$ be a word over the alphabet $\A$.
The \emph{length of} $W$ is denoted by $|W|$ and the number of
occurrences of a letter $a\in\A$ by $|W|_a$.
The \emph{content} of $W$ is the set of distinct
letters occurring in $W$, and is denoted by $\cont(W)$.
The \emph{evaluation} of $W\in \A ^*$ is the image of $W$ in $\Com(\A)$ under
the natural homomorphism.
In other words, if $\A = \set{a_i}{i \in I}$ and $W\in\A ^*$,
then $\ev(W)=a_1^{\delta_1}\cdots a_n ^ {\delta_n}$ where $\delta_i=|W|_{a_i}$
and $\cont(W)=\set{ a_i\in\A}{\delta_i>0}$.

Since $\A$ is totally ordered by $<$, we say that a word
$W= w_1 \cdots w_n\in \A^ *$ is \emph{non-decreasing} if
$w_1 \leq w_2 \leq \cdots \leq w_n$.
The word $W$ is \emph{increasing} if all previous inequalities are strict.

If $W= w_1 \cdots w_n\in \A^ *$ is an arbitrary word, an \emph{inflation} of
$W$ is a word $W^\sharp$ of the form
$w_1 ^ {\delta_1} \cdots w_n ^ {\delta_n}\in \A ^ *$ for some $\delta_1,
\delta_2, \ldots, \delta_n\in \NN^*$.

Denote by $\Theta_\A$ the unique order-reversing permutation of $\A$.
It extends uniquely to an anti-automorphism of $\A^*$, that we still denote
$\Theta_\A$ (or only by $\Theta$ if there is no ambiguity).
For example, with $A = \{a < b < c < d\}$, $\Theta(acdaadc) = baddabd$.
The mapping $\Theta$ is an involution.
For $\B\subset\A$, we will also denote
$\Theta_\A(\B)=\set{\Theta(b)\in\A}{b\in\B}$.

\paragraph{Combinatorial Objects}

A \emph{composition of $n$} is a list of positive integers
$I=(i_1,i_2,\cdots,i_p)\models n$ such that $\sum_{j=1}^p i_j=n$. The
\emph{weight} of $I$ is $n$ and its \emph{length} is $p$. The
\emph{descent set of} $I$ is the sequence $D(I)=\{d_1,\cdots,d_{p-1}\}$,
where $d_k=\sum_{j=1}^k i_j$. We say that $J$ is \emph{finer} than
$I$ if and only if $D(I)\subseteq D(J)$.

A \emph{partition of $n$} is a decreasing composition
$\lambda=(i_1\geq i_2\geq\cdots\geq i_p)$, noted $\lambda\vdash n$.
Let $\lambda=(i_1,i_2,\cdots,i_p)$ and
$\gamma=(j_1,j_2,\cdots,j_q)$ be two partitions. Then $\lambda$ is
\emph{contained} in $\gamma$ if $p\leq q$ and $i_k\leq j_k$ for all
$1\leq k\leq p$. We then write $\lambda\leq\gamma$.

Given a partition $\lambda=(i_1,i_2,\cdots,i_p)$, the \emph{Ferrers diagram of
shape} $\lambda$ is the collection of boxes such that its rows are
left-justified and their lengths are, from bottom to top (French notation),
$i_1,i_2,\cdots,i_p$. The \emph{size} of a Ferrers diagram is the weight of its
partition, \emph{i.e.}, the number of boxes.
Given two partitions $\lambda$ and $\gamma$ such that $\lambda<\gamma$, a
\emph{skew Ferrers diagram of shape} $(\lambda,\gamma)$ is the collection of
boxes obtained by deleting the boxes of the Ferrers diagram of shape
$\lambda$ inside the diagram of $\gamma$.

Given a composition $I$, a \emph{ribbon diagram of shape} $I$ is a skew
diagram such that row $j$ has $i_j$ consecutive boxes that are placed starting
just below the last box of the previous row. It is a special case of a skew
Ferrers diagram. For example, the (filled) ribbon diagram on \cref{insere-d}
has shape $(1,3,1,2,2)$ (read the sizes of the rows from top to bottom).

We shall also consider other combinatorial objects, namely trees.
These are always rooted but not necessarily binary and are represented with
the root at the top.

\paragraph{Monoids}

Asume that $\M(\A)$ is a monoid defined by the presentation
$\langle\A|\R\rangle$ with the generating set $\A$ and the relations $\R$.
If $U, V\in \A^*$ have the same image under the natural surjective
homomorphism $\pi_{\M}: A ^ *  \to \M$, then we say that $U$ and $V$ are
\emph{equivalent} in $\M(\A)$. If so, we write $U \equiv_{\M(\A)} V$.
We may also write $U \equiv_{\M} V$ if there is no ambiguity.
In any finitely-presented monoid $\M(\A)$, we refer to a word $U\in A^*$ as
being a \emph{shortlex normal form} for $V\in \A^*$ if $U$ is the
lexicographically smallest element among the shortest words of its class.

A useful set of tools for understanding the structure of a monoid are the
\emph{Green relations}. Given a monoid $\M$ and two elements $u,v\in\M$, we
have
\begin{align*}
    u\J v  \Leftrightarrow & \M u\M =\M v\M,\\
    u\L v  \Leftrightarrow & \M u =\M v.\\
\end{align*}
One can define the \emph{$\J$-preorder} (reflexive and transitive) $\leq_\J$
by letting $u\leq_\J v$ if $u\in\M v\M$.
A monoid $\M$ is \emph{$\J$-trivial} if $\M u\M=\M v\M$ if and only if $u=v$.
This can be expressed as if $u\in \M v\M$ and $v\in \M u\M$ then $u=v$.
If $\M$ is $\J$-trivial, than $\leq_\J$ is a partial order on $\M$ called the
\emph{$\J$-order}.
For the $\L$-relation, one defines similar concepts by replacing $\M x \M$ by
$\M x$ in the previous definitions.
We refer to Chapter 3 of \cite{P} for more details on these concepts.

If $\R\subseteq \A ^* \times \A ^ *$ has the property that
$\cont(U) = \cont(V)$ for all $(U, V) \in \R$, we say that the
presentation $\langle \A | \R\rangle$ is \emph{content-preserving}. 
If $\langle \A | \R\rangle$ is content-preserving, then we may, without
ambiguity, refer to the \emph{content} of $W\in \M(\A)$ and write it
$\cont(W)$.
In particular, the map $\cont: \langle \A | \R\rangle \to \Pp(\A)$,
where $\Pp(\A)$ is the power monoid of $\A$ with product $\cup$ and identity
$\varnothing$, is a homomorphism.

Similarly, we define \emph{evaluation-preserving} presentations if
$\ev(U) = \ev(V)$ for all $(U, V) \in \R$, and we write $\ev(W)$ for the
evaluation of $W\in \M(\A)$ when $\M(\A)$ is defined by an
evaluation-preserving presentation. When this is the case, $\ev: \langle \A | \R \rangle \to \Com(\A)$ is a homomorphism.
Note that an evaluation-preserving presentation is content-preserving.

\section{The \texorpdfstring{$\sigma$}{s}-Quotients and Types of Monoids}
\label{section-types}

In this section, we describe the $\sigma$-type quotients of general monoids.

While studying the first examples coming from combinatorics, we noticed that
general properties of monoids, whether enumerative or structural, make it
possible to significantly simplify their study and mechanically lead to
enumerative results. Having encountered these properties several times, we
decided to classify our monoids according to these properties, which are of
two non-mutually exclusive types.

The monoids of the first type, that include the plactic, Chinese, and
hypoplactic monoids, have a general quotient structure that is directly
deduced from the structure of $M(A,2)$. Fine enumerative formulas are directly
obtained from this.

The monoids of the second type, that include the hypoplactic monoids and
variations of the sylvester monoid, have a much more complex structure, yet
rigid enough to allow the characterization of a section that behaves well with
respect to the studied quotients. Enumerative questions are more challenging
to carry out there.

To clearly illustrate these various types of monoids, which we hope may be
useful in other contexts, we study throughout this section a toy example: the
stalactic monoid defined in \cite{HNT05}. This monoid has a simple description
but nevertheless allows us to understand how certain of its properties, which
will be those defining type 1 and type 2 monoids, help facilitate their
combinatorial study.

\subsection{Power-Quotient Monoids}\label{ss:QuotientMonoids}

\begin{defi}\label{def:quotient-monoid}
Let $\M(\A)$ be a quotient $\A^*/\equiv_\M$ for a congruence relation
$\equiv_\M$, and $\sigma : \A \to \NN$ such that $\sigma(a) \geq 2$
for all $a\in \A$.
The \emph{$\sigma$-$\M$ congruence}, denoted by $\equiv_{\M(\sigma)}$,
is the least congruence on
$\A ^ *$ containing $\equiv_{\M}$ and such that
$a ^ {\sigma(a)} = a$ for every $a\in \A$.

The \emph{$\sigma$-$\M$ monoid on $\A$},
or \emph{$\sigma$-quotient of $\M(\A)$}, is the quotient
$\A^* / {\equiv_{\M(\sigma)}}$ and is denoted by $\M(\A, \sigma)$.
\end{defi}

In particular, we shall repeatedly use the monoid $\Com(\A, \sigma)$, which is
the quotient of the commutative monoid on $\A$ by the $a ^ {\sigma(a)} = a$ for every $a\in \A$.

Except stated otherwise, we assume that $\sigma: \A \to \NN_{\geq 2}$ is
arbitrary throughout the whole manuscript.
For the sake of simplicity, if $\sigma$ is the constant function with value
$p\geq 2$, we then write $\M(\A, p)$ instead of $\M(\A, \sigma)$.

Note that $\Pp(\A)$, the power quotient that encodes the content of a word
can be naturally identified with $\Com(\A, 2)$.

\subsection{A first example of a power-quotient monoid}
\label{ss:stalac}

\subsubsection{The stalactic monoid}

Let us consider the following relations on a totally order alphabet $\A$:
\begin{equation}\label{eq-stalac}
aWba = aWab, \quad \text{for } a,b \in \A,\ W\in\A^*.
\end{equation}

The least congruence on the free monoid $\A^*$ containing these relations is
the \emph{stalactic congruence} (see~\cite{HNT05}) and is denoted
by $\equiv_{\stal}$. The quotient $\A ^ * /{\equiv_{\stal}}$ is called the
\emph{stalactic monoid} and is denoted by $\stal(\A)$. It is
evaluation-preserving.

It is easy to check that thanks to the relations all equal letters in a word
can be regrouped with their first occurrence so that any word is congruent
to a representative of the form
\begin{equation}
\label{eq-repres-stalac}
a_{i_1}^{\alpha_1} \dots a_{i_p}^{\alpha_p}.
\end{equation}
Moreover, these elements form a section of the monoid since applying one
relation cannot change the order of appearance of the letters inside a
congruence class.

On an alphabet of two letters $\{a,b\}$, let us encode the classes as
follows: the class of the empty word is encoded by $(\epsilon,0,0)$ where
$\epsilon$ is the empty word. And the class whose representative is $\ell^n
\ell'^p$ (where $\ell$ is $a$ or $b$ and $\ell'$ is the other) is encoded by
$(\ell,n,p)$ with the condition $n\not=0$.
For the product we have the following rule,
\begin{equation}
(\ell,n,p).(\ell',n',p') =
\left\{
\begin{array}{rl}
(\ell',n',p'), & \text{if } \ell=\epsilon,  \\
(\ell,n+n',p+p'), & \text{if } \ell'=\ell,  \\
(\ell,n+p',n'+p),& \text{if } \ell'\not=\ell. \\
\end{array}
\right.
\end{equation}.

\subsubsection{The \texorpdfstring{$2$}{2}-quotient of the stalactic monoid}%

Let us now consider the $2$-quotient of $\stal(\A)$, again with $\A=\{a,b\}$.
Note that this quotient is not evaluation-preserving anymore but it is
content-preserving.
The class of the empty word is still present and all other classes reduce to
four distinct classes:
\begin{equation}
(a,1,0), \quad (b,1,0), \quad (a,1,1), \quad\text{and}\quad (b,1,1).
\end{equation}
Indeed, the class $(\ell,n,0)$ is now directly equivalent to $(\ell,1,0)$ and
the class $(\ell,n,p)$ with $p\not=0$ is equivalent to $(\ell,1,1)$ for any
$\ell$ in $\A$.

The reader may recognize that $\stal(\A, 2)$ is the free left regular band generated by $\A$ (see~\cite[Example 1.3]{Brown2000}).

\subsubsection{The general \texorpdfstring{$\sigma$}{s}-quotients of the
stalactic monoid}%
As in the previous case, with an alphabet of two letters $\{a,b\}$ and the
relations $a^\alpha=a$ and $b^\beta=b$, we still have the class of the empty
word, the classes $(a,n,p)$ with
$n\in[1,\alpha-1]$ and $p\in[0,\beta-1]$,
and the classes $(b,n,p)$ with
$n\in[1,\beta-1]$ and $p\in[0,\alpha-1]$.

Computing the $\sigma$-quotient of any monoid seems rather complicated in
general but it is easy in this case since we have a set of representatives of
classes that is stable under the relations $\ell^{\sigma_\ell}=\ell$. We shall
give later a precise definition with extra conditions but this is the general
idea of monoids of type $2$.

Moreover, if one wants to enumerate the classes of $\stal(\A,\sigma)$, one
should partition the classes according to their content.
In the case of the stalactic monoid with two letters, one gets that
\begin{equation}
|\M(\A,\sigma)|= 1 + (\alpha-1) + (\beta-1) + 2(\alpha-1)(\beta-1),
\end{equation}
where the partition consists in the class of the empty words, then words with
only $a$s, then words with only $b$s, then words with both $a$s and $b$s.

With a general alphabet $\A$ and a general congruence $\sigma$, one would get
the cardinality
\begin{equation}
|\M(\A,\sigma)|
=\sum\limits_{\B\subset\A} m_{|\B|}\prod\limits_{b\in\B}(\sigma(b)-1).
\end{equation}
where $m_{|\B|}=|\B|!$.
Indeed, the number of classes of content $\B$ is the product of $|\B|!$
(choose the order of appearance of the elements of $\B$) by
$\prod\limits_{b\in\B}(\sigma(b)-1)$ (choose independently a power for
every $b\in \B$ in the interval $[1,\sigma_b-1]$).

Note that the factorial amounts to consider the classes as classes of
$\M(\B,2)$ and that the powers amounts to consider the classes as classes
of $\Com(\B,\sigma)$. The compatibility between these two monoid morphisms is
informally the definition of monoids of type $1$.

\subsection{Type 1 Monoids}\label{ss:Type1Monoids}

Let $\M(\A)$ be an evaluation-preserving monoid.
Note that, for all $a\in \A$ and for any $p\geq 2$, $a^p \equiv_{\M(2)}a$.
So $\M(\A,2)$ is a homomorphic image of $\M(\A, \sigma)$ for every $\sigma$.

The monoid $\M(\A, \sigma)$ is content-preserving but not
evaluation-preserving, and so $\ev(M)$ for $M\in \M(\A, \sigma)$ is not
well-defined.
However, one can define
$\ev_{\sigma}: \M(\A, \sigma) \to \Com(\A, \sigma)$ as the unique
surjective homomorphism extending the identity map from $\A$ to $\A$. It will
play a similar role in this context as $\ev: \M(\A) \to \Com(\A)$.
Denote by $\theta$ the homomorphism from $\M(\A, \sigma)$ to $\M(\A,2)$
extending the identity function from $\A$ to $\A$.

Let $\phi_\sigma$ be the product map $\theta \times \ev_\sigma$.

\begin{defi}\label{def:type1}
Let $\M(\A)$ be an evaluation-preserving monoid.
Then $\M(\A)$ is of \emph{type 1} if $\phi_\sigma$ is injective for
all $\sigma$.
\end{defi}

Consider a monoid $\M(\A)$ of type $1$, a word $W$, and a letter
$\ell$ in the content of $W$. Then repeating $\sigma(\ell)$ times any occurrence of $\ell$ in $W$
yields the same element by $\phi_\sigma$ so all these words belong to the same
$\M(\A,\sigma)$ class. In other words, any word $\ell^nU\ell^p$ with $n$ and
$p$ nonzero is equivalent to $\ell U\ell^{n+p-1}$ in any $\M(\A,\sigma)$ if
$\M$ is of type $1$.

Let us now characterize the image of $\phi_\sigma$.

\begin{lemma}
\label{lem-type1}
Let $\M(\A)$ be a monoid of type $1$.
Then the image of $\phi_\sigma$ are pairs consisting in an element of
$\M(\A,2)$ and an element of $\Com(\A,\sigma)$ of the same content.
\end{lemma}

\begin{proof}
First, the image of any element of $\M(\A,\sigma)$ is in the prescribed set.
Now consider any element $W$ of $\M(\A,2)$ as an element of $\M(\A)$ and
consider an element $C$ of $\Com(\A,\sigma)$ with the same content $\B$.
For all letters $\ell$ in $\B$, pick any occurrence of $\ell$ in $W$ and
change it into $\ell^k$ with $k>0$ so that the resulting word $W'$ satisfies
$\ev_\sigma(W')=C$. This is always possible since the $\sigma$-evaluation is
essentially computed modulo $\sigma_\ell\!-\!1$. Then $\phi_\sigma(W')=(W,C)$
since $\ell^k\equiv_{\M(\A,2)}\ell$ for any nonzero $k$ and any letter $\ell$.
\end{proof}
  
We shall see later that the plactic monoid, the hypoplactic monoid, and the
Chinese monoid are of type $1$.
As in the toy example of the stalactic monoid, if a monoid is of type $1$, one
can derive enumerative results with an extra hypothesis that will be satisfied
by our monoids.

\begin{thm}\label{thm:Type1-Enum}
Let $\A$ be a finite alphabet and $\M(\A)$ a monoid of type 1.
Partition the classes of $\M(\A,2)$ according to their content.
Assume that the number of classes of content $\B$, where $\B$ is of
cardinality $k$ only depends on $k$ and denote that number by $m_k$.
Then

\begin{equation}
|\M(\A,2)|=\sum\limits_{k\geq 0}^{|\A|} \binom{|\A|}{k}m_k.
\end{equation}

Moreover, the cardinality of $\M(\A,\sigma)$ is
\begin{equation}
|\M(\A,\sigma)|
 = \sum\limits_{\B\subset\A} m_{|\B|}\prod\limits_{b\in\B}(\sigma(b)-1).
\end{equation}

In particular, if $\sigma(x)=p$ for all $x\in\A$, then
\begin{equation}
|\M(\A,p)|=\sum\limits_{k\geq 0}^{|\A|} \binom{|\A|}{k}m_k(p-1)^k.
\end{equation}
\end{thm}

\begin{proof}
The first equation is direct given the hypothesis since the number of
subalphabets $\B$ of $\A$ of cardinality $k$ is $\binom{|\A|}{k}$.

Now, consider a subalphabet $\B$ of $\A$.
Since $\M(\A)$ is evaluation-preserving, we can apply $\phi_\sigma$
to the subset of $\M(\A,\sigma)$ of words of content $\B$.
Since $\phi_\sigma$ is bijective with its image and since it is
given by pairs consisting in an element of $\M(\B,2)$ and an element
of $\Com(\B,\sigma)$ both of content $\B$, these elements contribute to
$m_{|\B|} \prod\limits_{b\in\B}(\sigma(b)-1)$ disjoint classes, hence the
result by summing over all possible subalphabets $\B$.

The last formula is a direct corollary of the previous one.
\end{proof}

The stalactic monoid satisfies the condition that the number of classes of
$\M(\A)$ of content $\B$ only depends on the cardinality of $\B$ and that
$m_k=k!$.

\subsection{Type 2 Monoids}
\label{ss:GatheredMonoids}

As mentionned before, the plactic monoid is a monoid of type $1$. But, even
with monoids close to it in the algebraic sense, some are not of type $1$. For
example, as we shall see later, the sylvester monoid is not of type $1$.
But it still has a nice trait that will allow us to easily characterize a
section of its $\sigma$ quotients. This trait is shared by other monoids and
we regroup these monoids under the name of monoids of type $2$.

Note that our types are not mutually exclusive since both the stalactic and
hypoplactic monoids are of both types (see \cref{ss-hypoplac-type}).

\begin{de}
Let $\M(\A)$ be an evaluation-preserving monoid. 
We say that $E\in \A ^ *$ is a \emph{\gathered element}
if for any $W\in \A ^*$ such that $W\equiv_{\M(\A)} E$ and any $a\in\cont(E)$
\begin{equation}
W = U'a^2V' \quad\Rightarrow\quad  E = Ua^ 2V
\end{equation}
for some words $U, V, U', V'\in \A ^*$ and such that $|U|_a=|U'|_a$.
\end{de}

In other words, if the $k$-th and the $k+1$-st occurrence of a given letter
are consecutive somewhere in the congruence class, then they have to be
consecutive in the elements called the \gathered elements.

If an $\equiv_{\M}$-class contains a \gathered element, we refer to the
lexicographically largest \gathered word $G(W)$ in the class as the
\emph{canonical \gathered element}. Note that in an evaluation-preserving
monoid, every $\equiv_{\M}$-class is finite, and so $G(W)$ exists if there are
any \gathered words in the class of $W$.
Denote $\G_{\M}$ the set of canonical \gathered elements.

Given a word $W$, let $a \in \cont(W)$ and pick an integer $i\in [1,|W|_a]$.
The \emph{$(a,i)$-expansion} of $W$ is the word obtained from
$W$ by duplicating the $i$-th occurrence of $a$ in $W$.

\begin{de}\label{def:type2}
An evaluation-preserving monoid $\M(\A)$ is \emph{\gathered} (or of type 2) if:
\begin{enumerate}[(1)]
\item each $\equiv_{\M}$-class has a \gathered element;
\item for any word $W$, for any pair $(a, i)$ with $a \in \cont(W)$ and
$i\in[1,|W|_a]$,
the $(a,i)$-expansion of $G(W)$ is equal to the canonical \gathered element of
the $(a,i)$-expansion of $W$.
\end{enumerate}
\end{de}

We will see that the hypoplactic, the sylvester, the hyposylvester and the
metasylvester monoids are \gathered monoids.

Let $\sigma:\A\to\N_{\geq 2}$.
For $W\in\A^*$, $G_\sigma(W)$ is the word obtained from $G(W)$
by repeatedly replacing any factor $a^{\sigma(a)}$ of $G(W)$ by $a$, until
there are no such factors anymore. The result is independent of the
order of the replacements so this uniquely defines the
\emph{$\sigma$-reduced word} $G_\sigma(W)$ of $G(W)$ (and, thus, of $W$).
We will denote by $\G_{M(\sigma)}$ the set of $\sigma$-reduced words of
$\M(\A,\sigma)$.

Note that thanks to the second condition of \cref{def:type2}, $G_\sigma(W)$ is
the canonical \gathered element of its $\M(\A)$ class: if it was not, its
canonical \gathered element would be another word and by applying the same
expansions to both, $G(W)$ would not be the canonical \gathered element of its
class which is impossible by definition.

\begin{lemma}
\label{qr-equiv-simp}
Let $\M$ be a \gathered monoid. Then for any word $W$, we have
$G_\sigma(W)\equiv_{\M(\sigma)}W$.
\end{lemma}

\begin{proof}
We know that $G(W)\equiv_{\M} W$ thus $G(W)\equiv_{\M(\sigma)} W$. By
definition, $G_\sigma(W)\equiv_{\M(\sigma)} G(W)$, which concludes the proof.
\end{proof}

\cref{qr-equiv-simp} shows that any word is $\M(\A, \sigma)$-equivalent to a
$\sigma$-reduced word. Let us now show that no two $\sigma$-reduced words
are $\M(\A, \sigma)$-equivalent by considering individual $(a,i)$-expansions.

\begin{thm}\label{gen-gathered}
Let $\M(\A)$ be a \gathered monoid.
Then the $\sigma$-reduced words form a set of normal forms
for $\M(\A, \sigma)$.
\end{thm}

\begin{proof}
We already know that the $\sigma$-reduced words intersect all
congruence classes of $\M(\A)$ by \cref{qr-equiv-simp}.

Let $W$ and $W'$ two equivalent words by $\M(\A, \sigma)$. We will then show
that $G_\sigma(W) = G_\sigma(W')$.

Consider a word $W$ and assume that it is rewritten into $W'$
by an application of an $\M(\A, \sigma)$ elementary rewriting.
There are three possibilities:
either apply a relation from the relations of $\M$, expand $W$ by applying
$\sigma(a)-1$ times the $(a,i)$-expansion or contract $W$ by applying
$\sigma(a)-1$ times the $(a,i)$-contraction.

In each case, the $\sigma$-reduced word of $W'$ will be equal to
$G_{\sigma}(W)$ since either $G(W')=G(W)$
or $G(W')$ is an expansion of $G(W)$ that reduces to $G(W)$ by one
contraction or a contraction of $G(W)$ obtained from $G(W)$ by one
contraction.

Therefore all $\M(\A, \sigma)$-equivalent words have the same $\sigma$-reduced
word and since the reducing algorithm applied to a reduced element is itself,
they all belong to different classes.
\end{proof}

In the stalactic monoid, there is only one word in each class that is
\gathered{}, the canonical element shown in Equation~\eqref{eq-repres-stalac},
where all equal letters are consecutive.
The set of these elements is stable by expansion and contraction hence the
stalactic monoid is of type $2$.
We can then conclude that the $\sigma$-reduced elements are a section of
$\stal(\A,\sigma)$ for any $\sigma$.

\subsection{Other general properties}
\label{ss:blabla}

Given the simple structure of the stalactic monoid, it is immediate that
$\stal(\A,\sigma)$ is isomophic to $\stal(\A,\sigma')$ if $\sigma'$ is a
permutation of $\sigma$. This property is not true with any of the other
monoids studied in this paper: given such a monoid $\M$, the quotients
$\M(\A,\sigma)$ are all different monoids. 
Since this property is already true when $\A=\{a,b\}$, one can concentrate on
the relations only involving two different letters. It appears that all the
monoids considered in the sequel have very specific two-letter relations.

We begin with a technical lemma that will be used in the general proof.
\newpage

\begin{lemma}\label{lemma-antoine}
Assume that $\M$ and $\M'$ are two monoids respectively defined by the
content-preserving presentations $\langle \A | \R\rangle$ and
$\langle \A | \R'\rangle$ and that $\psi :\M\to \M'$ is a homomorphism. Then
there exists a unique homomorphism $\psi'$ such that the following diagram
commutes:
\begin{equation}
    \begin{tikzcd}
     \mathbf{M} \arrow[r, "\psi"] \arrow[d, "\cont"] & \M' \arrow[d, "\cont"] \\
     \Pp(\A) \arrow[r, "\psi'"] & \Pp(\A)
   \end{tikzcd}
\end{equation}
\end{lemma}

\begin{proof}
   We define
    \[
    \psi'(\{a_1, \ldots, a_n\}) =  \bigcup_{i = 1}^ n \cont(\psi(a_i))
    \]
    where $a_1, \ldots, a_n\in \A$ is arbitrary.
    Since $\cont$ and $\psi$ are homomorphisms, 
    \begin{equation}\label{eq-phi-prime}
    \psi'(\{a_1, \ldots, a_n\}) = \cont(\psi(a_1\cdots a_n)).
    \end{equation}
and so the diagram commutes and $\psi'$ is a homomorphism. Any two
homomorphisms satisfying \eqref{eq-phi-prime} coincide, and so $\psi'$ is
unique. 
\end{proof}

We now concentrate on the very specific relations on words on two letters that
appear in all our monoids.

\begin{thm}
\label{thm:isom}
Let $\M(\A)=\langle \A|\R\rangle$ be a content-preserving presentation,
and let $\sigma,\tau:\A\to\NN_{\geq 2}$ and $\M(\A,\sigma)$, $\M(\A,\tau)$ as
in \cref{def:quotient-monoid}. Then 
\begin{enumerate}[(i)]
\item if for all $a<b\in\A$, $(bab,bba)\in\R$ but $(aba,aab)\notin\R$,
$\M(\A,\sigma)\simeq\M(\A,\tau)$ if and only if $\sigma=\tau$;
\item if for all $a<b\in\A$, $(aba,baa)\in\R$ but $(bab,abb)\notin\R$,
$\M(\A,\sigma)\simeq\M(\A,\tau)$ if and only if $\sigma=\tau$.
\end{enumerate}
\end{thm}

\begin{proof}
Assume that $\psi:\M(\A,\sigma)\to \M(\A,\tau)$ is an isomorphism. Applying
\cref{lemma-antoine} repeatedly, it follows that
\begin{equation}
\psi' \circ (\psi ^ {-1})'
  = (\psi\circ \psi^{-1})' = (\id_{\M(\A, \tau)})' = \id_{\M(\A)},
\end{equation}
where $\id$ denotes the identity function. Hence $\psi'$ is an isomorphism,
and so $\psi'$ induces a permutation of $\A$.
In particular, for all $a\in \A$, $\psi(a)=\psi'(a)^{\alpha_a}$ for some
$\alpha_a\geq 1$.
Moreover $a = \psi(a ^ {\sigma(a)}) = \psi'(a) ^ {\sigma(a)}$ and so
$\sigma(a)$ divides $\tau(\psi'(a))$. Similarly, using $\psi^{-1}$,
$\tau(\psi'(a))$ divides $\sigma(a)$, and so  $\tau(\psi'(a))= \sigma(a)$ for
all $a\in \A$.
Thus if $\A = \{a_1, \ldots, a_n\}$, then  $(\tau(a_1),\ldots,\tau(a_n))$ is a
permutation of $(\sigma(a_1),\ldots,\sigma(a_n))$.

Assume that there exists $x<y\in \A$ such that $\psi(x)>\psi(y)$.

For (i), assume that for all $a<b\in\A$, $(bab,bba)\in\R$ but
$(aba,aab)\notin\R$. Then one would have
$\psi(y)\psi(x)\psi(y)\equiv_{M(\tau)} \psi(y)\psi(y)\psi(x)$ which is a
contradiction.

Similarly for (ii), assume that for all $a<b\in\A$, $(aba,baa)\in\R$ but
$(bab,abb)\notin\R$. Then one would have
$\psi(x)\psi(y)\psi(x)\equiv_{M(\tau)} \psi(y)\psi(x)\psi(x)$ which is a
contradiction.
    
Thus, in both cases, $\psi$ is the identity on $\A$ and $\sigma=\tau$.
\end{proof}


\begin{lemma}
\label{lem:expgenfunct}
Let $\M(\A)=\langle \A|\R\rangle$ be a content-preserving presentation.
Let $S_p(n)=|\M(\A,p)|$ for $|\A|=n$.
If $S_p(n)=\sum_{k=0}^{n} \binom{n}{k})s_k(p-1)^k$ for some $s_k\in\NN$, then
\begin{align}
\sum_{n\geq 0} S_p(n)\frac{z^n}{n!}
  = \sum_{k\geq 0} \frac{s_k((p-1)z)^k}{k!}\exp(z).
\end{align}
\end{lemma}

\begin{proof}
It is a very classical computation on generating functions:
\begin{align*}
\sum_{n=0}^\infty S_{p}(n)\frac{z^n}{n!}
    &= \sum_{n=0}^\infty\sum_{k=0}^{n}
         \binom{n}{k}s_k(p-1)^k\frac{z^n}{n!}\\
    &= \sum_{n=0}^\infty\sum_{k=0}^{n}
         \frac{s_k((p-1)z)^k}{k!}\frac{z^{n-k}}{(n-k)!}\\
    &=\sum_{k=0}^\infty
         \frac{s_k((p-1)z)^k}{k!}\sum_{j=0}^\infty\frac{z^{j}}{j!}\\
    &=\sum_{k=0}^\infty
         \frac{s_k((p-1)z)^k}{k!}\exp(z).\qedhere
\end{align*}
\end{proof}

\section{The \texorpdfstring{$\sigma$}{s}-Plactic Monoids}
\label{section-plax}

\subsection{The Plactic and Stylic Monoids}\label{ss-plax}

The \emph{Knuth relations} (see~\cite{K}) on a totally ordered alphabet $\A$
are
\begin{equation}\label{eq-plactic}
\begin{split}
acb = cab, & \quad \text{for } a \leq b < c\in \A, \\
bac = bca, & \quad \text{for } a < b \leq c\in \A.
\end{split}
\end{equation}
The least congruence on the free monoid $\A^*$ containing the Knuth relations
is called the \emph{plactic congruence} and is denoted by $\equiv_{\plax}$.
The quotient $\A ^ * /{\equiv_{\plax}}$ is called the \emph{plactic monoid};
see \emph{e.g.}, \cite{Lascoux1981aa} or \cite[Chapter 5]{Lothaire2002aa}.
It is evaluation-preserving. 

The \emph{stylic congruence} is the least congruence on $\A ^ *$ containing
$\equiv_{\plax}$ and $a^2 = a$ for every $a \in \A$; we denote it by
$\equiv_{\styl}$. Note that the \emph{stylic monoid} $\styl(\A)$ on the
alphabet $\A$ (first defined in~\cite{AR}) is the quotient of $\A ^ *$ by the
stylic congruence $\equiv_{\styl}$.
It is content-preserving but not evaluation-preserving.

As in \cref{def:quotient-monoid}, for $\sigma: \A \to \NN_{\geq 2}$,
the \emph{$\sigma$-plactic congruence} is the least congruence on
$\A ^ *$ containing the plactic congruence $\equiv_{\plax}$ and
$a ^ {\sigma(a)} \equiv a$ for every $a\in \A$; we denote this by
$\equiv_{\plax(\sigma)}$. The $\sigma$-plactic monoid on $\A$ is the quotient
$\A ^ * / {\equiv_{\plax(\sigma)}}$ and is denoted by $\plax(\A, \sigma)$. In
particular, the original stylic monoid from~\cite{AR} is the
2-plactic monoid, \emph{i.e.}, $\plax(\A, 2)=\styl(\A)$.

\paragraph{Tableaux and \texorpdfstring{$\N$}{N}-Tableaux\\}

\label{sss-tableaux}
Let $\lambda$ be a partition of $n\in \NN$.
A semi-standard \emph{Young tableau} (or Young tableau, for short) of shape
$\lambda$ is a filling of the shape of $\lambda$ with letters of $\A$ so that
rows weakly increase from left to right and columns strictly increase from
bottom to top (French notation).
A Young tableau is \emph{standard} if it has no repeated letters and its
content consists in an initial interval of $\A$.

The \emph{row reading} $\RW(T)$ of a tableau is obtained by reading each row
from left to right starting from the top.
The \emph{content} $\cont(T)$ of a tableau $T$ is the content of its row
reading.

In \cite{S}, Schensted defined left and right insertions of letters in
semi-standard Young tableaux.
For sake of completeness, we recall the \emph{right insertion}.
Consider a row $R$ and a letter $x\in\A$. There are two possible cases.
If $x\geq y$ for all $y\in R$, let $R'$ be the row obtained by adding a box
labelled $x$ at the end of $R$.
Otherwise, let $y=\min\{z\in R: z>x\}$ and $R'$ be the row obtained by
replacing the leftmost $y$ by $x$ in $R$.
We say that $y$ is \emph{bumped} by $x$.
The insertion of $x$ in a semi-standard Young tableau $T$ is done by inserting
$x$ in the bottom row, and recursively inserting the bumped letter in the row
above until there is no bumped letter anymore.

The left insertion is defined similarly by changing the rows by columns and,
replacing strict inequalities with large inequalities and vice-versa in order
to preserve the semi-standard properties.

Using these insertions, one can define the $\P$-tableau of a word
$W=w_1\cdots w_n$, denoted by $\P(W)$, by right-inserting the
letters of $W$ one after another from left to right into the empty tableau.
Thanks to a result of Knuth~\cite{K}, one can equivalently define the plactic
relation as $U\equiv_{\plax} V$ if and only if $\P(U)=\P(V)$.
Moreover, for $W\in\A^*$ and $a\in\A$, $\P(Wa)=\P(W)\leftarrow a$ and
$\P(aW)=a\rightarrow \P(W)$, where the arrows represent respectively the right
and left insertions.

In \cite{AR}, the authors defined an analogous object for the stylic monoid,
called \emph{$\N$-tableau}, that are semi-standard tableaux satisfying the
following extra properties:
\begin{enumerate}[(1)]
    \item the rows are strictly increasing;
    \item each row is contained in the row below.
\end{enumerate}
The \emph{right $\N$-insertion} of $x\in\A$ in row $Y$ of an $\N$-tableau
$T$ is as follows.
If $x\in Y$, do nothing.
If $x\notin Y$, insert $x$ in the correct position of $Y$ so that $Y$ remains
strictly increasing.
Let $y$ be the smallest letter in $Y$ such that $y>x$ (if any). 
If such a letter exists, we say that $y$ is bumped by $x$.
The right $\N$-insertion of $x$ in an $\N$-tableau $T$ is obtained
by $\N$-inserting $x$ in the bottom row. If $y>x$ is bumped, then $y$ is
inserted in the row above and this process stops when no letter is bumped.
Note that the difference with Schensted's insertion is that if $y$ is bumped,
it stays in the row and a copy of $y$ is inserted in the row above it.
If $T$ is an $\N$-tableau and $x\in \A$, then we denote the $N$-tableau
obtained from $T$ by right inserting $a$ by $T\leftarrow x$.
To right $\N$-insert a word $W$ in a $\N$-tableau, $\N$-insert each letter of
$W$ from left to right.

We also denote by $\N$ the function from $\A^*$ to the set of all
$\N$-tableaux that maps a word $W$ to the $\N$-tableau obtained by
$\N$-inserting $W$ in the empty tableau. We then have

\begin{thm}[Theorem 7.1 of~\cite{AR}]
\label{thm:N-tableau}
The mapping $w\to\N(w)$ induces a bijection from the monoid $\styl(\A)$
onto the set of $\N$-tableaux on $\A$.
\end{thm}

\subsection{Quotient Type}\label{ss-type-plax}

If $|\A|\geq 2$, $\plax(\A)$ is not of type 2 because for $a<b$,
$abba\equiv_{\plax}abab\equiv_{\plax}baab$ but they are equivalent to
neither $aabb$ nor $bbaa$. Thus, this class does not contain a \gathered
element.

We shall see that $\plax(\A)$ is of type 1, thus that $\phi$ is injective.
In order to prove that property, we need two technical results on
$\plax(\A,\sigma)$.
Let us first prove that any word $W$ is $\plax(\A,\sigma)$-equivalent to an
inflation of the row reading of its $\N$-tableau.
We then prove that all the inflations of a row reading of an $\N$-tableau that
have the same $\sigma$-evaluation are equivalent in $\plax(\A,\sigma)$.

\begin{lemma}
\label{lem:placticrelgeneralized}
Let $i,j,k\in\NN$ and $a,b,c\in\A$. We have
\begin{enumerate}[(i)]
\item if $a<b\leq c$ and $j>i$, $b^j c^k a^i \equiv_{\plax} b^j a^i c^k$,
\item if $a\leq b<c$ and $j>k$, $a^i c^k b^j \equiv_{\plax} c^k a^i b^j$.
\end{enumerate}
\end{lemma}

\begin{proof}
Property (i) holds since $j>i$ implies that inserting $a^i$ only bumps the
first $i$ occurrences of $b$ of both words.
	
Concerning (ii), since $i+j>j>k$, for both words, all the $c$ are bumped in
the second row and all the $a$ and $b$ stay in the first row.
\end{proof}

\begin{cor}
\label{cor:genrelsigmaplactic}
Let $i,j,k\in\NN$ and $a,b,c\in\A$. We have
\begin{enumerate}[(i)]
\item if $a<b\leq c$, $b^j c^k a^i \equiv_{\plax(\sigma)} b^j a^i c^k$,
\item if $a\leq b<c$, $a^i c^k b^j \equiv_{\plax(\sigma)} c^k a^i b^j$.
\end{enumerate}
\end{cor}

\begin{proof}
We shall see several statements similar to this one throughout the paper. The
idea is always the same: since any letter can be inflated as much as necessary
in the $\sigma$-monoid, inflate it, use a property on the usual monoid and
deflate the value that was inflated.
The first formula is proven the following way:
\begin{equation}
b^ja^ic^k \equiv_{\plax(\sigma)} b^{j+i(\sigma(b)-1)}a^ic^k
          \equiv_{\plax} b^{j+i(\sigma(b)-1)} c^k a^i
          \equiv_{\plax(\sigma)} b^{j}c^ka^i.
\end{equation}
The first and last rewritings come from iterating $i$ times the inflation of
letter $b$ and the middle rewriting comes from
Lemma~\ref{lem:placticrelgeneralized} since the power of $b$ is
$j+i(\sigma(b)-1)$ so is greater than the power of $a$ that is $i$.

The second formula is proven analogously.
\end{proof}

\begin{prop}
\label{prop:2plax2sigmaplax}
Let $U$ and $W$ be two $\plax(\A,2)$-equivalent words.
If $U^\#$ is an inflation of $U$, then there exists an inflation $W^\#$ of
$W$ such that $W^\#\equiv_{\plax(\sigma)}U^\#$.
\end{prop}

\begin{proof}
Since $U\equiv_{\plax(2)}W$, there exists a sequence of 2-plactic relations
going from $U$ to $W$. We only need to check the result for the elementary
2-plactic relations.
	
Assume that $U=V_1bacV_2$ and $W=V_1bcaV_2$, with $a<b<c$.
This means that \mbox{$U^\#=V_1^\# b^j a^i c^k V_2^\#$} for some inflations
$V_1^\#,V_2^\#$ and $i,j,k\in\NN_{>0}$.
By \cref{cor:genrelsigmaplactic}.(i), $U^\#$ is $\plax(\sigma)$ equivalent to
$V_1^\# b^j c^k a^i V_2^\#$, an inflation of $W$.

Using \cref{cor:genrelsigmaplactic}.(ii), one proves the result for the other
elementary plactic relations in a similar way.

Assume that $U=V_1xxV_2$ and $W=V_1xV_2$ for some $x\in\A$. This means that
$U^\#=V_1^\# x^i x^j V_2^\#$ for some inflations $V_1^\#,V_2^\#$ and
$i,j\in\NN_{>0}$. We then have $$U^\#=V_1^\# x^i x^j V_2^\#=V_1^\# x^{i+j}
V_2^\#,$$ an inflation of $W$.
	
Finally, assume that $U=V_1xV_2$ and $W=V_1xxV_2$ for some $x\in\A$.
This means that \mbox{$U^\#=V_1^\# x^i V_2^\#$} for some inflations $V_1^\#,V_2^\#$
and $i\in\NN_{>0}$.
We then have $$U^\#=V_1^\# x^i V_2^\#\equiv_{\chinese(\sigma)}V_1^\# x^i
x^{\sigma(x)-1} V_2^\#,$$ an inflation of $W$.
\end{proof}

\begin{lemma}
Let $b\in \A$ and let $A = a_1\cdots a_n$ and $C = c_1\cdots c_{m}\in \A^*$
be non-decreasing words such that $a_i < b < c_j$ for all $i$ and $j$.
Then if $u\geq n$, and $v\geq m$, the result of the Robinson-Schensted
insertion of $b^uCAb^v$ is the tableau
\begin{equation}
\begin{array}{|c|c|c}
\cline{1-2}
b^{n}&c_1\cdots c_{m}\\
\cline{1-3}
a_1\cdots a_n&\multicolumn{2}{c|}{\hbox{$\quad b^{u+v-n}\qquad$}}\\
\cline{1-3}
\end{array}
\end{equation}
whose row reading is $b^{n}CAb^{u+v-n}$.
\end{lemma}

\begin{proof}
The word $b^{u}C$ is non-decreassing and therefore inserts into a single
row. The insertion of $A$ of length $n\leq u$ in $b^uC$ bumps $n$ letters $b$
which results into
\begin{equation}
\begin{array}{|c|c|c|}
\cline{1-1}
b^{n}\\
\cline{1-3}
a_1\cdots a_n&b^{u-n}&c_1\cdots c_{m}\\
\cline{1-3}
\end{array}.
\end{equation}
Finally, inserting at least $|C|$ letters $b$ bump all the $c_i$, hence
returning the expected tableau.
\end{proof}

\begin{lemma}[Plactic Gathering Lemma]
\label{lemma-plactic-gathering}
Let $b\in \A$, and let $A = a_1\cdots a_n$ and $C = c_1\cdots c_{m}\in \A^*$
be non-decreasing words such that $a_i < b < c_j$ for all $i$ and $j$. Then
\begin{equation}
b^{\alpha+1} CA b \equiv_{\plax(\sigma)} b CA b^{\alpha+1},
\end{equation}
for all $\alpha \geq 1$.
\end{lemma}

\begin{proof}
Thanks to the previous lemma, for sufficiently large integers $k$ and $\ell$,
the words $b^{1+k(\sigma(b)-1)} CAb^{\alpha + 1 + \ell(\sigma(b)-1)}$ and
$b^{\alpha+1+k(\sigma(b)-1)} CAb^{1 + \ell(\sigma(b)-1)}$
have the same tableau and hence are plactic-equivalent.
We then have
\begin{align}
b^{\alpha + 1}CAb
&\equiv_{\plax(\sigma)}
   b^{\alpha + 1+k(\sigma(b)-1)}CAb^{1 + \ell(\sigma(b)-1)}\\
&\equiv_{\plax} b ^ {1+k(\sigma(b)-1)} CAb^{\alpha + 1 + \ell(\sigma(b)-1)}
\equiv_{\plax(\sigma)}
b CA b^{\alpha + 1}\,,
\end{align}
the $\plax(\sigma)$ equivalences coming from the fact that
$b^{1+ (\sigma(b)-1)}\equiv_{\plax(\sigma)} b$.
\end{proof}

\begin{cor}
\label{cor-inflations}
If $T$ is an $\N$-tableau, and $U, V\in \A^*$ both are inflations of the row
reading of $T$ such that $\ev_{\sigma}(U) = \ev_{\sigma}(V)$,
then $U \equiv_{\plax(\sigma)} V$.

In particular, $U$ is equivalent to an inflation of the row reading of $T$
where only the last row has possibly been inflated.
\end{cor}

\begin{proof}
By definition of the $\N$-tableaux, if $b\in \A$ belongs to a row of $T$, then
it belongs to every row below that one in $T$.
In particular, the word between two consecutive occurrences of the same letter
in the row reading of $T$ is composed of a word $C=c_1\dots c_m$ with $c_j>b$
for all $j$ followed by a word $A=a_1\dots a_n$ with $a_i<b$ for all $i$.

Hence, by iterating \cref{lemma-plactic-gathering}, if $U\in \A^*$ is an
inflation of the row reading of $T$, then $U$ is equivalent in $\plax(\A,
\sigma)$ to the inflation $U'$ of the row reading of $T$ where all rows but
the last are not inflated and the values of the last row are inflated in such
a way that $\ev(U)=\ev(U')$. This is always possible since all letters in $T$
appear in the last row.

Hence if $V\in A^*$ is an inflation of the row reading of $T$ such that
$\ev_{\sigma}(U) = \ev_{\sigma}(V)$, then
\begin{equation}
U \equiv_{\plax(\sigma)} U' \equiv_{\plax(\sigma)} V' \equiv_{\plax(\sigma)} V,
\end{equation}
the second equivalence coming from the fact that $U'$ and $V'$ coincide except
on the last row where, $U$ and $V$ having same $\sigma$-evaluation, they only
differ by multiples of $\sigma(x)-1$ for all letters $x$.
\end{proof}

This property has two very interesting corollaries.

\begin{cor}
\label{cor-plax-type1}
The plactic monoid is of type $1$.
So the image of $\phi_\sigma$ consists in the pairs of $(T,E)$
of an $\N$-tableau and a commutative word of the same contents.
\end{cor}

\begin{proof}
Recall that a monoid is of type $1$ if, for any $\sigma$, the map
$\phi_\sigma$ sending $\M(\A,\sigma)$ to the pair $\M(\A,2)\times
\Com(\A,\sigma)$ is injective.

In our case, it amounts to sending a $\sigma$-plactic class $U$ to the pair
$(T,E)$ consisting of its $N$-tableau $T$ and its commutative image $E$.
Consider now $U$ and $V$ having the same image $(T,E)$ by $\phi_\sigma$.
Thanks to \cref{prop:2plax2sigmaplax}, we know that $U$ being an inflation
of itself, $U$ is $\sigma$-plactic equivalent to an inflation $U'$ of $T$. The
same hold for $V$ with an inflation $V'$ of $T$. Now, since $U$ and $V$ have
the same $\sigma$-evaluation, $U'$ and $V'$ also have same $\sigma$-evaluation
and are both inflation of the same word. Thanks to \cref{cor-inflations}, we
know that they are $\sigma$-equivalent, hence equal.
So $\phi_\sigma$ is injective and the plactic monoid is of type $1$.

By general results on type $1$ monoids, the second property holds
(see~\cref{lem-type1}).
\end{proof}

\begin{cor}
\label{cor-sigplax-section}
A section of the $\sigma$-plactic monoid is given by the inflations of
the row readings of the $N$-tableaux where only the last occurrence of each
letter may have been inflated and such that the power of the last occurrence
of any letter $a$ is strictly smaller that $\sigma(a)$.
\end{cor}

\begin{proof}
As already seen before, any word is $\sigma$-equivalent to an inflation of an
$N$-tableau and since all inflations are $\sigma$-equivalent, one can choose
that only the last occurrence of each letter is possibly inflated.
Now, among those words, one can ask for the shortest one and since all powers
are essentially computed modulo $\sigma(a)-1$, the result follows.
\end{proof}

The canonical word $\RW_\sigma(T, E)$ of a $\sigma$-plactic class $U$ whose
image by $\phi_\sigma$ is $(T,E)$ is computed by the following procedure:
start with the row reading $\RW(T)$ of the $\N$-tableaux $T$.
For every $a\in \cont(\RW(T))$, if  $\alpha$ is the number of occurrences of
$a$ in $\RW(T)$ and $\beta$ is the exponent of $a$ in $E$, then we replace
the last occurrence of $a$ in $\RW(T)$ by $a^\gamma$ where $\gamma$ is the
least positive integer such that $\gamma + \alpha - 1 = \beta
\pmod{\sigma(a)-1}$.

For example, if $\A = \{a,b,c\}$ and $\sigma$ is constant with value $4$,
then
\begin{equation}
\RW_\sigma\left(\vcenter{\hbox{\gyoung(b,ab)}}, a^1b^1\right)
  = bab ^ 3
\qquad\text{and}\qquad
\RW_\sigma\left(\vcenter{\hbox{\gyoung(bc,abc)}}, a^3b^2c^1\right)
  = bca^3bc^3\,.
\end{equation}

\subsection{The \texorpdfstring{$\sigma$}{s}-Plactic Monoids}
\label{ss-plax-enum-idem-iso}

As a direct corollary of \cref{cor-plax-type1}, \cref{thm:Type1-Enum}, and
\cite[Corollary 8.5]{AR}, one has

\begin{cor}
\label{cor-plax-size}
If $\A$ is a finite alphabet and $\sigma: \A \to \NN$ is any function such
that $\sigma(a)\geq 2$ for all $a\in \A$, then
\begin{equation}
\label{eq-sigma-stylic-size}
|\plax(\A, \sigma)|
   = \sum_{\B\subseteq \A}
     \left(b_{|\B|} \prod_{x\in \B}\left(\sigma(x) - 1\right)\right),
\end{equation}
where $b_{k}$ is the $k$-th Bell number. If $\sigma(a) = p$
for every $a\in \A$, then \eqref{eq-sigma-stylic-size} simplifies into
\begin{equation}
\label{eq-sigma-stylic-size-uniform}
|\plax(\A, p)|
  = \sum_{k = 0}^{|\A|} \left(\binom{|\A|}{k}b_k (p-1)^k\right)\,.
\end{equation}
If we define $S_{p}(n) = |\plax(\A, p)|$ for any $\A$ of size $n$, then the
exponential generating function of the $S_{p}(n)$ is given by
\begin{equation}
\sum_{n=0}^{\infty} S_{p}(n) \frac{z^n}{n!} = \exp(z + \exp((p-1)z) - 1)\,.
\end{equation}
\end{cor}

\begin{proof}
Recall that the exponential generating series for Bell numbers is
$\sum_n b_n\frac{z^n}{n!} = \exp(\exp(z) -1)\,.$ Using this with
\cref{lem:expgenfunct} and \cref{eq-sigma-stylic-size-uniform}, one gets
\begin{equation}
\sum_{i=0}^\infty S_{p}(n) = \exp(\exp((p-1)z) - 1)\exp(z)\,. 
\end{equation}
\qedhere
\end{proof}

One can find in Table~\ref{table-sigma-stylic-sizes} the first seven values of
$\plax(\A,p)$ for $p$ between $2$ and $6$.

\begin{table}[ht]
   \[
\begin{array}{@{}l*{9}{r}@{}}
\toprule
\multirow{2}{*}{$p$} & \multicolumn{7}{c@{}}{|\A|}\\
    \cmidrule(l){2-8}
    & 2 & 3 & 4 & 5 & 6 & 7 & 
    \text{OEIS}\\
\midrule
  2& 5& 15& 52& 203& 877& 4,140& 
  \href{https://oeis.org/A000110}{A000110} \\
  3& 13& 71& 457& 3,355& 27,509& 248,127& 
  \href{https://oeis.org/A126390}{A126390}\\
  4& 25& 199& 1,876& 20,257& 245,017& 3,266,914& 
  \href{https://oeis.org/A284859}{A284859}\\
  5& 41& 429& 5,329& 75,989& 1,215,481& 21,453,693& 
  \href{https://oeis.org/A285064}{A285064}\\
  6& 61& 791& 12,196& 216,151& 4,300,781& 94,482,336
  \\
\bottomrule
\end{array}
\]
\caption{The sizes of $\plax(\A,p)$ for $|\A|$ up to $7$ and $p$ up to $6$.}
\label{table-sigma-stylic-sizes}
\end{table}


Let us conclude this section dedicated to the $\sigma$-plactic monoids
by some monoidal properties.

Given $\B=\{b_1 < b_2 < \cdots < b_k\}\subseteq \A$, let
$e_{\B}$ be $b_k^{\sigma(b_k)-1}\cdots b_1^{\sigma(b_1)-1}$.

We have 
\begin{equation}
 \phi_\sigma(e_{\B})
 = \left(b_k \cdots b_1, b_1^{\sigma(b_1)-1} \cdots b_k^{\sigma(b_k)-1}
   \right),
\end{equation}
  with
\begin{equation}
  \N(b_k \cdots b_1) = 
    \newcommand\hT{1.2}
    \newcommand\lT{1.5}
      \raisebox{0.9cm}{\Yboxdim{16pt}
        \gyoung(<b_k>,%
        |\hT\vdts:\ddots,%
      <b_3>_\lT\hdts;<b_k>,%
      <b_2><b_3>_\lT\hdts;<b_k>,%
      <b_1><b_2><b_3>_\lT\hdts<b_k>,%
      )}.
\end{equation}

\begin{prop}
\label{prop-plax-idempotents}
The idempotents of $\plax(\A, \sigma)$ are precisely the $e_{\B}$ for every
$\B\subseteq \A$.
In particular, there are $2^{|\A|}$ idempotents in $\plax(\A, \sigma)$. 
\end{prop}

\begin{proof}
Since $\phi_\sigma$ is injective (see~\cref{cor-plax-type1}),
$X\in\plax(\A, \sigma)$ is idempotent if and only if $\phi_\sigma(X)$ is an
idempotent.

Hence $X$ is an idempotent if and only if $\theta(X)$ and $\ev_{\sigma}(X)$
are idempotents respectively in $\plax(\A, 2)$ and $\Com(\A, \sigma)$.
By \cite[Theorem 12.1]{AR}, the idempotents of $\plax(\A, 2)$ are precisely 
the $b_k\cdots b_1$ where $b_1 < b_2 < \cdots < b_k \in \A$. 
One easily checks that the idempotents of $\Com(\A, \sigma)$ are precisely
$c_1^{\sigma(c_1)-1} \cdots c_\ell^{\sigma(c_\ell)-1}$,
where $c_1, c_2, \ldots, c_\ell\in \A$.
Combining both results, one gets the statement.
\end{proof}

We finally have the following corollary of \cref{thm:isom}.(i).

\begin{cor}
\label{plax-iso}
Let $\sigma, \tau:\A \to \NN_{\geq2}$.
Then $\plax(\A,\sigma)\cong\plax(\A,\tau)$ if and only if $\sigma=\tau$.
\end{cor}

\section{The \texorpdfstring{$\sigma$}{s}-Chinese monoids}
\label{section-chinese}
\subsection{The Chinese Monoid}\label{ss-chinese}

For $\A$ an ordered alphabet,
the \emph{Chinese relations} are the relations~\cite{DK}
\begin{equation}\label{eq-chinese}
\begin{split}
    cba = cab = bca, &\quad \text{for } a < b < c\in\A, \\
    aba = baa, \quad bba = bab, &\quad \text{for } a < b\in\A.
\end{split}
\end{equation}
The least congruence containing the relations \eqref{eq-chinese} is called
the \emph{Chinese congruence} and is denoted by $\equiv_{\chinese}$.
The quotient $\A^*/\equiv_{\chinese}$ is called the \emph{Chinese
monoid}~\cite{DK,CASSAIGNE2001} and, as one can see, it is
evaluation-preserving.

As in \cref{def:quotient-monoid}, for $\sigma: \A \to \NN_{\geq 2}$,
the \emph{$\sigma$-chinese congruence} is the least congruence on
$\A ^ *$ containing the Chinese congruence $\equiv_{\chinese}$ and
$a ^ {\sigma(a)} = a$ for every $a\in \A$; we denote this by
$\equiv_{\chinese(\sigma)}$. The $\sigma$-Chinese monoid on $\A$ is the
quotient $\A ^ * / {\equiv_{\chinese(\sigma)}}$ and is denoted by
$\chinese(\A, \sigma)$.

\paragraph{Chinese Staircases}\label{Chinese_Staircases}

In this section, we recall the details of the normal forms of $\chinese(\A)$
from \cite{CASSAIGNE2001} that are relevant here.

A \emph{Chinese staircase} $S$ over a finite ordered alphabet
$\A = \{a_1< \ldots< a_{n}\}$ is a triangular diagram of boxes with rows and
columns labelled by $\A$. The $i$-th row starting from the top is labelled by
$a_i$ and the $i$-th column starting from the right is labelled by $a_i$. The
row of $S$ labelled by $a_i$ has length $i$ and is drawn right-justified.
The entries of the boxes are non-negative integers, but for readability we
will omit the $0$ values,
see \cref{f:chistair} for an example.
\begin{figure}[ht]
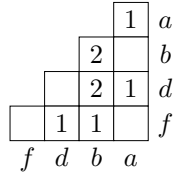

$$\gyoung(:::;1:a,::;2;:b,:;;2;1:d,;;1;1;:f,:f:d:b:a)$$
\caption{\label{f:chistair}The Chinese staircase over the alphabet
$\{a,b,d,f\}$ with row word $ab^2(db)^2(da)(fd)(fb)$. }
\end{figure}

For $S$ a Chinese staircase, we denote by $S(r,c)$ the entry with row
label $r$ and column label $c$.
An entry $S(r,c)$ is on the diagonal of $S$ if $r=c$.
An entry $S(r',c')$ \emph{covers} an entry $S(r,c)$ if
$r'\geq r$ and $c' \leq c$.
We say that $S(r',c')$ \emph{strictly covers} $S(r,c)$ if $r'>r$ and $c'<c$.
Equivalently, we say that $(r',c')$ is (strictly) at the south-east (SE) of
$(r,c)$, or that $(r',c')$ and $(r,c)$ are in a SE-NW position with each
other..

The \emph{row word of a Chinese staircase $S$}~\cite{CASSAIGNE2001}, denoted
by $\RW(S)$, is the word obtained by reading from top to bottom and left to
right the boxes of the staircase as follows.
We write $r^{S(r,r)}$ for boxes on the diagonal and
$(rc)^{S(r, c)}$ for the box at the intersection of row $r$ and column $c$,
see \cref{f:chistair} for an example.
Because of the labelling and the staircase shape, $r>c$.
Given a Chinese staircase $S$, the \emph{content} and the
\emph{evaluation} of $S$ respectively are the content and the evaluation of
$\RW(S)$.
Note that a letter appears in the reading of a Chinese staircase $S$ if and
only if there is a non-zero entry in its associated \emph{hook} (\emph{i.e.},
the reunion of its row and its column).
If all hooks are nonempty, we say that $S$ in \emph{reduced}.

As it is the case for the plactic monoid, there exists an algorithm inserting
a letter into a Chinese staircase~\cite[Algorithm 2.6]{CASSAIGNE2001} whose
iteration associates with any word $W\in\A^*$ a Chinese staircase $S(W)$ such
that $\RW(S(W))\equiv_\chinese W$~\cite[Theorem 2.14]{CASSAIGNE2001}.

Here is how it works. Consider a Chinese staircase $S$ indexed by an alphabet
$A\{a_1,\dots,a_n\}$ and split it into two parts: its bottom row $R$ and the
rest of $S$, denoted by $S'$.
Consider a letter $a_k$ to be inserted in $S$.
\begin{itemize}
\item If $k=n$ then add $1$ to the cell $(a_n,a_n)$.
\item If not, let $\ell$ be the largest value so that $S(a_n,a_\ell)$ is
nonzero. If the row is empty or if $\ell\leq k$, then do not change
$R$ and insert $a_k$ into $S'$.
Otherwise, subtract $1$ to the cell $R(a_\ell)$ and add $1$ to the cell
$R(a_k)$. Moreover, if $\ell<n$ then insert $a_\ell$ into $S'$. Following
the case of the plactic monoid, in that case, we say that $a_k$ bumps
$a_\ell$.
\end{itemize}

Now, to compute the Chinese staircase of a word $w$ on $A$, start
with the empty Chinese staircase indexed by $A$ and insert successively the
letters of $w$ from left to right in the staircase.

\begin{thm}[Proposition 3.12 of~\cite{CASSAIGNE2001}]
\label{thm-chinese}
Given a Chinese staircase $S$, the row word $\RW(S)$ is the shortlex normal
form of its Chinese class.
\end{thm}

As a consequence of \cref{thm-chinese} we may identify elements of
$\chinese(\A)$ and the corresponding staircases.

\begin{note}
\label{note-insertion}
Inserting a letter $a_k$ changes a unique entry 
on column $a_k$ and possibly other entries to the non-strict NW on that entry.
Indeed, $a_k$ bumps a letter that is strictly greater than itself so that it
will be inserted to its left and above it.
\end{note}

\subsection{The 2-Chinese Monoid}
\label{ss-2-chinese}

The goal of the whole Chinese section is to study the $\sigma$-Chinese
monoids in general. We will prove in \cref{ss-chinese-type} that these are
monoids of type $1$ and since such monoids $\M(\A,\sigma)$ derive all their
properties from $\M(\A,2)$, we first concentrate on $\chinese(\A,2)$.
However, unlike the 2-plactic monoid or stylic monoid~\cite{AR}, the
2-Chinese monoid does not appear in the literature.
Therefore, in this section, we provide a description of the 2-Chinese monoid.
More precisely, we present in \cref{subsection-2-chinese-staircase} and
\cref{subsection-2-chinese-functions} two normal forms of the 2-Chinese
monoid.
The former is the shortlex normal form of the classes ;
the latter is obtained by an insertion algorithm.
In \cref{ss-2-chinese-more}, we use these normal forms to show properties
of the 2-Chinese monoid.

\subsubsection{2-Chinese Staircases}
\label{subsection-2-chinese-staircase}

\begin{de}
\label{de-2-chinese-staircase}
A \emph{2-Chinese staircase} on the alphabet $\A$ is a Chinese staircase
such that the only entries are $0$ and $1$ and
\begin{enumerate}[(1)]
\item there is at most one non-zero entry in every row and every column;
\item every non-zero entry strictly to the north-west of a non-zero entry is
on the diagonal.
\end{enumerate}
\end{de}

For a 2-Chinese staircase $S$, the \emph{row word of $S$} is the row word of
$S$ considered as a Chinese staircase. We denote this word by $\RW(S)$.
See \cref{fig-chinese-examples} for some examples of 2-Chinese and not
2-Chinese staircases.

\begin{figure}[ht]
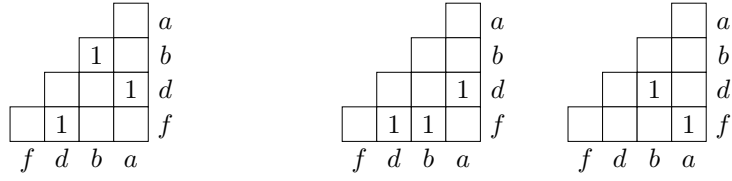

\centering
 \gyoung(:::;:a,::;1;:b,:;;;1:d,;;1;;:f,:f:d:b:a)\qquad\qquad\qquad
 \gyoung(:::;:a,::;;:b,:;;;1:d,;;1;1;:f,:f:d:b:a)\qquad
 \gyoung(:::;:a,::;;:b,:;;1;:d,;;;;1:f,:f:d:b:a)
\caption{The leftmost diagram is the diagram of a 2-Chinese staircase on the
alphabet $\{a,b,d,f\}$ whose row word is $bdafd$; the second is not a
2-Chinese staircase, because there are two non-zero values in the row labelled
$f$; the third is also not a 2-Chinese staircase because the $1$ in box $(d,
b)$ is to the north-west of $(f,a)$ and not on the diagonal
(\emph{i.e.}, $d\neq b$).}
\label{fig-chinese-examples}
\end{figure}

Thanks to the fact that the $2$-Chinese staircases have values $0$ or $1$ in
each of their cells, there are a finite number of such staircases with a given
underlying alphabet.

Here are the two $2$-Chinese staircases with content $\{a,b\}$ with their
associated row words:
\begin{equation}
\begin{matrix}
\label{chin2-2}
\Yboxdim{1cm}\begin{tikzpicture}[scale=0.5,baseline={(current bounding box.east)} ]
\tgyoung(0cm,0cm,%
:!\Ygr;1!\Yw:a,%
!\Ygr;1!\Yw;:b,%
:b:a)
\draw[ultra thick] (0, -1) -- (1, -1) -- (1, 0) -- (2, 0) -- (2, 1) -- (2, 1);
\end{tikzpicture}
&
\Yboxdim{1cm}\begin{tikzpicture}[scale=0.5,baseline={(current bounding box.east)} ]
\tgyoung(0cm,0cm,%
:!\Ygr;!\Yw:a,%
!\Ygr;;1!\Yw:b,%
:b:a)
\draw[ultra thick] (0, -1) -- (2, -1) -- (2, 0) -- (2, 0) -- (2, 1) -- (2, 1);
\end{tikzpicture}
\\[.8cm]
ab & ba \\
\end{matrix}
\end{equation}

And here are the five $2$-Chinese staircases with content $\{a,b,c\}$ with
their associated row words:
\begin{equation}
\begin{matrix}
\label{chin2-3}
\Yboxdim{1cm}\begin{tikzpicture}[scale=0.5,baseline={(current bounding box.east)} ]
\tgyoung(0cm,0cm,%
::!\Ygr;1!\Yw:a,%
:!\Ygr;1!\Yw;:b,%
!\Ygr;1!\Yw;;:c,%
:c:b:a)
\draw[ultra thick] (0, -2) -- (1, -2) -- (1, -1) -- (2, -1) -- (2, 0) -- (3, 0) -- (3, 1) -- (3, 1);
\end{tikzpicture}
&
\Yboxdim{1cm}\begin{tikzpicture}[scale=0.5,baseline={(current bounding box.east)} ]
\tgyoung(0cm,0cm,%
::!\Ygr;1!\Yw:a,%
:!\Ygr;!\Yw;:b,%
!\Ygr;;1!\Yw;:c,%
:c:b:a)
\draw[ultra thick] (0, -2) -- (2, -2) -- (2, -1) -- (2, -1) -- (2, 0) -- (3, 0) -- (3, 1) -- (3, 1);
\end{tikzpicture}
&
\Yboxdim{1cm}\begin{tikzpicture}[scale=0.5,baseline={(current bounding box.east)} ]
\tgyoung(0cm,0cm,%
::!\Ygr;!\Yw:a,%
:!\Ygr;;1!\Yw:b,%
!\Ygr;1!\Yw;;:c,%
:c:b:a)
\draw[ultra thick] (0, -2) -- (1, -2) -- (1, -1) -- (3, -1) -- (3, 0) -- (3, 0) -- (3, 1) -- (3, 1);
\end{tikzpicture}
&
\Yboxdim{1cm}\begin{tikzpicture}[scale=0.5,baseline={(current bounding box.east)} ]
\tgyoung(0cm,0cm,%
::!\Ygr;!\Yw:a,%
:!\Ygr;;1!\Yw:b,%
!\Ygr;;1!\Yw;:c,%
:c:b:a)
\draw[ultra thick] (0, -2) -- (2, -2) -- (2, -1) -- (3, -1) -- (3, 0) -- (3, 0) -- (3, 1) -- (3, 1);
\end{tikzpicture}
&
\Yboxdim{1cm}\begin{tikzpicture}[scale=0.5,baseline={(current bounding box.east)} ]
\tgyoung(0cm,0cm,%
::!\Ygr;!\Yw:a,%
:!\Ygr;1;!\Yw:b,%
!\Ygr;;;1!\Yw:c,%
:c:b:a)
\draw[ultra thick] (0, -2) -- (3, -2) -- (3, -1) -- (3, -1) -- (3, 0) -- (3, 0) -- (3, 1) -- (3, 1);
\end{tikzpicture}
\\[1.2cm]
abc & acb & bac & bacb & bca \\
\end{matrix}
\end{equation}

In general, the number of 2-Chinese staircases with content
$\{a_1,\dots,a_n\}$ is the $n$-th Catalan number
$C_n := \frac1{n+1}\binom{2n}{n}$ (Sequence A000108 of~\cite{oeis}). We shall
prove this property by providing a bijection between these and the Dyck paths
of (semi-)length $n$, a well-known combinatorial object enumerated by Catalan
numbers.

Given a non-negative integer $n$, recall that a
\emph{Dyck path of (semi-)length $n$} is a path on an $n\times n$ grid going
from $(0,0)$ to $(n,n)$ having only east and north steps, and staying under
the diagonal; see~\cref{fig-chinese-to-path-example} for some examples.

\begin{algo}
\label{algo-chin2dyck}
Given a Chinese staircase $S$ on $\A$ with content $\A$, that is, a reduced
Chinese staircase, we associate a Dyck path $\dyck(S)$ as follows:
color in gray the boxes that are covered by a box containing a non-zero entry.
The associated path is the south-east border of the gray boxes.

If a Chinese staircase $S$ on $\A$ is not reduced, the same algorithm
applies after removal of its empty hooks.
\end{algo}

The Dyck paths associated with Chinese staircases have been represented on
these staircases as bold lines. See \cref{fig-chinese-to-path-example} and
also Equations~\eqref{chin2-2} and~\eqref{chin2-3}.

Note that since there is at least one non-zero entry in each hook in a
Chinese staircase of content $\A$, the diagonal is always fully gray. So the
path associated with a Chinese staircase $S$ is a Dyck path. See again
\cref{fig-chinese-to-path-example} for an example.

\begin{de}
A box $(r,c)\in S$ is a \emph{peak of $\dyck(S)$} if the only gray
box that covers $(r,c)$ is itself.
\end{de}

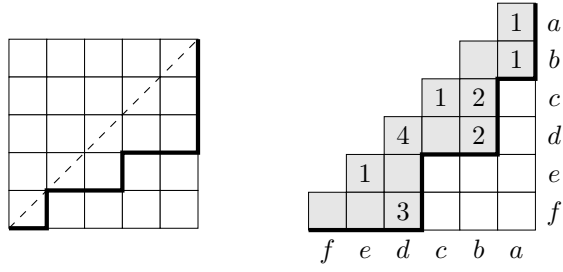
\begin{figure}[ht]
 \[
 \Yboxdim{1cm}\begin{tikzpicture}[scale=0.5,baseline={(current bounding box.east)}]
   \tgyoung(0cm,0cm,%
   ;;;;;,%
   ;;;;;,%
   ;;;;;,%
   ;;;;;,%
   ;;;;;)
   \draw[ultra thick](0,-4)--(1,-4)--(1,-3)--(3,-3)--(3,-2)--(5,-2)--(5,1);
   \draw[dashed](0,-4)--(5,1);
 \end{tikzpicture}
 \qquad\qquad
 \Yboxdim{1cm}\begin{tikzpicture}[scale=0.5,baseline={(current bounding box.east)}]
 \tgyoung(0cm,0cm,%
 :::::!\Ygr1!\Yw:a,%
 ::::!\Ygr;;1!\Yw:b,%
 :::!\Ygr;12!\Yw;:c,%
 ::!\Ygr4;;2!\Yw;:d,%
 :!\Ygr;1;!\Yw;;;:e,%
 !\Ygr;;;3!\Yw;;;:f,%
 :f:e:d:c:b:a)
 \draw[ultra thick] (0, -5) -- (3, -5) -- (3, -4) -- (3, -4) -- (3, -3) -- (5, -3) -- (5, -2) -- (5, -2) -- (5, -1) -- (6, -1) -- (6, 0) -- (6, 0) -- (6, 1) -- (6, 1);
 \end{tikzpicture}
 \]
\caption{Examples of a Dyck path and a Dyck path associated with a Chinese
staircase $S$. The peaks of $S$ is the set $\{(f,d);(d,b);(b,a)\}$.}
\label{fig-chinese-to-path-example}
\end{figure}

\begin{algo}
\label{algo-dyck2chin}
Let $p$ be a Dyck path of semi-length $|\A|$.
The \emph{corner-diagonal filling} $\cdf_\A(p)$ of $p$ is the Chinese
staircase over $\A$ obtained by
\begin{enumerate}[(1)]
\item putting a $1$ in every peak cell,
\item putting a $1$ on any diagonal cell whose associated hook is still empty.
\end{enumerate}
\end{algo}

\begin{lemma}
Let $p$ be a Dyck path of semi-length $|\A|$.
Then $\cdf_\A(p)$ is a 2-Chinese staircase of content $\A$.
\end{lemma}

\begin{proof}
All corners of the Dyck path are in different rows and columns, so the first
condition of \cref{de-2-chinese-staircase} holds. Moreover, all these corners
are to the north-east of each other so that the second condition of
\cref{de-2-chinese-staircase} also holds.
Finally, the corner-diagonal filling ensures that each hook has a nonzero
entry, so $\cont(\cdf_\A(p)) = \A$.
\end{proof}

\begin{prop}
\label{p:unicity2stairs}
The maps $\dyck$ and $\cdf_\A$ are mutually inverse bijections between the
\mbox{2-Chinese} staircases of content $\A$ and the Dyck paths of semi-length $|\A|$.
\end{prop}

\begin{proof}
First, both maps have images inside the expected set.

Now, starting with a Dyck path $p$ of semi-length $|\A|$, the 2-Chinese
staircase $\cdf_\A(p)$ has a~$1$ in each corner of the path, ensuring that
$\dyck(\cdf_\A(p)) = p$.
Conversely, given a 2-Chinese staircase $S$, the $1$s that are not on the
diagonal are in the corners of the path $\dyck(S)$.
These $1$s are exactly those not on the diagonal in $\cdf_\A(\dyck(S))$ which
therefore must be equal to $S$.
\end{proof}

\begin{cor}
\label{cor-counting-2chin}
The number of 2-Chinese staircases of content $\A$ of size $n$ is the $n$-th
Catalan number.
\end{cor}

Let us now state our main result on the 2-Chinese monoid:

\begin{thm}
\label{thm-2-chinese}
Let $W$ and $W'$ be two words.
Then $W \equiv_{\chinese(2)} W'$ if and only if \mbox{$\cont(W) = \cont(W')$}
and $\dyck(S(W)) = \dyck(S(W'))$.

Moreover, the shortlex 2-Chinese normal form of $W$ is the row word of the
2-Chinese staircase of the same content and Dyck path as $S(W)$.
\end{thm}

In order to prove this result, we first need two simple lemmas on the
(classical) chinese insertion algorithm.

\begin{lemma}
\label{lem-chin-aa2}
Let $W$ be a word and $a$ a letter. Then
$\dyck(S(W a))=\dyck(S(W aa))$.
\end{lemma}

\begin{proof}
Consider the Chinese staircase $S$ of $W$ on an alphabet containing both
its content and $a$.

When one inserts the letter $a$ in $S$, it inserts in its column
and then only changes values to the north-west of this insertion position 
(see Note~\ref{note-insertion}). When inserting a second $a$ immediately
after, this second $a$ necessarily inserts to the north of the first one and
then bumps letters to its own north-west, hence not modifying the path.
Removing the empty hooks now proves the result.
\end{proof}

\begin{lemma}
\label{lem-chin-aa}
Let $W$ and $W'$ be two words with same content such that

\mbox{$\dyck(S(W))=\dyck(S(W'))$} and let $a$ be a letter. Then
$\dyck(S(W a))=\dyck(S(W' a))$.
\end{lemma}

\begin{proof}
Consider the staircases of both $W$ and $W'$ on an alphabet containing both
their content and $a$. Having same content and same Dyck paths implies that
the paths below the nonzero values of $S(W)$ and $S(W')$ are the same.
Now, $a$ inserts in its column either on the same nongray cell of both $S(W)$
and $S(W')$ or it inserts in gray cells in both $S(W)$ and $S(W')$.
In the first case, it modifies both the paths in the same way whereas in the
second case, it modifies neither paths.
Removing the empty hooks now proves the result.
\end{proof}

\begin{proof}[Proof of the Theorem]
Let us first prove that if $W \equiv_{\chinese(2)} W'$ then they have
same contents and same Dyck paths.

The content property is direct thanks to the definition of the congruence.
Now, we shall prove the property for an elementary rewriting of the 2-Chinese
congruence. If $W$ and $W'$ differ by a classical Chinese rewriting, the
property is direct since they have the same Chinese staircase.
If they differ by doubling a letter $a$, that is
$W=W_1 a W_2$ and $W'=W_1 a^2 W_2$, then $\dyck(W_1\cdot a)=\dyck(W_1\cdot a^2)$
(see Lemma~\ref{lem-chin-aa2}) and then the Dyck paths evolve in the
same way when inserting all letters of $W_2$ (see Lemma~\ref{lem-chin-aa}).

\medskip
So the 2-Chinese classes are subsets of words having same content and same
Dyck paths. We shall now show that the words having same content and same
Dyck paths are all equivalent. To do that, we will show that all but one
are not shortlex normal forms by showing a sequence of 2-Chinese rewritings
either making these shorter or lexicographically smaller.

Since the shortlex element of a given Chinese staircase is its row word, we
can restrict ourselves to row words.
Now, if the sum of the values in row $c$ of $S$ is more than $1$, use either
$c^2\equiv_{\chinese(2)} c$ if there is a value at least $2$ in the diagonal,
or $cca\equiv_{\chinese(2)} ca$ if there is a $1$ in the diagonal and another
nonzero value elsewhere on the row, or
$cacb \equiv_{\chinese(2)} ccab \equiv_{\chinese(2)} cab$ with $a\leq b$
if there is either a value at least $2$ in a cell or two nondiagonal cells
with at least a $1$.
So these Chinese row words are not the shortlex normal form of their
class.
The same holds for columns because of the \emph{column reading} of a Chinese
staircase described in~\cite{CASSAIGNE2001} on which we apply the same
rewritings as before, reducing the length of the word. So the row word
associated with this reduced column word is strictly shorter than the initial
one.

Now, if two nonempty nondiagonal cells in $S$ are in SE-NW position,
consider such two cells $(d,a)$ and $(c,b)$ with $a<b<c<d$ and no such cells
in between in the order of the row reading of $S$.
There might still be nonempty rows between these two but these nonempty rows
necessarily have their one on their diagonal or it would contradict the
previous assumption.

Now consider the factor of $\RW(S)$ beginning with $cb$ and ending with $da$.
If there are no values in between, then one can rewrite this part as $bcad$
using only 2-Chinese elementary rewritings:
\begin{equation}
cbda \equiv_{\chinese(2)} cbdaa
     \equiv_{\chinese} cbada
     \equiv_{\chinese} bcada
     \equiv_{\chinese} bcdaa
     \equiv_{\chinese(2)} bcda.
\end{equation}

If there are values in between, these are diagonal elements, hence with values
between $a$ and $d$. But thanks to the Chinese congruence, all these values
commute with the block $da$: it is the rewriting $b'da\equiv dab'$ if
$a<b'<d$, so that
\begin{equation}
cb\, {b'_1}\dots {b'_p}\, da
   \equiv_{\chinese}  cbda\,  {b'_1}\dots {b'_p}
   \equiv_{\chinese(2)}  bcda\,  {b'_1}\dots {b'_p}
   \equiv_{\chinese}  bc\, {b'_1}\dots {b'_p}\, da
\end{equation}
the middle rewriting being done as before.

So in both cases, we found a word that is equivalent to $\RW(S)$ but
lexicographically smaller. Putting all the previous observations together, we
see that the set of shortlex normal forms are row words of the 2-Chinese
staircases (see Definition~\ref{de-2-chinese-staircase}). And since we know
that there is only one such staircase with given contents and Dyck paths (see
Proposition~\ref{p:unicity2stairs}), all words with a given content and a
given Dyck path are 2-Chinese equivalent and that their shortlex normal form
is the row reading of their corresponding 2-Chinese staircase.
\end{proof}

\begin{cor}
\label{cor-main-chinese}
Let $\A$ be a finite alphabet.
Then the set of shortlex normal forms of $\chinese(\A,2)$ coincide with the
set of row words of the 2-Chinese staircases of content $\B\subseteq\A$.
\end{cor}

\subsubsection{2-Chinese Functions}
\label{subsection-2-chinese-functions}

In this section, we introduce an alternative to 2-Chinese staircases and
use these objects to provide an equivalent of the Chinese insertion algorithm
in the context of the 2-Chinese monoid. Given ~\cref{thm-2-chinese}, one could
directly adapt the classical Chinese insertion algorithm to our case but our
version in one-line notation makes more sense given the combinatorial object
it encodes.

\paragraph{2-Chinese functions: definition and first property\\}

\begin{de}
\label{2-chinese-func}
A \emph{2-Chinese function} on an alphabet $\A$ is a function $\f:\B\to \B$ 
for some $\B\subseteq\A$ that satisfies for all $x, y\in \B$ 
\begin{enumerate}[(i)]
\item $x\leq y$ implies $\f(x)\leq \f(y)$;
\item $\f(x)\leq x$.
\end{enumerate}
\end{de}

We shall represent the the $2$-Chinese functions as a biword where the top
word is the domain of the function and the bottom word the images of each
element of the domain, see the left of~\cref{fig-2ch-red-dyck} for an example.

If $\B$ is fixed, the number of such functions is the Catalan number
$C_{|\B|}$: indeed, if $\B=\{1,\dots,n\}$, their images are known as
non-decreasing parking function, one of the many combinatorial objects
enumerated by the Catalan numbers (see~\cite{oeis}).
And indeed, there is a very natural bijection that makes sense in our
context between 2-Chinese functions on a given alphabet and Dyck paths of
semi-length $\cont(\f)$.

\begin{prop}
\label{prop-bij-2fns-2stair}
Consider a 2-Chinese function $\f$ with domain $\B$. Represent it in the
Chinese staircase indexed by $\B$ with a $1$ (a $2$ if it is on the diagonal)
at all coordinates $(x,\f(x))$ for $x$ in~$\B$. This staircase is reduced
since there is exactly one nonzero value in each row. Thus it has an
associated Dyck path.

The second map goes as follows: consider a Dyck path of semi-length
$\ell$ and an alphabet $\B$ of size~$\ell$.
Draw the path in the Chinese staircase indexed by $\B$.
Then for any $x$ in $\B$, define $\f(x)$ as the column label of the box
immediately to the left of the Dyck path on row $x$.

These maps are inverses of each other.
\end{prop}

\begin{proof}
First, both maps have the expected image: the first one returns a Dyck path
since the staircase is reduced and the second one returns a 2-Chinese
function. Indeed, the path being above the diagonal implies that $\f(x)\leq x$
and being a path with only east and north steps implies that $\f(x)\leq \f(y)$
if $x<y$.

Now, consider a 2-Chinese function $\f$ and its associated staircase.
It has a nonzero value on each row and there are no two values in a SE-NW
position. So each nonzero value has a non gray cell to its right (or no cell
at all if it is on the rightmost column). So the nonzero cells are exactly the
ones that are immediately to the left of the corresponding Dyck path.
\end{proof}

For example, \cref{fig-2ch-red-dyck} shows a 2-Chinese function, its
associated reduced staircase and Dyck path.

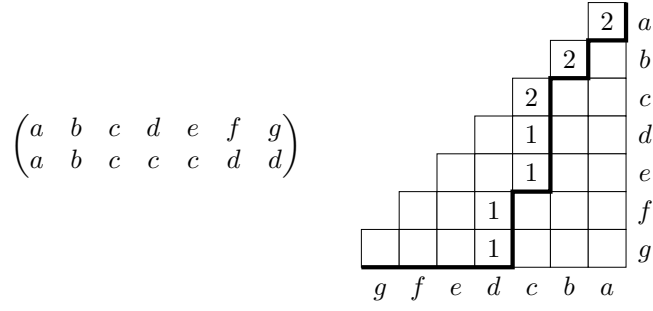
\begin{figure}[ht]
\begin{equation*}
\begin{pmatrix}
a&b&c&d&e&f&g\\
a&b&c&c&c&d&d\\
\end{pmatrix}
\qquad 
\Yboxdim{1cm}\begin{tikzpicture}[scale=0.5,baseline={(current bounding box.east)}]
 \tgyoung(0cm,0cm,%
 ::::::;2!:a,%
 :::::;2!;:b,%
 ::::;2!;;:c,%
 :::;;1!;;:d,%
 ::;;;1!;;:e,%
 :;;;1!;;;:f,%
 ;;;;1!;;;:g,%
 :g:f:e:d:c:b:a)
 \draw[ultra thick] (0, -6) -- (4, -6) -- (4, -4) -- (5, -4) -- (5, -1) -- (6, -1) -- (6, 0) -- (7, 0) -- (7, 1);
\end{tikzpicture}
\end{equation*}
\caption{A $2$-Chinese function, its staircase and its Dyck path}
\label{fig-2ch-red-dyck}
\end{figure}

\begin{de}
\label{def-stair2func}
Let $S$ be a Chinese staircase. We shall denote by $\gd(S)$ the 2-Chinese
function of same content as $S$ and associated with the Dyck path of $S$.
\end{de}

With a 2-Chinese function, we associate a \emph{reading word} the following
way.

\begin{de}
\label{2-chinese-func-word}
Let $\f$ be a 2-Chinese function. Its reading word is the reading word of its
corresponding staircase.
\end{de}

\begin{exam}
Here are the images by $\gd$ of the 2-Chinese staircases of
Equation~\eqref{chin2-3} with their reading words.
\begin{equation*}
\begin{matrix}
\begin{pmatrix} a&b&c\\ a&b&c\\ \end{pmatrix}
&
\begin{pmatrix} a&b&c\\ a&b&b \end{pmatrix}
&
\begin{pmatrix} a&b&c\\ a&a&c\\ \end{pmatrix}
&
\begin{pmatrix} a&b&c\\ a&a&b\\ \end{pmatrix}
&
\begin{pmatrix} a&b&c\\ a&a&a\\ \end{pmatrix}
\\[.3cm]
aabbcc & aabbcb & aabacc & aabacb & aabaca \\
\end{matrix}
\end{equation*}
\end{exam}

\begin{prop}\label{2chifunct}
Let $S$ be a 2-Chinese staircase and conside $W$ its row reading and $\f$
its corresponding 2-Chinese function through the bijection with Dyck paths.
Then $W$ and the reading word of $\f$ are $2$-Chinese equivalent.

In particular, the set of all reading words of the 2-Chinese functions on
$\A$, is a set of normal forms for $\chinese(\A, 2)$.

We shall denote by $\rd_W$ the $2$-Chinese function associated with the
$2$-Chinese class of $W$.
\end{prop}

\begin{proof}
First, both words in the statement are $2$-Chinese equivalent by
\cref{thm-2-chinese} since their staircases have the same underlying alphabet
and correspond to the same Dyck path. Since the \mbox{$2$-Chinese} staircases are a
set of normal forms for $\chinese(\A, 2)$, so are the reading words of
the $2$-Chinese functions.
\end{proof}

\paragraph{An insertion algorithm on 2-Chinese functions\\}

Let us now provide an insertion algorithm on 2-Chinese functions that
generalizes the usual insertion of words into Chinese staircases. The main
result of this section is that two words are $2$-Chinese equivalent iff they
give the same result by our insertion algorithm.
Let us denote by $\emptyf$ the unique function whose domain is empty.

\begin{de}
\label{2-chinese-func-insert}
The \emph{right insertion} of $y\in\A$ in a 2-Chinese function $\f$ is the
function $\f\leftarrow y$ whose domain is $\dom(\f)\cup\{y\}$ such that for
all $x\in\dom(\f\leftarrow y)$
\begin{align}
(\f\leftarrow y)(x) := & 
\min(\{y\} \cup \f(x^\uparrow)),
\end{align}
where $x^{\uparrow}$ is the set of elements of $\A$ greater than or equal to
$x$ and $\f(x^{\uparrow})$ the image by $\f$ of the subset of these belonging
to the domain of $\f$.

Given a 2-Chinese function $\f$ and $W=w_1\dots w_n\in \A ^*$, we denote by
$\f\leftarrow W$ the result of inserting $w_1$, then $w_2$, and so on
into $\f$.
\end{de}

A left insertion will be defined in \cref{ss-2-chinese-more}.

Let us highlight two subcases of particular interest:
\begin{itemize}
\item if $y>\max(\dom(\f))$ then
\begin{equation}
(\f\leftarrow y)(x)=\begin{cases}
\f(x),&\text{if $x\neq y$},\\
y,&\text{if $x=y$},\\
\end{cases}
\end{equation}
so that, for example,
\begin{equation}
\begin{pmatrix}
a&b&c&e&f\\
a&b&c&c&f\\
\end{pmatrix}
\leftarrow g =
\begin{pmatrix}
a&b&c&e&f&g\\
a&b&c&c&f&g
\end{pmatrix}.
\end{equation}
\end{itemize}

\begin{itemize}
\item if $x\in\dom(\f)$, then $\min(\f(x^\uparrow))=\f(x)$ so that
\begin{equation}
x\in\dom(\f) \implies
(\f\leftarrow y)(x)=\min(y, \f(x))\,,
\end{equation}
so that, for example,
\begin{equation}
\begin{pmatrix}
a&b&c&e&f\\
a&b&c&c&f\\
\end{pmatrix}
\leftarrow d =
\begin{pmatrix}
a&b&c&d&e&f\\
a&b&c&c&c&d\\
\end{pmatrix}.
\end{equation}
\end{itemize}

\begin{prop}
Let $\f$ be a 2-Chinese function and let $y\in\A$. Then $\f\leftarrow y$
is a 2-Chinese function.
\end{prop}

\begin{proof}
Let $a<b\in\dom(\f)\cup\{y\}$.
Then $\f(a^{\uparrow}) \supseteq \f(b^{\uparrow})$ so that
 $\{y\}\cup\f(a^{\uparrow}) \supseteq \{y\}\cup\f(b^{\uparrow})$.

Taking the $\min$ on both sides gives
\[
(\f\leftarrow y)(a) = \min(\{y\}\cup\f(a^{\uparrow})) \leq
\min(\{y\}\cup\f(b^{\uparrow})) = (\f\leftarrow y)(b).
\]

It remains to show that $(\f\leftarrow y)(a) \leq a$ for all $a\in
\dom(\f\leftarrow y)$.
If $a\in\dom(\f)$ then we know that $\min(\f(a^\uparrow))=\f(a)$ so that
$(\f\leftarrow y)(a) = \min(y, \f(a)) \leq a$ since $\f(a) \leq a$.
Now, if $a\notin\dom(\f)$, then $a=y$ and $(\f\leftarrow y)(y)
\leq y$ by definition.
\end{proof}

\begin{prop}
Let $S$ be a 2-Chinese staircase and let $W$ be its row reading and $\f$
its corresponding 2-Chinese function through the bijection with Dyck paths.

We have $\emptyf \leftarrow W = \f$.
\end{prop}

\begin{proof}
Insert $W=w_1\dots w_n$ as a $2$-Chinese function. Thanks to the
definition of the $2$-Chinese staircases, the word $W'=w_1\dots w_{n-1}$ is
also the row reading of a $2$-Chinese staircase, so we can assume by induction
on $n$ that $\emptyf\leftarrow W'=\f'$ where $\f'$ is the function having the
same Dyck path and content as $W'$.

Now, if $w_n$ is the maximal value of $W$, that is, if the last row of $S$ has
its nonzero value on the diagonal, then the property holds since we are in the
first hightlighted subcase $y>\max(\dom(\f))$ of the insertion algorithm.

Otherwise, we insert $w_n$ that is smaller than $w_{n-1}$. This insertion
modifies the Dyck path from $S'$ to $S$ is a very simple way: all north steps
that were to the left of the column indexed by $w_n$ become north steps
immediately after the column indexed by $w_n$ and the east step in between
move to the beginning of the Dyck path.

The same happens with the function $\f$: 
it modifies the values of $\f(x)=y$ to $\f(x)=w_n$ if $y>x$, hence changing
the values that were greater than $w_n$ uniformly to $w_n$ so its
corresponding Dyck path changes as the one of $S$, hence the result.
\end{proof}

For example, \cref{fig-2ch-insere} shows a 2-Chinese staircase and its
corresponding 2-Chinese function and the same objects after insertion of the
letter $d$.

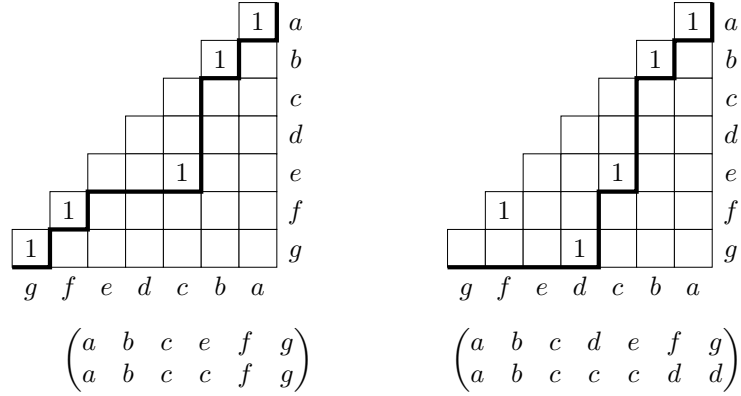
\begin{figure}[ht]
\begin{equation*}
\begin{matrix}
\Yboxdim{1cm}\begin{tikzpicture}[scale=0.5,baseline={(current bounding box.east)}]
 \tgyoung(0cm,0cm,%
 ::::::;1!:a,%
 :::::;1!;:b,%
 ::::;;;:c,%
 :::;;;;:d,%
 ::;;;1!;;:e,%
 :;1!;;;;;:f,%
 ;1!;;;;;;:g,%
 :g:f:e:d:c:b:a)
 \draw[ultra thick] (0,-6) -- (1,-6) -- (1,-5) -- (2,-5) -- (2,-4) -- (4, -4) -- (5, -4) -- (5, -1) -- (6, -1) -- (6, 0) -- (7, 0) -- (7, 1);
\end{tikzpicture}
\qquad
&
\qquad
\Yboxdim{1cm}\begin{tikzpicture}[scale=0.5,baseline={(current bounding box.east)}]
 \tgyoung(0cm,0cm,%
 ::::::;1!:a,%
 :::::;1!;:b,%
 ::::;;;:c,%
 :::;;;;:d,%
 ::;;;1!;;:e,%
 :;1!;;;;;:f,%
 ;;;;1!;;;:g,%
 :g:f:e:d:c:b:a)
 \draw[ultra thick] (0, -6) -- (4, -6) -- (4, -4) -- (5, -4) -- (5, -1) -- (6, -1) -- (6, 0) -- (7, 0) -- (7, 1);
\end{tikzpicture}
\\[2.2cm]
\begin{pmatrix} a&b&c&e&f&g\\ a&b&c&c&f&g\\ \end{pmatrix} &
\qquad
\begin{pmatrix} a&b&c&d&e&f&g\\ a&b&c&c&c&d&d\\ \end{pmatrix} \\
\end{matrix}
\end{equation*}
\caption{The $2$-Chinese staircases associated with $abecfg$ and $abecfgd$
(insertion of $d$) and their $2$-Chinese functions.}
\label{fig-2ch-insere}
\end{figure}

\begin{cor}
\label{chinese2-rd}
For $W\in\A^*$, $\rd_W=\emptyf\leftarrow W$.
\end{cor}

Let us now state the main result of this section.
  
\begin{thm}
For $W_1, W_2\in \A ^ *$, we have $W_1\equiv_{\chinese(2)} W_2$ if and only if
$(\emptyf \leftarrow W_1) = (\emptyf \leftarrow W_2)$.
\end{thm}

\begin{proof}
Let us first show that if $W_1$ and $W_2$ are $2$-Chinese equivalent, they
have the same result by the insertion algorithm. We can restrict ourselves to
the case where $W_1$ and $W_2$ differ by a single $2$-Chinese rewriting.

By definition of the insertion algorithm, if $W_2$ differs by a contraction
(or a duplication) of a letter, $\rd_{W_1}=\rd_{W_2}$.

Now let $W_1=UbcaV$, $W_2=UcbaV$ and $W_3=UcabV$.
For any $x\in\cont(W_1)\setminus \{b\}$, \mbox{$x_{W_1}=x_{W_2}=x_{W_3}$}.

If $b\in\cont(U)$, $b_{W_1}=b_{W_2}=b_{W_3}$ so that
$\rd_{W_1}=\rd_{W_2}=\rd_{W_3}$.
If $b\notin\cont(U)$, we have
\mbox{$c_{W_1}\leq \min(\cont(aV))\leq b_{W_1}=b_{W_2}\leq b_{W_3}$}.
Because of the definition of $\rd$, we have \mbox{$\rd_{W_1}=\rd_{W_2}=\rd_{W_3}$}.

We now know that being $2$-Chinese equivalent implies having same $2$-Chinese
functions by the insertion algorithm. But since the row readings of the
$2$-Chinese staircases have their corresponding $2$-Chinese functions by the
insertion algorithm and since there are as many $2$-Chinese functions as
$2$-Chinese equivalence classes, we get the result.
\end{proof}

\subsubsection{Properties of the 2-Chinese monoid}
\label{ss-2-chinese-more}

Let us conclude the study of the $2$-Chinese monoid with some properties that
are direct corollaries of the previous results.

\paragraph{A direct way to compute the insertion algorithm\\}
\label{p-2-chinese}

For $W\in A^*$ a word on the finite ordered alphabet $A$ and $x\in \cont(W)$, we let
\begin{equation}
x_W=\min\set{\cont(xV)}{W=UxV,\; x\notin\cont(U)}.
\end{equation}

\begin{defi}
\label{def-rd}
For $W\in\A^*$, define $\rdb_W$ as the function such that
$\dom(\rdb_W)=\cont(W)$ and for all $x\in\cont(W)$,
$\rdb_W(x)=\min\set{y_W}{y\in x^{\uparrow}\cap\cont(W)}$.
\end{defi}

\begin{lemma}
For $W\in\A^*$, $\rdb_W=\rd_W$.
\end{lemma}

\begin{proof}
First, it is immediate to check that $\rdb_W$ is indeed a $2$-Chinese function.
And it is as immediate to check that it is invariant under the $2$-Chinese
rewritings and that $\rdb$ applied to the reading word of a $2$-Chinese
function $\f$ gives back $\f$.
\end{proof}

\paragraph{Enumeration\\}
\label{p-chinese-enum}

Recall that the \emph{Catalan triangle}
(\href{https://oeis.org/A009766}{A009766} of~\cite{oeis}) is the triangle of
integers $T_{n,k}$ (sometimes called ballot numbers) that count the number of
Dyck paths of semi-length $n$ starting with exactly $k$ east steps.

These results are corollaries of \cref{cor-main-chinese}.

\begin{cor}
Let $\A=\{a_1<\cdots<a_n\}$.
The number of 2-Chinese shortlex normal forms with content $\A$ ending by
$a_k$ is the coefficient $T_{n, k}$ of the Catalan triangle.
\end{cor}

\begin{cor}
\label{cor:enum2chin}
Let $\A$ be a finite alphabet of cardinality $n$.
Then, $|\chinese(\A,2)|=\sum_{k=0}^n\binom{n}{k} C_k$ where $C_k$ is the
$k$-th Catalan number.
\end{cor}

\paragraph{Idempotents\\}

For $\B=\{b_1< \ldots< b_k\}\subseteq \A$
define $e_{2,\B} := b_2b_3\cdots b_{k-1}b_{k}b_{1}$.

\begin{prop}
The idempotents of $\chinese(\A,2)$ are exactly the $e_{2,\B}$ for every
$\B\subset\A$.
As a consequence, the number of idempotents in $\chinese(\A, 2)$ is
$2^{|\A|}$.
\end{prop}

\begin{proof}
Consider a 2-Chinese function $\f$ representing an idempotent of
$\chinese(\A,2)$.
This means that $\RW(\f)\RW(\f)\equiv_{\chinese(2)}\RW(\f)$, which
equivalently means $\f\leftarrow \RW(\f)=\f$.
Let $b_1=\min\dom(\f)$.
Since $\RW(\f)$ starts with $b_1$, for all $x\in\dom(\f)$, $(\f\leftarrow
\RW(\f))(x)=\f(x)=b_1$.

Conversely, since $e_{2,\B}$ is the reading word of the 2-Chinese staircase
associated with $\f$, it ends the proof.
\end{proof}

\paragraph{Order reversing involution\\}
For a (finite) ordered alphabet $\A$, recall that $\Theta$ is the unique
anti-automorphism on $\A^*$ that extends the order-reversing permutation of
$\A$.

As pointed out in \cite[Section 1.2]{CASSAIGNE2001}, by simply looking at the
relations \cref{eq-chinese}, one can remark the following results.

\begin{lemma}
\label{lem:thetaChinese}
If $W\equiv_{\chinese}U$ then $\Theta(W)\equiv_{\chinese}\Theta(U)$.
\end{lemma}

\begin{cor}
\label{cor:theta2Chinese}
If $W\equiv_{\chinese(2)}U$ then $\Theta(W)\equiv_{\chinese(2)}\Theta(U)$.
\end{cor}

Therefore, $\Theta$ induce an antiautomorphism on the Chinese monoid.
In \cite{CASSAIGNE2001}, the authors describe the image by $\Theta$ of a
staircase, $\Theta(S)$, as the symmetry of $S$ by the anti-diagonal
\cite[Proposition 2.20]{CASSAIGNE2001}.
Using this and recalling that, by \cref{thm-2-chinese}, a $2$-Chinese staircase
$\f$ is completely determined by $\cont(\f)$ and $\dyck(\f)$, one can deduce the
following result.

\begin{thm}
Let $\A$ a finite ordered alphabet and $W\in \A^*$.
Then $\cont(\Theta(W))=\Theta(\cont(W))$ and $\dyck(\Theta(W))$ is the mirror
of $\dyck(W)$ along the anti-diagonal.
\end{thm}

\paragraph{Left Insertion\\}

Let us define the left 2-Chinese-insertion of a letter $x$ using the right
insertion and $\Theta$.

\begin{prop}
For a letter $x\in\A$ and a word $W\in\A^*$,
$x\rightarrow \rd(W)=\Theta(\Theta(\rd(W))\leftarrow\Theta(x))$.
\end{prop}

Let us give a more explicit definition of the left 2-Chinese insertion of $x$
into a 2-Chinese function~$\f$.
Denote by $x\rightarrow \f(W)$, the 2-Chinese function such that
\begin{enumerate}[(i)]
\item $\dom(x\rightarrow \f)=\dom(\f)\cup\{x\}$,
\item $\forall y\in \dom(x\rightarrow \f)$,
      $(x\rightarrow \f)(y)=
          \left\lbrace\begin{array}{l}\min(\dom(x\rightarrow \f)),\text{ if }y\leq x,\\ 
          \f(y),\text{ if }y>x.\end{array}\right.$
\end{enumerate}
\newpage

\begin{prop}
For $x\in\A$ and $W\in\A^*$, $x\rightarrow \rd_W=\rd_{xW}$.
\end{prop}

\begin{proof}
Recall that $W\equiv_{\chinese(\A,2)}\RW(\rd_W)$.
Letting $\fp=\rd_x$, we have
\begin{equation}
\rd_{xW}=\fp\leftarrow W=\fp\leftarrow\RW(\rd_W).
\end{equation}
Since $\RW(\rd_W)$ starts with $b_1=\min(W)$, we have
$\rd_{xW}(x)=\min\{\dom(W)\cup\{x\}\}$.
Therefore, for every $y\leq x$, $\rd_{xW}(y)=\min\{\dom(W)\cup\{x\}\}$.
And clearly, for $y>x$, $(\fp\leftarrow W)(y)=\rd_W(y)$, which concludes the proof.
\end{proof}

\paragraph{$\J$-triviality and $\J$-order\\}
\label{p-chinese-Jtriv}

Given a 2-Chinese function $\f$, let
$\operatorname{area}(\f)=\operatorname{area}(\dyck(\f))$, which is the number
of boxes between the diagonal and $\dyck(\f)$.
Define the map $\operatorname{sz}:\chinese(\A,2)\to\NN$ as
$\operatorname{sz}(\f)=|\dom(\f)|+\operatorname{area}(\f)$.

\begin{lemma}\label{sz}
For a 2-Chinese function $\f$ and a letter $x\in A$, one has
\begin{align*}
\operatorname{sz}(x\rightarrow \f)\geq \operatorname{sz}(\f),\\
\operatorname{sz}(\f\leftarrow x)\geq \operatorname{sz}(\f).
\end{align*}

Moreover, one has the equality if and only if $x\rightarrow \f= \f$ (resp.
$\f=\f\leftarrow x$).
\end{lemma}

\begin{proof}
By definition of the left insertion, one has $\dom(\f)\subseteq
\dom(x\rightarrow \f)$ and $(x\rightarrow \f)(y)\leq \f(y)$ for all $y\in
\dom(\f)$. Thus, $sz(x\rightarrow \f)\geq sz(\f)$.

Moreover, in order to have the equality, we need $x\in\dom(\f)$ and
$(x\rightarrow \f)(y)=\f(y)$ for all $y\in\dom(\f)$.

The right insertion is proved in a similar manner.
\end{proof}

\begin{thm}
The 2-Chinese monoid is $\J$-trivial.
\end{thm}

\begin{proof}
Let $X,Y\in A^*$ such that $X\in \chinese(\A, 2)\cdot Y\cdot \chinese(\A, 2)$
and $Y\in \chinese(\A, 2)\cdot X\cdot \chinese(\A, 2)$. Because $X\in
\chinese(\A, 2)\cdot Y\cdot \chinese(\A, 2)$, there exists $U,V\in A^*$ such
that $UYV\equiv_{\chinese(2)}X$. By \cref{sz} one has $sz(Y)\leq sz(X)$.
Similarly, one has $sz(X)\leq sz(Y)$ thus $sz(Y)= sz(X)$. We can conclude
using the second part of \cref{sz} that $X\equiv_{\chinese(2)}Y$.
\end{proof}

By looking at \cref{fig:2-chinese j-order}, one easily sees that the
$\J$-order of the 2-Chinese monoid is not ranked for $|\A|\geq3$.

\begin{figure}
\centering
\scalebox{0.75}{
\begin{tikzpicture}
	\node []   (1) at (0,12) {$\emptyf$};
	\node []   (a) at (-3,10) {$\begin{pmatrix}
		a\\
		a
		\end{pmatrix}$
		};
	\node []   (b) at (0,10) {$\begin{pmatrix}
		b\\
		b
		\end{pmatrix}$
		};
	\node []   (c) at (3,10) {$\begin{pmatrix}
		c\\
		c
		\end{pmatrix}$
		};
	\node []   (ab) at (-3,8) {$\begin{pmatrix}
		a&b\\
		a&b
		\end{pmatrix}$
		};
	\node []   (ac) at (0,8) {$\begin{pmatrix}
		a&c\\
		a&c
		\end{pmatrix}$
		};
	\node []   (bc) at (3,8) {$\begin{pmatrix}
		b&c\\
		b&c
		\end{pmatrix}$
		};
	\node []   (ab-aa) at (-3,6) {$\begin{pmatrix}
		a&b\\
		a&a
		\end{pmatrix}$
		};
	\node []   (abc) at (0,6) {$\begin{pmatrix}
		a&b&c\\
		a&b&c
		\end{pmatrix}$
		};
	\node []   (bc-bb) at (3,6) {$\begin{pmatrix}
		b&c\\
		b&b
		\end{pmatrix}$
		};
	\node []   (aac) at (-3,4) {$\begin{pmatrix}
		a&b&c\\
		a&a&c
		\end{pmatrix}$
		};
	\node []   (ac-aa) at (0,4) {$\begin{pmatrix}
		a&c\\
		a&a
		\end{pmatrix}$
		};
	\node []   (abb) at (3,4) {$\begin{pmatrix}
		a&b&c\\
		a&b&b
		\end{pmatrix}$
		};
	\node []   (aab) at (0,2) {$\begin{pmatrix}
		a&b&c\\
		a&a&b
		\end{pmatrix}$
		};
	\node []   (aaa) at (0,0) {$\begin{pmatrix}
		a&b&c\\
		a&a&a
		\end{pmatrix}$
		};
	\draw (1) -- (a);
	\draw (1) -- (b);
	\draw (1) -- (c);
	\draw (a) -- (ab);
	\draw (a) -- (ac);
	\draw (b) -- (ab);
	\draw (b) -- (bc);
	\draw (c) -- (ac);
	\draw (c) -- (bc);
	\draw (ab) -- (ab-aa);
	\draw (ab) -- (abc);
	\draw [bend left=45] (ac) edge (ac-aa);
	\draw [bend left=10] (ac) edge (abb);
	\draw [bend right=10] (ac) edge (aac);
	\draw (bc) -- (abc);
	\draw (bc) -- (bc-bb);
	\draw (ab-aa) -- (aac);
	\draw (abc) -- (aac);
	\draw (abc) -- (abb);
	\draw (bc-bb) -- (abb);
	\draw (aac) -- (aab);
	\draw [bend left=50] (ac-aa) edge (aaa);
	\draw (abb) -- (aab);
	\draw (aab) -- (aaa);
\end{tikzpicture}}
\hfill
\scalebox{0.75}{
\begin{tikzpicture}
	\node []   (1) at (0,12) {
$\emptyset$
};
	\node []   (a) at (-3,10) {
$\Yboxdim{.5cm}\begin{tikzpicture}[scale=0.6,baseline={(current bounding box.east)} ]
\tgyoung(0cm,0cm,%
!\Ygr;1!:a,%
:a,%
)
\draw[ultra thick] (0,0) -- (.5,0) -- (.5,.5);
\end{tikzpicture}$
};
	\node []   (b) at (0,10) {
$\Yboxdim{.5cm}\begin{tikzpicture}[scale=0.6,baseline={(current bounding box.east)} ]
\tgyoung(0cm,0cm,%
!\Ygr;1!:b,%
:b,%
)
\draw[ultra thick] (0,0) -- (.5,0) -- (.5,.5);
\end{tikzpicture}$
};
	\node []   (c) at (3,10) {
$\Yboxdim{.5cm}\begin{tikzpicture}[scale=0.6,baseline={(current bounding box.east)} ]
\tgyoung(0cm,0cm,%
!\Ygr;1!:c,%
:c,
)
\draw[ultra thick] (0,0) -- (.5,0) -- (.5,.5);
\end{tikzpicture}$
};
	\node []   (ab) at (-3,8) {
$\Yboxdim{.5cm}\begin{tikzpicture}[scale=0.6,baseline={(current bounding box.east)} ]
\tgyoung(0cm,0cm,%
:!\Ygr;1!:a,%
!\Ygr;1!\Yw;!:b,%
:b:a,%
)
\draw[ultra thick] (0, -0.5) -- (.5, -0.5) -- (.5, 0) -- (1, 0) -- (1, 0.5);
\end{tikzpicture}$
};
	\node []   (ac) at (0,8) {
$\Yboxdim{.5cm}\begin{tikzpicture}[scale=0.6,baseline={(current bounding box.east)} ]
\tgyoung(0cm,0cm,%
:!\Ygr;1!:a,%
!\Ygr;1!\Yw;!:c,%
:c:a,%
)
\draw[ultra thick] (0, -0.5) -- (.5, -0.5) -- (.5, 0) -- (1, 0) -- (1, 0.5);
\end{tikzpicture}$
};
	\node []   (bc) at (3,8) {
$\Yboxdim{.5cm}\begin{tikzpicture}[scale=0.6,baseline={(current bounding box.east)} ]
\tgyoung(0cm,0cm,%
:!\Ygr;1!:b,%
!\Ygr;1!\Yw;!:c,%
:c:b,%
)
\draw[ultra thick] (0, -0.5) -- (.5, -0.5) -- (.5, 0) -- (1, 0) -- (1, 0.5);
\end{tikzpicture}$
};
	\node []   (ab-aa) at (-3,6) {
$\Yboxdim{.5cm}\begin{tikzpicture}[scale=0.6,baseline={(current bounding box.east)} ]
\tgyoung(0cm,0cm,%
:!\Ygr;!:a,%
!\Ygr;!;1!:b,%
:b:a,%
)
\draw[ultra thick] (0, -0.5) -- (1, -0.5) -- (1, 0.5);
\end{tikzpicture}$
};
	\node []   (abc) at (0,6) {
$\Yboxdim{.5cm}\begin{tikzpicture}[scale=0.6,baseline={(current bounding box.east)} ]
\tgyoung(0cm,0cm,%
::!\Ygr;1!:a,%
:!\Ygr;1!\Yw;!:b,%
!\Ygr;1!\Yw;!;!:c,%
:c:b:a,%
)
\draw[ultra thick] (0, -1) -- (.5, -1) -- (.5, -.5) -- (1, -.5) -- (1, 0) -- (1.5, 0) -- (1.5,.5);
\end{tikzpicture}$
};
	\node []   (bc-bb) at (3,6) {
$\Yboxdim{.5cm}\begin{tikzpicture}[scale=0.6,baseline={(current bounding box.east)} ]
\tgyoung(0cm,0cm,%
:!\Ygr;!:b,%
!\Ygr;!;1!:c,%
:c:b,%
)
\draw[ultra thick] (0, -0.5) -- (1, -0.5) -- (1, 0.5);
\end{tikzpicture}$
};
	\node []   (aac) at (-3,4) {
$\Yboxdim{.5cm}\begin{tikzpicture}[scale=0.6,baseline={(current bounding box.east)} ]
\tgyoung(0cm,0cm,%
::!\Ygr;!:a,%
:!\Ygr;!;1!:b,%
!\Ygr;1!\Yw;!;!:c,%
:c:b:a,%
)
\draw[ultra thick] (0, -1) -- (.5, -1) -- (.5, -.5) -- (1.5, -.5) -- (1.5,.5);
\end{tikzpicture}$
};
	\node []   (ac-aa) at (0,4) {
$\Yboxdim{.5cm}\begin{tikzpicture}[scale=0.6,baseline={(current bounding box.east)} ]
\tgyoung(0cm,0cm,%
:!\Ygr;!:a,%
!\Ygr;!;1!:c,%
:c:a,%
)
\draw[ultra thick] (0, -0.5) -- (1, -0.5) -- (1, 0.5);
\end{tikzpicture}$
};
	\node []   (abb) at (3,4) {
$\Yboxdim{.5cm}\begin{tikzpicture}[scale=0.6,baseline={(current bounding box.east)} ]
\tgyoung(0cm,0cm,%
::!\Ygr;1!:a,%
:!\Ygr;!\Yw;!:b,%
!\Ygr;!;1!\Yw;!:c,%
:c:b:a,%
)
\draw[ultra thick] (0, -1) -- (1, -1) -- (1, 0) -- (1.5, 0) -- (1.5,.5);
\end{tikzpicture}$
};
	\node []   (aab) at (0,2) {
$\Yboxdim{.5cm}\begin{tikzpicture}[scale=0.6,baseline={(current bounding box.east)} ]
\tgyoung(0cm,0cm,%
::!\Ygr;!:a,%
:!\Ygr;!;1!:b,%
!\Ygr;!;1!\Yw;!:c,%
:c:b:a,%
)
\draw[ultra thick] (0, -1) -- (1, -1) -- (1, -.5) -- (1.5, -.5) -- (1.5,.5);
\end{tikzpicture}$
};
	\node []   (aaa) at (0,0) {
$\Yboxdim{.5cm}\begin{tikzpicture}[scale=0.6,baseline={(current bounding box.east)} ]
\tgyoung(0cm,0cm,%
::!\Ygr;!:a,%
:!\Ygr;1!;!:b,%
!\Ygr;!;!;1!:c,%
:c:b:a,%
)
\draw[ultra thick] (0, -1) -- (1.5, -1) -- (1.5,.5);
\end{tikzpicture}$
};
	\draw (1) -- (a);
	\draw (1) -- (b);
	\draw (1) -- (c);
	\draw (a) -- (ab);
	\draw (a) -- (ac);
	\draw (b) -- (ab);
	\draw (b) -- (bc);
	\draw (c) -- (ac);
	\draw (c) -- (bc);
	\draw (ab) -- (ab-aa);
	\draw (ab) -- (abc);
	\draw [bend left=45] (ac) edge (ac-aa);
	\draw [bend left=10] (ac) edge (abb);
	\draw [bend right=10] (ac) edge (aac);
	\draw (bc) -- (abc);
	\draw (bc) -- (bc-bb);
	\draw (ab-aa) -- (aac);
	\draw (abc) -- (aac);
	\draw (abc) -- (abb);
	\draw (bc-bb) -- (abb);
	\draw (aac) -- (aab);
	\draw [bend left=50] (ac-aa) edge (aaa);
	\draw (abb) -- (aab);
	\draw (aab) -- (aaa);
\end{tikzpicture}}
\caption{2-chinese $\J$-order for $n=3$ with element represented as 2-chinese functions on the left and 2-Chinese staircases on the right. The anti-automorphism $\Theta$ stabilizes the middle column and, for each line, it exchanges the elements on the right with the ones on the left. \label{fig:2-chinese j-order}}
\end{figure}


\subsection{Quotient Type}
\label{ss-chinese-type} 

We are now in a position to prove that $\chinese(\A)$ is of type 1. As in the
case of the plactic monoid, we need to exhibit some properties of
$\chinese(\A,\sigma)$ throught a series of technical lemmas.

Note that, if $|\A|\geq 2$, $\chinese(\A)$ is not of type 2 for the same
reason than $\plax(\A)$: for $a<b$, $abab\equiv_{\chinese}
abba\equiv_{\chinese} baab$ but they are equivalent to neither $aabb$ nor
$bbaa$.

\begin{lemma}\label{lemma-push-pull}
If $A = a_1\cdots a_m \in \A ^ *$ and $b\in \A$ are such that $a_i < b$ for all
$i$, then $b ^ 2 A\equiv_{\chinese(\sigma)} b A b$.
\end{lemma}

\begin{proof}
We proceed by induction on $m = |A|$. If $|A| = 0$, then the statement holds
trivially. Suppose that $b ^ 2 a_1 \cdots a_{m- 1}\equiv_{\chinese(\sigma)}
ba_1 \cdots a_{m- 1} b$. Then
\begin{align*}
 b ^ 2 a_1 \cdots a_m &
 \equiv_{\chinese(\sigma)} b ^ {\sigma(b) + 1}a_1\cdots a_{m - 1}a_m,
 && b ^ {\sigma(b)} = b, \\
  &
  \equiv_{\chinese(\sigma)} b ^ {\sigma(b) - 1} a_1 \cdots a_{m -1} b ^ 2 a_m,
  &&
  \text{induction (twice)},
  \\
  &
  \equiv_{\chinese} b ^ {\sigma(b) - 1} a_1 \cdots a_{m -1}b a_m b,
  &&
  bba_m \equiv_{\chinese} ba_mb,
  \\
  &
  \equiv_{\chinese(\sigma)} b ^ {\sigma(b)} a_1 \cdots a_m b,
  &&
  \text{induction}
  \\
  &
  \equiv_{\chinese(\sigma)} b a_1 \cdots a_m b.
  &&
  \qedhere
\end{align*}
\end{proof}

\begin{cor}\label{cor-pull-push}
If $C = c_1\cdots c_m \in \A ^ *$ and $b\in \A$ are such that $b < c_i$ for
all $i$, then \mbox{$Cb ^ 2\equiv_{\chinese(\sigma)} b C b$.}
\end{cor}

\begin{proof}
It is a direct corollary of \cref{cor:theta2Chinese} and \cref{lemma-push-pull}.
\end{proof}

\begin{lemma}
Let $i,j,k\in\NN_{\geq 1}$ and $a<b<c\in\A$. Then
\begin{enumerate}[(i)]
\item if $i\leq k$, $c^kb^ja^i\equiv_{\chinese} c^ka^ib^j$,
\item if $i\geq k$, $c^kb^ja^i\equiv_{\chinese} b^jc^ka^i$.
\end{enumerate}
\end{lemma}

\begin{proof}
We shall prove the first statement.
If $i=0$, it is trivial.
Assume that it is true for all $0\leq i\leq k\leq n$. Consider $0< i\leq k=n+1$. We have
\begin{align*}
c^kb^ja^i&\equiv_{\chinese} c^kbab^{j-1}a^{i-1},\\
&\equiv_{\chinese} c^kab^ja^{i-1},\\
&\equiv_{\chinese} cac^{k-1}b^ja^{i-1},\\
&\equiv_{\chinese} cac^{k-1}a^{i-1}b^j, &\text{induction}\\
&\equiv_{\chinese} c^ka^{i}b^j.
\end{align*}

The second statement is proved in a similar way.
\end{proof}

As we have already seen in the case of the plactic monoid, the previous lemma
directly implies this one, see~\cref{cor:genrelsigmaplactic}.

\begin{cor}\label{cor:genrelsigmachinese}
For $i,j,k\in\NN_{\geq 1}$ and $a<b<c\in\A$, we have
\begin{enumerate}[(i)]
\item $c^k b^j a^i\equiv_{\chinese(\sigma)} c^k a^i b^j\equiv_{\chinese(\sigma)} b^j c^k a^i$,
\item $b^j a^i b^k\equiv_{\chinese(\sigma)} b^j b^k a^i$,
\item $b^j a^i a^k\equiv_{\chinese(\sigma)} a^i b^j a^k$.
\end{enumerate}
\end{cor}

\begin{prop}\label{prop:2chinese2sigmachinese}
Let $W,U\in\A^*$ such that $W\equiv_{\chinese(2)}U$. If $W^\#$ is an inflation
of $W$, then there exists $U^\#$ an inflation of $U$ such that
$W^\#\equiv_{\chinese(\sigma)}U^\#$.
\end{prop}

\begin{proof}
Since $W\equiv_{\chinese(2)}U$, there exists a sequence of 2-Chinese relations
going from $W$ to $U$. We only need to check the result for elementary
2-Chinese relations.

Assume that $W=V_1cbaV_2$ and $U=V_1cabV_2$, with $a<b<c$. This means that
\mbox{$W^\#=V_1^\# c^k a^i b^j V_2^\#$} for some inflations $V_1^\#,V_2^\#$ and
$i,j,k\in\NN$. By \cref{cor:genrelsigmachinese}, we have $$W^\#=V_1^\# c^k b^j
a^i V_2^\#\equiv_{\chinese(\sigma)}V_1^\# c^k a^i b^j V_2^\#=U^\#.$$

Using \cref{cor:genrelsigmachinese}, one can prove it the same way for any
elementary Chinese relations.

Now assume that $W=V_1xxV_2$ and $U=V_1xV_2$ for some $x\in\A$. This means
that \mbox{$W^\#=V_1^\# x^i x^j V_2^\#$} for some inflations $V_1^\#,V_2^\#$ and
$i,j\in\NN$. We then have $$W^\#=V_1^\# x^i x^j V_2^\#=V_1^\# x^{i+j}
V_2^\#=U^\#.$$

Finally, assume that $W=V_1xV_2$ and $U=V_1xxV_2$ for some $x\in\A$. This
means that \mbox{$W^\#=V_1^\# x^i V_2^\#$} for some inflations $V_1^\#,V_2^\#$ and
$i\in\NN$. We then have $$W^\#=V_1^\# x^i V_2^\#\equiv_{\chinese(\sigma)}V_1^\#
x^i x^{\sigma(x)-1} V_2^\#=U^\#,$$ hence the result.
\end{proof}


We now need an analogue of \cref{lemma-plactic-gathering}:

\begin{lemma}[Chinese Gathering Lemma]
\label{lemma-chinese-gathering-function}
Let $a_0a_1\cdots a_{k-1}b$ be an increasing word, $bc_1\cdots c_k$ astrictly increasing word and
$\alpha_0,\cdots,\alpha_{k-1},\beta,\gamma_1,\cdots,\gamma_k$ positive
integers.
Then
\begin{equation}
bA_0C_1A_1\cdots C_{k-1}A_{k-1}C_kB\equiv_{\chinese(\sigma)}BA_0C_2A_1\cdots
C_{k-1}A_{k-1}C_kb,
\end{equation}
where $A_i=a_i^{\alpha_i}$, $B=b^{\beta}$ and $C_i=c_i^{\gamma_i}$.
\end{lemma}

\begin{proof}
Note that, for all $1\leq i\leq k$, $a_i\leq b<c_i$. Then apply
\cref{cor:genrelsigmachinese} $k$ times to obtain the result. 
\end{proof}

Before stating the next result, we need a definition of some particular words.

\begin{de}
Let $\f$ be a 2-Chinese function.
A word $W\in\A^*$ is a \emph{left gathered $\sigma$-inflation of} $\RW(\f)$ if:
\begin{enumerate}[(i)]
\item $W$ is an inflation of $\RW(\f)$,
\item for any letter $a\in \cont(\RW(\f))$, only its first occurrence in
      $\RW(\f)$  may have been duplicated in $W$,
\item if $a^\alpha$ is a factor of $W$, $\alpha<\sigma(a)$.
\end{enumerate}
\end{de}

\begin{prop}
Let $\f$ be a 2-Chinese function.
If $W$ is an inflation of $\RW(\f)$ then there exists $U$ a left gathered
inflation of $\RW(\f)$ such that $W\equiv_{\chinese(\sigma)} U$.

Moreover, if $U$ and $V$ are two left gathered inflations of $\RW(\f)$,
$U\equiv_{\chinese(\sigma)}V$ implies $U=V$.
\end{prop}

\begin{proof}
It is a direct consequence of \cref{lemma-chinese-gathering-function} and the
definitions of a 2-Chinese function and its reading word.
\end{proof}

\begin{cor}\label{cor-chinese-inflations-function}
Let $\f$ be a 2-Chinese function, and $U$ and $V$ be inflations of $\RW(\f)$.
If \mbox{$\ev_{\sigma}(U) = \ev_{\sigma}(V)$}, then $U \equiv_{\chinese(\sigma)} V$.
\end{cor}

\begin{de}[$\sigma$-Chinese Reading]
For $(\f, E)\in \chinese(2)\times \Com(A, \sigma)$, if
\mbox{$\cont(\f)=\cont(E)$}, its \emph{$\sigma$-reading} $\RW_\sigma(\f, E)$ is the
left gathered inflation of $\RW(\f)$ such that
$\ev_\sigma(\RW_{\sigma}(W))=\ev_\sigma(E)$.
\end{de}

\begin{thm}\label{thm-sigma-chinese-normal-forms-function}
Let $W\in \A ^ *$.
Then $W \equiv_{\chinese(\sigma)} \RW_{\sigma}(\rd_W,\ev_{\sigma}(W))$.
\end{thm}

\begin{proof}
By \cref{chinese2-rd}, we know that $W\equiv_{\chinese(2)} \RW(\rd_W)$.
The rest is a direct corollary of \cref{prop:2chinese2sigmachinese} and
\cref{cor-chinese-inflations-function}.
\end{proof}

We can now state and prove the main theorem of this section, from which the
fact that the Chinese monoid is of type 1 is a direct corollary.

\begin{thm}
\label{thm-sigma-chinese-normal-forms}
Let $W_1, W_2\in \A ^ *$ and let $\sigma: \A \to \NN_{\geq 2}$ be arbitrary.
Then $W_1\equiv_{\chinese(\sigma)} W_2$ if and only if $W_1
\equiv_{\chinese(2)} W_2$ and $\ev_{\sigma}(W_1) =\ev_{\sigma}(W_2)$.
\end{thm}

\begin{proof}
Since $\chinese(\A,2)$ and $\Com(\A,\sigma)$ are both quotients of
$\chinese(\A,\sigma)$, the direct implication is trivial.
\cref{thm-sigma-chinese-normal-forms-function} gives the converse.
\end{proof}

\begin{cor}\label{cor-char-chinese}
The function $\phi_\sigma$ is an embedding and
$$\im(\phi) = \{(S, E) \in \chinese(\A,2) \times \Com(\A, \sigma) : \cont(S)=\cont(E)\}.$$
\end{cor}

\begin{proof}
We already know that $\phi_{\sigma}$ is a monoid morphism.
To prove the injectivity of $\phi_{\sigma}$, it is enough to show that
$\RW_{\sigma}\circ \phi_{\sigma} = \id_{\chinese(\A, \sigma)}$.
This follows from
$\RW_{\sigma}(\phi_{\sigma}(W))
 =\RW_{\sigma}(\rd(W),\ev_{\sigma}(W))\equiv_{\chinese(\A, \sigma)} W$,
which holds by \cref{thm-sigma-chinese-normal-forms-function}.

If $X \in \chinese(\A, \sigma)$ is arbitrary,
then $\cont(\theta(X)) = \cont(\ev_{\sigma}(X))$ because both $\chinese(\A,2)$
and $\Com(\A,\sigma)$ are content-preserving.
Conversely, for $(\f, E) \in \chinese(\A,2) \times \Com(\A, \sigma)$ such that $\cont(\f)=\cont(E)$,
$\rd(\RW_\sigma(\f, E)))=S$ because $\RW_\sigma(\f, E)$ is an inflation of $\RW(\f)$ and,
\mbox{$\ev_\sigma(\RW_{\sigma}(\f, E)) = E$} by definition of $\RW_\sigma$.
\end{proof}

\begin{cor}\label{cor:M-Type}
For a finite ordered alphabet $\A$, $\chinese(\A)$ is of type 1.
\end{cor}

\subsection{The \texorpdfstring{$\sigma$}{s}-Chinese Monoids}
\label{ss-chinese-enum-idem-iso}

As a corollary of \cref{thm-sigma-chinese-normal-forms-function}, the sizes of
the $\sigma$-Chinese monoids can be determined;
see \cref{table-sigma-chinese-sizes} for some values. 

\begin{cor}\label{cor-sigma-chinese-size}
If $\A$ is a finite alphabet and $\sigma: \A \to \NN$ is any function such that
$\sigma(a)\geq 2$ for all $a\in \A$, then
\begin{equation}\label{eq-sigma-chinese-size}
|\chinese(\A, \sigma)| = \sum_{\B\subseteq \A} \left(C_{|\B|} \prod_{b\in
\B}\left(\sigma(b) - 1\right)\right), \\
\end{equation}
where $C_{k}$ is the $k$-th Catalan number. If $\sigma(a) = p$, for every $a\in
\A$, then \eqref{eq-sigma-chinese-size} simplifies into
\begin{equation}
\label{eq-sigma-chinese-size-uniform}
|\chinese(\A, p)| = \sum_{k = 0}^{|\A|} \left(\binom{|\A|}{k}C_k (p - 1) ^
k\right)\,.
\end{equation}
If we write $S_p(n)=|\chinese(\A,p)|$ for any $\A$ of size $n$, then the
exponential generating function of $S_p(n)$ is given by
\begin{align}
\sum_{n=0}^\infty S_p(n)\frac{z^p}{p!}=\exp((2p-1)z) (I_0(2(p-1)z)-I_1(2(p-1)z)),
\end{align}
where $I_i$ is the $i$-th modified Bessel function of the first kind.
\end{cor}

\begin{proof}
Recall that the exponential generating series for Catalan numbers is
$$\sum_k C_k\frac{z^k}{k!} = \exp(2z)(I_0(2z)-I_1(2z)),$$ where $I_n$ is a
modified Bessel functions of the first kind, see the E.g.f. of
Sequence \href{https://oeis.org/A000108}{A000108} of~\cite{oeis}.
Using this with \cref{lem:expgenfunct} and
\cref{eq-sigma-chinese-size-uniform}, one gets
\begin{alignat*}{1}
\sum_{n=0}^\infty S_{p}(n)
= \exp(2(p-1)z) (I_0(2(p-1)z)-I_1(2(p-1)z)) \exp(z)\,.
\end{alignat*}
\end{proof}

\begin{table}[ht]
    \[
 \begin{array}{@{}l*{8}{r}@{}}
 \toprule
 \multirow{2}{*}{$p$} & \multicolumn{7}{c@{}}{|\A|}\\
     \cmidrule(lr){2-8}
     & 2 & 3 & 4 & 5 & 6 & 7 & \text{OEIS}\\
 \midrule
   2 & 5 & 15 & 51 & 188 & 731 & 2\,950 & \href{https://oeis.org/A007317}{A007317} \\
   3 & 13 & 71 & 441 & 2\,955 & 20\ 805 & 151\ 695 & \href{https://oeis.org/A162326}{A162326}\\
   4 &25& 199& 1\,795& 17\,422& 177\,463& 1\,870\,960 & \href{https://oeis.org/A337167}{A337167} \\
   5 & 41& 429& 5\,073& 64\,469& 859\,385& 11\,853\,949 & \href{https://oeis.org/A386387}{A386387}  \\
   6 & 61& 791& 11\,571& 181\,776& 2\,994\,531& 51\,038\,586 \\
 \bottomrule
 \end{array}
 \]
\caption{The sizes of $\chinese(\A,p)$ for $|\A|$ up to $7$ and $p$ up to $6$.}
\label{table-sigma-chinese-sizes}

\end{table}

Note that thanks to Equation~\eqref{eq-sigma-chinese-size-uniform}, each
column of~\cref{table-sigma-chinese-sizes} is a polynomial in $p$ of degree
$|\A|$, easily described as a polynomial in $p-1$.

The idempotents in the $\sigma$-Chinese monoids are similar to those in the
$\sigma$-plactic monoids (see~\cref{prop-plax-idempotents}).
For $\B=\{b_1 < b_2 < \cdots < b_k\}\subseteq \A$, define
$$e_{\sigma,\B} = b_2^{\sigma(b_2)-1}b_3^{\sigma(b_3) - 1}\cdots
b_k^{\sigma(b_k)-1} b_1^{\sigma(b_1)-1}.$$
If $\phi$ is the embedding defined in~\cref{def:type1},
then
\begin{equation}
\phi(e_{\B}) = \left(
\begin{pmatrix}
b_1 & b_2 & \cdots & b_k\\
b_1 & b_1 & \cdots & b_1\\
\end{pmatrix}, b_1 ^ {\sigma(b_1)-1} \cdots b_k ^ {\sigma(b_k)-1}\right).
\end{equation}

\begin{prop}
\label{prop-sigma-chinese-idempotents}
The idempotents of $\chinese(\A, \sigma)$ are precisely the $e_{\sigma,\B}$
for every $\B\subseteq \A$. In particular, there are $2^{|\A|}$ idempotents in
$\chinese(\A, \sigma)$.
\end{prop}

As a corollary of \cref{thm:isom}.(ii), one has

\begin{cor}\label{chinese-iso}
Let $\sigma, \tau:\A \to \NN_{\geq2}$.
Then $\chinese(\A,\sigma)\cong\chinese(\A,\tau)$ if and only if
$\sigma=\tau$.
\end{cor}

\section{The \texorpdfstring{$\sigma$}{s}-Hypoplactic Monoids}
\label{section-hypoplax}
\subsection{The Hypoplactic Monoid}
\label{ss-hypoplax}

The \emph{hypoplactic monoid} $\HPl(\A)$ was originally defined in~\cite{KT}
as a quotient of the plactic monoid $\plax(\A)$ by some extra quartic relations.
It has a simpler definition as the quotient of $\A^*$ by the following
relations~\cite{N}:
\begin{equation}\label{rel:hypoplax}
\begin{split}
acWb= caWb, \quad & \text{for } a \leq b < c\in\A,\\
bWac= bWca, \quad & \text{for } a < b \leq c\in\A,\\
\end{split}
\end{equation}
for any $W\in\A^*$. We define the \emph{hypoplactic congruence},
$\equiv_{\HPl}$ to be the least congruence on $\A ^ *$ containing
the relations \cref{rel:hypoplax}.
The hypoplactic monoid is evaluation-preserving.

As in \cref{def:quotient-monoid}, for $\sigma: \A \to \NN_{\geq 2}$,
we define the \emph{$\sigma$-hypoplactic congruence} to be the least
congruence on $\A ^ *$ containing the hypoplactic congruence $\equiv_{\HPl}$
and $a ^ {\sigma(a)} = a$ for every $a\in \A$; we denote this by
$\equiv_{\HPl(\sigma)}$. The $\sigma$-hypoplactic monoid on $\A$ is the
quotient $\A ^ * / {\equiv_{\HPl(\sigma)}}$ and is denoted by
$\HPl(\A, \sigma)$.

\paragraph{Quasi-Ribbon Tableaux\\}
\label{ss:QRTableaux}

A \emph{quasi-ribbon tableau} $Q$ is a pair $(D,f)$ consisting in a ribbon
diagram $D$ and a filling $f$ of $D$ by letters that are weakly increasing on
rows from left to right and strictly increasing on columns from \emph{top to
bottom} (see \cref{insere-d} for examples of quasi-ribbon tableaux).


\begin{defi}[Quasi-ribbon tableau right insertion]\label{def:QRTinsert}
Let $Q$ be a quasi-ribbon tableau over $\A$ and $x\in\A$.
Let $C$ be the rightmost and bottom-most box of $Q$ whose content is
smaller than or equal to $x$.
\begin{enumerate}[(1)]
\item If such a box does not exist, let $Q'$ be the quasi-ribbon tableau
obtained by gluing $Q$ underneath a box containing $x$.
\item If such a box $C$ exists, let $Q'$ be the quasi-ribbon tableau obtained
by
  \begin{enumerate}[(i)]
  \item cutting $Q$ in two after $C$ (the second part might be empty if $C$ is
        the last cell of $Q$),
  \item gluing a cell $C'$ containing $x$ to the right of $C$,
  \item gluing the remaining part of $Q$ underneath that new cell $C'$.
 \end{enumerate}
\end{enumerate}
The \emph{right insertion of $x$ into $Q$} is the quasi-ribbon tableau $Q'$.
We denote the insertion $Q\leftarrow x$.
\end{defi}

\begin{figure}
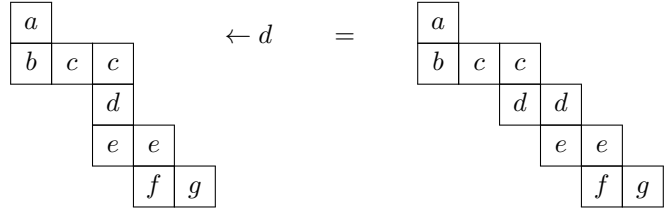

\centering
\ytableausetup{notabloids}
\begin{ytableau}
a\\
b & c & c\\
 \none & \none & d\\
 \none & \none & e & e\\
\none & \none & \none & f & g
\end{ytableau}
$\leftarrow d$
\qquad $ =$
\qquad
\begin{ytableau}
a\\
b & c & c\\
 \none & \none & d & d\\
 \none & \none & \none & e & e\\
\none & \none & \none & \none & f & g
\end{ytableau}
\caption{\label{insere-d}The right insertion of $d$ in the quasi-ribbon
tableau of $fgeedbcca$. It gives the tableau of canonical reading word
$fgeedbcca\cdot d\equiv_{\HPl} fgeeddbcca$.}
\end{figure}

\cref{insere-d} gives an example of such insertion.
As for Schensted insertion, one then defines iteratively the \emph{right
insertion of a word} $W=a_1\cdots a_n$ in a quasi-ribbon tableau by inserting
$a_1$, then $a_2$, and so on.
The resulting tableau is denoted by $\Q(W)=\emptyset\leftarrow W$.
We then have

\begin{thm}[\cite{KT}]
Two words $U,V\in\A^*$ are hypoplactically equivalent if and only if
\mbox{$\Q(U)=\Q(V)$}.
\end{thm}

The \emph{canonical row-reading} $\RW(Q)$ of a quasi-ribbon tableau $Q$
consists in reading its rows from left to right and from bottom to top.
A word is a \emph{quasi-ribbon word} if it is the canonical row-reading of a
quasi-ribbon tableau. For example, $fgeedbcca$ is a quasi-ribbon word since
it is the canonical reading of the quasi-ribbon tableau on the left of
\cref{insere-d}.

Thanks to its definition, a quasi-ribbon $Q$ is uniquely determined by
$(I,\ev(Q))$, where $I$ is the composition of its ribbon diagram and
$\ev(Q)=\ev(\RW(Q))$. We also define $\cont(Q)=\cont(\RW(Q))$.

The following result from \cite{KT} gives a set of representatives of
the hypoplactic monoid.

\begin{thm}[\cite{KT}]\label{hypoplax-rep}
The quasi-ribbon words form a section of the hypoplactic monoid and they are
the lexicographic maximal elements of their class.
\end{thm}

\subsection{Quotient Type}
\label{ss-hypoplac-type}

A quasi-ribbon word $W$ is \emph{$\sigma$-reduced} if $|W|_a<\sigma(a)$ for
all $a\in \A$ and similarly, a quasi-ribbon tableau $Q$ is
\emph{$\sigma$-reduced} if $\RW(Q)$ is $\sigma$-reduced.

\begin{prop}
\label{cor-hypo-gathered}
The hypoplactic monoid is of type 2.
\end{prop}

\begin{proof}
First, all equal letters are consecutive in the quasi-ribbon words so each
class has a \gathered element and since they are maximal in their classes, they
are their canonical \gathered elements, so that $\RW(W)\in\G_{\HPl}$.

Moreover, inserting two $x$ one immediately after the other amounts to
duplicate the $x$ in the quasi-ribbon tableau without changing anything else,
so the canonical \gathered element of an expansion is the same as the expansion
of the original canonical \gathered element. The same property holds for the
contraction.
\end{proof}

Thanks to \cref{gen-gathered}, we have

\begin{cor}\label{sqrandhsm}
The $\sigma$-reduced quasi-ribbon words form a section of the
$\sigma$-hypoplactic monoid.
\end{cor}

For a word $W$, $G_\sigma(W)$ is its $\sigma$-reduced quasi-ribbon word.

Denote by $\theta$ and $\ev_\sigma$ the projection of $\HPl(\A,\sigma)$ into
$\HPl(\A,2)$ and $\Com(\A,\sigma)$ respectively, as in \cref{ss:Type1Monoids}.

\begin{lemma}
Let $W\in\G_{\HPl(\sigma)}$. Then $W$ is an inflation of
$\theta(W)\in\G_{\HPl(2)}$ such that \mbox{$\ev(W)=\ev_\sigma(W)$}.
\end{lemma}

\begin{proof}
The proof is straightforward given the characterization of the
$\sigma$-reduced words and \cref{sqrandhsm}.
\end{proof}

\begin{cor}\label{hypoplax-type1}
The hypoplactic monoid is of type 1.
\end{cor}

\subsection{The 2-Hypoplactic Monoid}
\label{ss-2-hypoplax}

Before studying the general $\sigma$-quotients, let us consider the special
case $\HPl(\A,2)$.

\paragraph{Reversing anti-automorphism\\}
\label{2-Hypo-Theta}

It is easy to check (it is direct if one looks at the rewriting rules) that
$\Theta$ preserves hypoplactic classes and $2$-hypoplactic classes.
In the latter case, this involution is quite easy.

Let $I=(i_1,\cdots,i_n)$ be a composition. Define $flip(I)$ as
$(i_n,\cdots,i_1)$. Given a ribbon diagram $D$ associated with $I$,
let $flip(D)$ be the ribbon diagram associated with $flip(I)$. Now consider
$Q$, a $n$-reduced quasi-ribbon tableau of shape $D$ and define $flip(Q)$
as the reduced quasi-ribbon tableau of shape $flip(D)$ with
$\ev(flip(Q))=(\Theta(k_n),\cdots,\Theta(k_1))$ if $\ev(Q)=(k_1,\cdots,k_n)$.

For example, with $A=\{a,b,c,d,e,f,g\}$, the image by $flip$ of the
quasi-ribbon diagram on the left of~\cref{insere-d} is
\begin{equation}
\ytableausetup{notabloids}
\begin{ytableau}
a & b \\
\none & c & c \\
\none & \none & d\\
 \none & \none & e & e & f \\
\none & \none & \none & \none & g
\end{ytableau}
\end{equation}

The following result is immediate.

\begin{lemma}
Let $Q=(I,\B)$ be a $2$-reduced quasi-ribbon tableau.
Then $\Theta(\RW(Q))=\RW(flip(Q))$.
\end{lemma}

One easily checks the result on the previous example since $\Theta(fgeedbcca)
= geefdccab$.

\paragraph{The 2-hypoplactic insertions\\}
\label{2-Hypo-insert}

One can define the \emph{right insertion} of a letter $x$ in a
2-reduced quasi-ribbon tableau~$Q$. It is similar to \cref{def:QRTinsert}
but if $x\in\cont(R)$, one omits \emph{2.(ii)} and glues the second part 
of $Q$ underneath $C$ instead of $C'$ in step \emph{2.(iii)}.
We use the same notation, $Q\leftarrow x$, for the resulting tableau.
For the rest of this section, we only consider the insertion into
2-reduced quasi-ribbon tableaux.

One also defines the \emph{right insertion of a word} $W=w_1\cdots w_n$ in
a 2-reduced quasi-ribbon tableau $Q$ by right-inserting successively the
letters from $w_1$ to $w_n$ in $Q$.
Let $\Q_{2}(W)$ be the 2-reduced quasi-ribbon tableau obtained by right
inserting $W$ in the empty tableau.

Thanks to \cref{sqrandhsm}, two words $U,V\in\A^*$ are 2-hypoplactic
equivalent if and only if \mbox{$\Q_2(U)=\Q_2(V)$}.

\begin{figure}
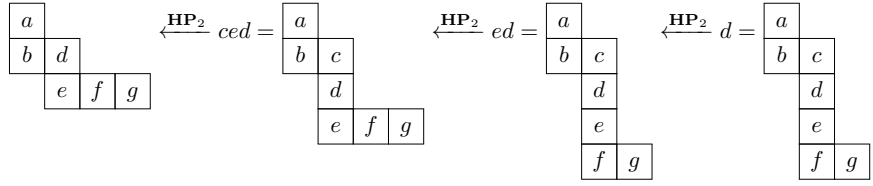

\centering
\scalebox{0.86}{
\ytableausetup{notabloids}
\begin{ytableau}
a\\
b & d \\
 \none & e & f & g
\end{ytableau}
$\Hlleft ced =$
\begin{ytableau}
a\\
b & c\\
 \none & d\\
 \none & e & f & g
\end{ytableau}
$\Hlleft ed =$
\begin{ytableau}
a\\
b & c\\
 \none & d\\
 \none & e\\
 \none & f & g
\end{ytableau}
$\Hlleft d =$
\begin{ytableau}
a\\
b & c\\
 \none & d\\
 \none & e\\
 \none & f & g
\end{ytableau}}
\caption{\label{ninsere}The right insertion of $ced$ in the quasi-ribbon
tableau $\Q_2(efgbda)$. It gives the tableau of canonical reading word
$fgedbca$.}
\end{figure}

Thanks to the flip property, the \emph{left insertion of} $x\in \A$ in the
2-reduced quasi-ribbon tableau $Q$ is
$x\rightarrow Q = flip(flip(Q)\leftarrow \Theta(x)).$

\paragraph{Idempotents\\}
\label{2-Hypo-idem}

Knowing both the left and right insertions, it is easy to describe the set
of idempotents of the 2-hypoplactic monoid:
they are exactly the same than in the stylic monoid: the strictly decreasing
words.
For $\B=\{b_1 < b_2 < \cdots < b_k\} \subseteq \A$, let
$e_{\B} = b_k\cdots b_1$.

\begin{prop}
\label{hypo-idem}
The idempotents of $\HPl(\A, 2)$ are precisely the $e_{\B}$ for every
$\B\subseteq \A$.
In particular, there are $2^{|\A|}$ idempotents in $\HPl(\A, 2)$. 
\end{prop}
 
\begin{proof}
Note that the insertion of a strictly decreasing word is a column quasi-ribbon
tableau, that is, a quasi-ribbon tableau of shape $(1,1,\dots,1)$. It is a
$2$-reduced quasi-ribbon tableau.

Now, inserting any letter already present is such a tableau yields the same
tableau as result (see the last step in ~\cref{ninsere}).
So any strictly decreasing word is an idempotent.

Conversely, if a quasi-ribbon tableau has two letters $x<y$ on the same row,
it cannot be an idempotent since inserting $x$ in that tableau would yield a
tableau with one more row than the previous one. So this element is not an
idempotent.
\end{proof}

\paragraph{$\J$-triviality and $\J$-order\\}
\label{2-Hypo-Jtriv}

Theorem 11.1 of~\cite{AR} shows that the stylic monoid is $\J$-trivial.
It is well-known that the class of finite $\J$-trivial monoids is a
pseudovariety (see \emph{e.g.}, Chapter 3 Section 4 of \cite{P}) which is, by
definition, closed under quotienting. Therefore, the 2-hypoplactic monoid,
being a quotient of the stylic monoid by the quartic relations, is
$\J$-trivial.

Using 2-reduced quasi-ribbon tableaux and their insertions, one can describe
the $\J$-order of the 2-hypoplactic monoid.
See \cref{posetJ} for the $\J$-order for $|\A|=3$.

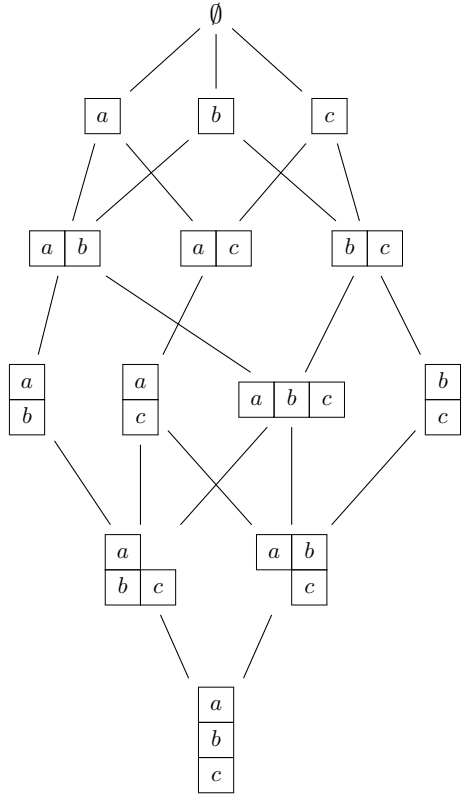
\begin{figure}[ht]
\centering
\ytableausetup{notabloids}
\begin{tikzpicture}
\node []   (1) at (0,9.35) {$\emptyset$};

\node []   (a) at (-1.5,8) {\scalebox{0.86}{$\begin{ytableau}	a\end{ytableau}$}};
\node []   (b) at (0,8) {\scalebox{0.86}{$\begin{ytableau}	b\end{ytableau}$}};
\node []   (c) at (1.5,8) {\scalebox{0.86}{$\begin{ytableau}	c\end{ytableau}$}};

\node []   (ab) at (-2,6.25) {\scalebox{0.86}{$\begin{ytableau} a & b\end{ytableau}$}};
\node []   (ac) at (0,6.25) {\scalebox{0.86}{$\begin{ytableau} a & c\end{ytableau}$}};
\node []   (bc) at (2,6.25) {\scalebox{0.86}{$\begin{ytableau} b & c\end{ytableau}$}};

\node []   (ba) at (-2.5,4.25) {\scalebox{0.86}{$\begin{ytableau} a \\ b\end{ytableau}$}};
\node []   (ca) at (-1,4.25) {\scalebox{0.86}{$\begin{ytableau} a \\ c\end{ytableau}$}};
\node []   (abc) at (1,4.25) {\scalebox{0.86}{$\begin{ytableau} a & b & c\end{ytableau}$}};
\node []   (cb) at (3,4.25) {\scalebox{0.86}{$\begin{ytableau} b \\ c\end{ytableau}$}};

\node []   (acb) at (-1,2) {\scalebox{0.86}{$\begin{ytableau} a \\ b & c\end{ytableau}$}};
\node []   (bac) at (1,2) {\scalebox{0.86}{$\begin{ytableau} a & b \\ \none & c\end{ytableau}$}};

\node []   (cba) at (0,-0.25) {\scalebox{0.86}{$\begin{ytableau} a \\ b \\ c\end{ytableau}$}};
		
\draw (1) -- (a); \draw (1) -- (b); \draw (1) -- (c);
		
\draw (a) -- (ab); \draw (a) -- (ac); \draw (b) -- (ab); \draw (b) -- (bc);
\draw (c) -- (ac); \draw (c) -- (bc);
		
\draw (ab) -- (ba); \draw (ab) -- (abc); \draw (ac) -- (ca); \draw (bc) --
(cb); \draw (bc) -- (abc);

\draw (ba) -- (acb); \draw (ca) -- (acb); \draw (ca) -- (bac); \draw (abc) --
(acb); \draw (abc) -- (bac); \draw (cb) -- (bac);
		
\draw (acb) -- (cba); \draw (bac) -- (cba);
\end{tikzpicture}
\caption{\label{posetJ}%
The poset of the $\J$-order of the 2-hypoplactic monoid on $\A=\{a,b,c\}$, with
$\emptyset$ being the empty tableau.}
\end{figure}

Let $Q$ be the 2-reduced quasi-ribbon tableau associated with the pair
$(I,\B)$, with $I=(i_1,\cdots,i_k)$ and $\B\subseteq\A$.
Define $\Delta(Q)$ as 0 if $Q$ is the empty tableau and $|B|+k-1$ if $Q$ is not empty.

\begin{thm}
The $\J$-order of the 2-hypoplactic monoid on $\A$ is graded. The co-rank
function is $\Delta$.
\end{thm}

The proof is a direct consequence of the upcoming \cref{couverture}. We
first need a simple result.

\begin{lemma}
\label{deltaflip}
Let $Q$ be a 2-reduced quasi-ribbon tableau,
then $\Delta(Q)=\Delta(flip(Q))$.
\end{lemma}

\begin{proof}
The map $flip$ does not change the number of boxes neither the height of a
diagram.
Therefore, $\Delta(Q)=\Delta(flip(Q))$.
\end{proof}

\begin{lemma}
\label{couverture}
Let $Q$ be a 2-reduced quasi-ribbon tableau and $x\in \A$.
\begin{enumerate}[(1)]
\item
  Then $\Delta(Q\leftarrow x)= \Delta(Q) + \gamma$
  with $\gamma\in \{0,1,2\}$.
  And if $\gamma=2$, there exists $y$ such that
  $\Delta(y\rightarrow Q)=\Delta(Q)+1$ and $(y\rightarrow Q)\leftarrow x=Q\leftarrow x$.
\item
  Then $\Delta(x\rightarrow Q)= \Delta(Q) + \gamma$
  with $\gamma\in \{0,1,2\}$.
  And if $\gamma=2$, there exists $y$ such that
  $\Delta(Q\leftarrow y)=\Delta(Q)+1$ and $x\rightarrow (Q\leftarrow y)=x\rightarrow Q$.
\end{enumerate}
\end{lemma}

\begin{proof}
Using $flip$, \cref{deltaflip}, and the definition of the left
insertion, $(2)$ is a direct consequence of $(1)$. Therefore we only have
to prove $(1)$.

Note that when doing the right insertion of $x$ in $Q$, one adds at most
one box to the quasi-ribbon tableau and this only happens if
$x\notin \cont(Q)$.
Moreover, inserting $x$ in $Q$ can add at most a row and this happens only if
there exist two letters $y\leq x<z$ in $\cont(Q)$ that are on the same row in
$Q$ or if $x<z$ for all $z\in\cont(Q)$. Therefore, $\gamma\in \{0,1,2\}$,
hence the first part of the statement.

Now, consider the case where $\gamma=2$. Then by definition of
$\Delta$, $Q$ cannot be the empty tableau.
By definition of the insertion, we have
$x\notin \cont(Q)$ and either there exist letters $y< x<z$ in $\cont(Q)$
that are in the same row or $x<z$ for all $z\in\cont(Q)$.

Let $x^-=\max(z\in\cont(Q)|z<x)$ and $x^+=\min(z\cont(Q)|z>x)$.
In the first case, both letters exist. The left insertion of $x^+$ in $Q$ does
not add any box and only adds a row by gluing $x^+$ and the rest of $R$ below
$x^-$.
Then doing the right insertion of $x$ afterwards only adds the box containing
$x$, gluing $x^+$ and the rest of $R$ below it, which is exactly the tableau
$R\leftarrow x$.

In the second case, $x^-$ does not exist. Then, the left insertion of $x$ only
glues $x$ on the left of $x^+$, adding only one box and no row. The right
insertion of $x$ only adds a row and clearly $(x\rightarrow Q)\leftarrow
x=Q\leftarrow x$, hence the second part of the statement in all cases.
\end{proof}

\paragraph{Enumeration\\}
\label{2-Hypo-enum}

The total number of 2-hypoplactic classes is equal to the number of
2-reduced quasi-ribbon tableaux.
Since a 2-reduced quasi-ribbon tableau $Q$ is uniquely determined by its
ribbon diagram and its content, \emph{i.e.}, $Q=(I,\B)$ such that
$\B\subset \A$, and $I$ being a composition of $|\B|$,
one can easily enumerate classes of $\HPl(\A,2)$.

\begin{prop}\label{2-hypo-enum}
Let $\A$ be a finite alphabet. Then
\begin{equation}
|\HPl(\A,2)| = 1+\sum_{k=1}^{|\A|} \binom{|\A|}{k} 2^{k-1}
             = \frac{3^{|\A|} + 1}{2}.
\end{equation}
\end{prop}

\subsection{The \texorpdfstring{$\sigma$}{s}-Hypoplactic Monoids}
\label{ss-hypoplax-enum-idem-iso}

Let us begin with the enumeration of the $\sigma$-hypoplactic monoids.
Since the hypoplactic monoid is a type 1 monoid, one has the following result.

\begin{cor}\label{hypo-enum}
Let $\A$ be a finite alphabet.
Then the $\HPl(\A,\sigma)$ satisfies
\begin{equation}\label{equ}
|\HPl(\A,\sigma)| = 
1+ \sum_{\B\subseteq \A ; \B\neq \emptyset}
     2^{|\B|-1}\prod_{b\in\B} \sigma(b).
\end{equation}
If $\sigma(a)=p$ for all $a\in \A$ then Equation~\eqref{equ} simplifies into
\begin{equation}
|\HPl(\A,p)|=1+\sum_{k=1}^{|\A|} \binom{|\A|}{k} 2^{k-1} (p-1)^{k}
= \frac{(2p-1)^{|\A|}+1}{2}.
\end{equation}

If we let $S_p(n)=|\HPl(\A,p)|$ for any $\A$ of size $n$, then the exponential
generating function of $S_p(n)$ is given by
\begin{align}
	\sum_{n=0}^\infty S_p(n)\frac{z^n}{n!}=\frac{\exp((2p-1)z)+\exp(z)}{2}.
\end{align}
\end{cor}

The next result is a corollary of \cref{hypo-idem} and \cref{hypoplax-type1}.

\begin{cor}\label{HPsigmaIdem}
Let $\sigma:\A\to\NN_{\geq2}$. The set of idempotents of the
$\sigma$-hypoplactic monoid is exactly the set of word of the form
$b_k^{\sigma(b_k)-1}\cdots b_1^{\sigma(b_1)-1}$ for
$\{b_1<\cdots<b_k\}\subset\A$.
In particular, there are $2^{|\A|}$ of them.
\end{cor}

As a corollary of \cref{thm:isom}.(i), one has

\begin{cor}\label{Hpl-iso}
Let $\sigma, \tau:\A \to \NN_{\geq2}$.
Then $\HPl(\A,\sigma)\cong\HPl(\A,\tau)$ if and only if $\sigma=\tau$.
\end{cor}

\section{The \texorpdfstring{$\sigma$}{s}-Sylvester Monoids}
\label{section-sylv}
\subsection{The Sylvester Monoid}

The \emph{sylvester congruence}~\cite{HNT05}, denoted $\equiv_{\Sylv}$, is the
smallest congruence on $\A^*$ containing the (infinite) set of relations:
\begin{equation}\label{eq-sylvester}
acWb = caWb \text{ for } a \leq b < c\in\A,
\end{equation}
where is any $W\in \A^*$. The \emph{sylvester monoid} is
$\A ^ */{\equiv_{\Sylv}}$ and is denoted by $\Sylv(\A)$.

As in \cref{def:quotient-monoid}, for $\sigma: \A \to \NN_{\geq 2}$,
we define the \emph{$\sigma$-sylvester congruence} to be the least congruence
on $\A ^ *$ containing the sylvester congruence and $a ^ {\sigma(a)} = a$ for
every $a\in \A$; we denote this by $\equiv_{\Sylv(\sigma)}$. The
$\sigma$-sylvester monoid on $\A$ is the quotient $\A ^ * /
{\equiv_{\Sylv(\sigma)}}$ and is denoted by $\Sylv(\A, \sigma)$.

\paragraph{The Binary Search Trees\\}

A \emph{binary search tree} $T = (L, R)$ is a binary tree labelled by $\A$
such that the label of each node is greater than or equal to all labels in its
left subtree $L$ and strictly smaller than all labels in its right subtree
$R$.
Given a binary search tree $T$, the \emph{right to left postfix reading} (also
called the reading) of $T$ is the word $\RW(T)$
obtained by recursively visiting the right subtree first then
the left one, and finally the root. For example, the reading of the tree on
the left of \cref{insere-3} is $768541255$.  The \emph{left to right infix
search} of a tree $T$, denoted $\R(T)$, consists in recursively visiting the
left subtree of the root, then the root, and finally its right subtree.
With the same tree as before, one obtains the word $124555678$.
The \emph{content} $\cont(T)$ of $T$ is defined to be $\cont(\RW(T))$.

Finally, the \emph{left branch} of a node $r$ in a binary tree is the subtree
consisting of all the nodes on the path starting at $r$ and only going left.
The \emph{right branch} of a node in a binary tree is defined analogously.
\newpage
\begin{note}
\label{note-bst}
Binary search trees have many well-known properties and we shall use these
without proof:
\begin{enumerate}[(1)]
\item All nodes with equal label are on a path from the root to a leaf.
\item If a node and its parent have equal labels, this node has an empty
right subtree.
\item Its left to right infix search is the weakly increasing word of content
$\cont(T)$. The letter immediately before a given node in that search in the
bottom extremity of the right branch of its left subtree.
\end{enumerate}
\end{note}

The left insertion of a letter into a binary search tree is defined
the following way.

\begin{defi}[Binary search tree insertion]
\label{def:BSTinsert}
Let $T=(L,R)$ a binary search tree and $x\in\A$.
\begin{enumerate}[(1)]
\item If $T$ is the empty tree, let $T'$ be the tree with only one node
labelled $x$.
\item If $T$ is not empty, then the root of $T$ is labelled by some $y\in \A$.
\begin{enumerate}[(i)]
\item If $x\leq y$, let $T'=(x\rightarrow L,R)$,
\item if $x > y$, let $T'=(L,x\rightarrow R)$.
\end{enumerate}
The resulting tree $T'$ is the left insertion of $x$ into $T$ and is denoted
$x\rightarrow T$.
\end{enumerate}
\end{defi}

Given a word $W=a_1\cdots a_n$, let $\T(W)$ be the binary search tree
obtained by inserting $a_n$ in the empty tree, then $a_{n-1}$ in the
resulting tree and so on. Note that, for a binary search tree $T$,
$\T(\RW(T))=T$ (see ~\cref{insere-3} for an example of the insertion
algorithm).

It is shown in \cite{HNT05} that for all $U,V\in \A ^ *$,
$U\equiv_{\Sylv}V$ if and only if $\T(U)=\T(V)$.
It is also shown that $\RW(T)$ is the lexicographic maximal element of its
class.

\begin{figure}
\centering
\begin{forest}
for tree={grow=south, minimum size=2ex, inner sep=0, s sep=1mm, l=0}
[,no edge[$3\rightarrow$,no edge[,no edge[,no edge[,no edge]]]]]
\end{forest}
\begin{forest}
for tree={grow=south, circle, draw, minimum size=3ex, inner sep=0, s sep=1mm, l=0}
[5[5
  [2[1]
      [,no edge, draw=none]
      [4[,no edge, draw=none]
        [5]]]
    [,no edge, draw=none]
    [,no edge, draw=none]]
  [8[6[,no edge, draw=none]
      [7]]
    [,no edge, draw=none]
    [,no edge, draw=none]]]
\end{forest}
\begin{forest}
for tree={grow=south, minimum size=2ex, inner sep=0, s sep=1mm, l=0}
[,no edge[\text{=},no edge[,no edge[,no edge[,no edge]]]]]
\end{forest}
\begin{forest}
for tree={grow=south, circle, draw, minimum size=3ex, inner sep=0, s sep=1mm, l=0}
[5[5
  [2[1]
      [,no edge, draw=none]
      [4[3]
        [5]]]
    [,no edge, draw=none]
    [,no edge, draw=none]]
  [8[6[,no edge, draw=none]
      [7]]
    [,no edge, draw=none]
    [,no edge, draw=none]]]
\end{forest}
\caption{\label{insere-3}The (left) insertion of $3$ in a binary search tree.}
\end{figure}
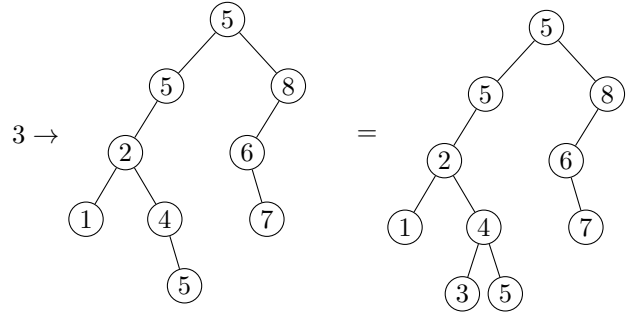

\subsection{Quotient Type}
\label{ss-sylv-type}

When $|\A|\geq 2$, the monoid $\Sylv(\A)$ is not a type 1 monoid. Indeed, for
$a<b$, using \cref{ss:Type1Monoids} notation, one has
$\theta(bbab)=\theta(babb)=bab$ and
$\ev_{\sigma}(bbab)=\ev_{\sigma}(babb)=abbb$
but they are not $\sigma$-sylvester equivalent if $\sigma(b)> 2$.

We will prove that it is of type 2.

\begin{lemma}
\label{SylvWord}
Let $\A$ be an alphabet and $W\in\A^*$. Let $T$ be its corresponding
binary search tree.
Then if $W$ has two consecutive equal letters, the left one is the
left child of the other one in $T$, so that they are consecutive in $\RW(T)$.

Moreover, if these two letters were the $i$-th and $(i+1)$-st occurrences
of an $x$ in $W$, they are also the $i$-th and $(i+1)$-st occurrences of $x$
in $\RW(T)$.
\end{lemma}

\begin{proof}
The first part directly comes from the definition of the insertion algorithm.

Now, consider all occurrences of $x$ in $W$. The ones before the $i$-th one
get inserted after those two so end up below those two in $T$. The ones after
the $(i+1)$-st one get inserted before those two, so end up between the
root and those two (see \cref{note-bst}) in $T$.
Reading the tree $T$ in the \emph{right to left postfix reading}, the
occurrences of $x$ are read from bottom to top, hence preserving their
original order.
\end{proof}

\begin{prop}
\label{cor-sylv-gathered}
The sylvester monoid is a \gathered monoid.
\end{prop}

\begin{proof}
First, \cref{SylvWord} shows that each class has a \gathered element,
$\RW(T)$. Since $\RW(T)$ is the maximal element of its class,
$$\G_\Sylv=\set{\RW(T)}{T \text{ is a binary search tree}}.$$

Now, we just need to see what happens when one expands or contracts an $x$ in
the insertion algorithm.
If one inserts two consecutive $x$ in a tree, the resulting tree is
exactly the same as the one obtained by inserting only one $x$ except near
that $x$: the node $x$ is duplicated as a left child of itself, its left
subtree becoming the left subtree of this new child.
The same is also true for the contraction and is well-defined since if a left
child has same label as its parent, its right subtree is empty, hence defining
the contraction without ambiguity (see \cref{note-bst}).
\end{proof}

A binary search tree $T$ is $\sigma$-\emph{reduced} if for all $a\in\A$, there
are no sequences of $\sigma(a)$ consecutive nodes labelled by $a$.
For example, then every child and parent in a $2$-reduced binary search tree
have different labels in $\A$.

Thanks to the fact that the sylvester monoid is \gathered, we have

\begin{cor}
\label{redtrees}
The $\sigma$-reduced trees form a section of the $\sigma$-sylvester monoid.
\end{cor}

We denote by $\G_{\Sylv(\sigma)}$ the set of $\sigma$-reduced \gathered words,
the right to left postfix reading of $\sigma$-reduced trees.

\subsection{The 2-Sylvester Monoid}
\label{ss:2-sylv}

\paragraph{Left Insertion\\}

The $2$-insertion $T'$ of $x$ in a $2$-reduced binary search tree
$T$ is defined the following way.

\begin{defi}[2-reduced binary tree insertion]
\label{def:2-redtree}
Let $T$ be a $2$-reduced binary search tree and $x\in\A$.
\begin{enumerate}[(1)]
\item If $x\notin \cont(T)$, then $T'$ is the left insertion of $x$ in $T$ as
in \cref{def:BSTinsert}.
\item If $x\in \cont(T)$, let $s$ be the deepest node labelled by $x$.
\begin{enumerate}[(i)]
\item If $s$ has a non-empty left subtree $L$, $T'$ is the left insertion of
$x$ in $T$ as in \cref{def:BSTinsert}.
\item If $s$ does not have a left child, $T'=T$.
\end{enumerate}
\end{enumerate}
\end{defi}

Then, given a word $W=w_1\cdots w_n$, we define $\T_2(W)$ to be the
$2$-reduced binary search tree obtained by $2$-inserting $w_n$ into
the empty tree, then $w_{n-1}$, and so on.

Clearly, one has $U\equiv_{\Sylv(2)} V$ if and only if $\T_2(U)=\T_2(V)$.

Note that the 2-sylvester monoid is not $\J$-trivial. Indeed, consider
the words $212$ and $21$. We have $212\cdot1=21$ and $21\cdot2=212$, so
$\Sylv(\{1,2\}, 2)\cdot212\cdot \Sylv(\{1,2\},2)
=\Sylv(\{1,2\},2)\cdot21\cdot \Sylv(\{1,2\},2)$ but
$212\not\equiv_{\Sylv(\{1,2\}, 2)} 21$.
But it is $\L$-trivial since after inserting a letter, either the tree is the
same or its number of nodes strictly increases.

\paragraph{Enumeration\\}

Given an alphabet $\A=\{a_1<a_2<\cdots<a_p\}$, a $2$-reduced binary
search tree $T$ is \emph{initial} if $T$ is either the empty tree, or if
$\cont(T)=\{a_1,a_2\cdots, a_i\}$ for some $1\leq i\leq n$.
We shall call these trees \emph{IT} trees.
A node is a \emph{pivot} if it is a left child and has a non-empty right
subtree.

In order to enumerate all $2$-reduced binary search trees on an alphabet $A$
with $p$ letters, it is enough to enumerate the trees that in addition are
initial since any tree is up to an increasing relabeling of its nodes
isomorphic to only one initial tree.
Numerical evidence suggests that the number of IT trees with $n$ nodes and $k$
different labels are counted by the so-called Borel triangle~(see
Sequence~A234950 of~\cite{oeis}), whose first rows are given
in~\cref{table-borel-std}.

\begin{table}[ht]
\begin{equation*}
\begin{array}{l|rrrrrr}
{n\backslash\, k} &  0 & 1 & 2 & 3 &  4 & 5\\
\midrule
     0 & 1 &   &   \\
     1 & 2 & 1 &   \\
     2 & 5 & 6 & 2 \\
     3 &14 &28 &20 & 5 \\
     4 &42 &120&135& 70& 14 \\
     5 &132&495&770&616&252&42\\
\bottomrule
\end{array}
\end{equation*}
   \caption{The Borel triangle $B_{n,k}$.}
   \label{table-borel-std}
\end{table}

More precisely, the statement we will prove is that there are
$B_{n\!-\!k\!-\!1,k}$ IT trees with $n$ nodes and $n-k$ different labels. For
example, the $6$ in the table means that there are $6$ such trees with $4$
nodes and $3$ different labels; see \cref{ex-sylv-6-43} to see all these six trees.
%
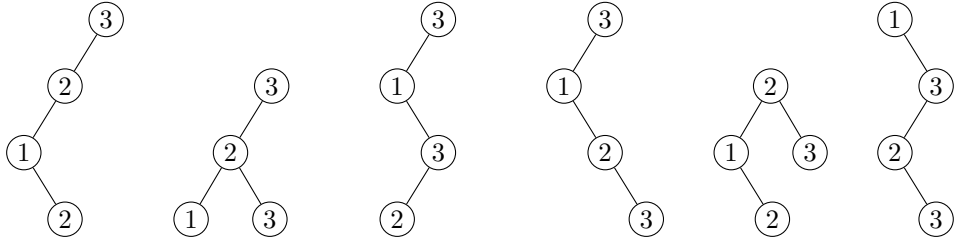
\begin{figure}
\centering
\begin{forest}
for tree={grow=south, circle, draw, minimum size=3ex, inner sep=0, s sep=1mm, l=0}
[3 [2 [1[,no edge, draw=none][,no edge, draw=none][2]]
      [,no edge, draw=none]
      [,no edge, draw=none]
   ]
      [,no edge, draw=none]
      [,no edge, draw=none]
]
\end{forest}
\begin{forest}
for tree={grow=south, circle, draw, minimum size=3ex, inner sep=0, s sep=1mm, l=0}
[3 [2
      [1] [,no edge, draw=none] [3]
   ]
      [,no edge, draw=none]
      [,no edge, draw=none]
]
\end{forest}
\begin{forest}
for tree={grow=south, circle, draw, minimum size=3ex, inner sep=0, s sep=1mm, l=0}
[3 [1
      [,no edge, draw=none]
      [,no edge, draw=none]
      [3 [2][,no edge, draw=none] [,no edge, draw=none]]
   ]
      [,no edge, draw=none]
      [,no edge, draw=none]
]
\end{forest}
\begin{forest}
for tree={grow=south, circle, draw, minimum size=3ex, inner sep=0, s sep=1mm, l=0}
[3 [1
      [,no edge, draw=none]
      [,no edge, draw=none]
      [2 [,no edge, draw=none] [,no edge, draw=none][3]]
   ]
      [,no edge, draw=none]
      [,no edge, draw=none]
]
\end{forest}
\begin{forest}
for tree={grow=south, circle, draw, minimum size=3ex, inner sep=0, s sep=1mm, l=0}
[2 [1
      [,no edge, draw=none]
      [,no edge, draw=none][2]]
   [,no edge, draw=none]
   [3]
]
\end{forest}
\begin{forest}
for tree={grow=south, circle, draw, minimum size=3ex, inner sep=0, s sep=1mm, l=0}
[1 [,no edge, draw=none] [,no edge, draw=none]
   [3 
      [2 [,no edge, draw=none] [,no edge, draw=none][3]]
      [,no edge, draw=none] [,no edge, draw=none]
   ]
]
\end{forest}
\caption{\label{ex-sylv-6-43} The six 2-reduced trees with four nodes using exactly 3 letters.}
\end{figure}

Using our numerical evidence, we prove this property by providing a bijection
between the IT trees and the trees considered in the paper of Cai and
Yan~\cite{CY} found in the OEIS page of the Borel triangle. Their
combinatorial objects are the binary trees with some marked leaves where
the rightmost leaf can never be marked. More precisely, the number of
such trees with $n$ vertices and $k$ marked leaves is $B_{n-k-1,k}$, with the
extra convention $B_{-1,0}=1$.

The bijection between our IT trees and their trees is the composition of two
simple bijections.
First consider an IT tree having two nodes with the same label $\ell$. We know
they are on the same path from a leaf to the root of the tree and that the
bottom one $b$ belongs to the left subtree of the top one $t$ by the
equality conditions on the binary search trees. Let $t'$ be the left child of
$t$.
Then $b$ belongs to the right subtree of $t'$ since its label is strictly
greater than the label of $t'$ since the IT tree is in particular $2$-reduced.
Moreover, the node $x$ at the bottom extremity of the right branch starting in
$t'$ has label $\ell$ again. Indeed, consider the left to
right infix reading of that tree. Node $x$ is the node that is read
immediately before node $t$ and since there is a letter equal to the label of
$t$ before it and since the left to right infix reading is weakly increasing,
the property follows. So in the path between two consecutive equal labels in
an IT tree, there is exactly one node that is a pivot, the left child of $t$.
With that in mind, the following property is very natural:

\begin{prop}
There is a bijection between IT trees and unlabelled binary trees with some
pivots marked. The number of nodes stays the same and if an IT tree has $n-k$
different letters, its corresponding tree has $k$ marked pivots.

Starting with an IT tree, its image is the tree of the same shape whose
marked pivots are the nodes such that their parent and the extremity of their
right branch have the same label.
The reverse bijection consists in linking the parent of a marked pivot with
the extremity of the right branch of its child and then number the nodes in
left to right infix order starting with $1$ and adding $0$ to the previous
value if a node is connected to the previous one and adding $1$ otherwise.
\end{prop}

See \cref{f:bijMarkedPivots} for a example of the bijection.

\begin{figure}
\centering
\begin{forest}
for tree={grow=south, circle, draw, minimum size=3ex, inner sep=0, s sep=1mm, l=0}
[6 [1
      [,no edge, draw=none]
      [,no edge, draw=none]
      [6 [4 [2 [,no edge, draw=none][,no edge, draw=none][3]]
            [,no edge, draw=none]
            [5 [,no edge, draw=none][,no edge, draw=none][6]]
         ]
         [,no edge, draw=none]
         [,no edge, draw=none]
      ]
   ]
   [9 [7
         [,no edge, draw=none]
         [,no edge, draw=none]
         [9[8][,no edge, draw=none][,no edge, draw=none]]
      ]
      [,no edge, draw=none]
      [,no edge, draw=none]
   ]
]
\end{forest}
\qquad
\qquad
\qquad
\def\taille{0.9}
\def\th{0.6}
\begin{tikzpicture}
\begin{scope}[every node/.style={circle,scale=1,fill=white,draw}]
   \node (A) at (0,0) {};
   \node (B)[color=black] at (-2*\th,-1*\taille) {};
   \node (C) at (-1*\th,-2*\taille) {};
   \node (D)[color=black] at (-2*\th,-3*\taille) {};
   \node (E) at (-3*\th,-4*\taille) {};
   \node (F) at (-2*\th,-5*\taille) {};
   \node (G) at (-1*\th,-4*\taille) {};
   \node (H) at (0*\th,-5*\taille) {};
   \node (I) at (2*\th,-1*\taille) {};
   \node (J)[color=black] at (1*\th,-2*\taille) {};
   \node (K) at (2*\th,-3*\taille) {};
   \node (L) at (1*\th,-4*\taille) {};
\end{scope}
\begin{scope}[>={Stealth[black]}, every edge/.style={draw=black,thick}]
    \path [-] (A) edge (B); \path [-] (B) edge (C);
    \path [-] (C) edge (D);
    \path [-] (D) edge (E);
    \path [-] (E) edge (F);
    \path [-] (D) edge (G);
    \path [-] (G) edge (H);
    \path [-] (A) edge (I);
    \path [-] (I) edge (J);
    \path [-] (J) edge (K);
    \path [-] (K) edge (L);
\end{scope}
\end{tikzpicture}
\caption{\label{f:bijMarkedPivots}We have, on the left, an initial tree and, on the right, its associated tree with marked pivots. For the three marked pivots, on the initial tree, their parents and the last node of their right branch have the same labels and, the remaining pivot does not.}
\end{figure}
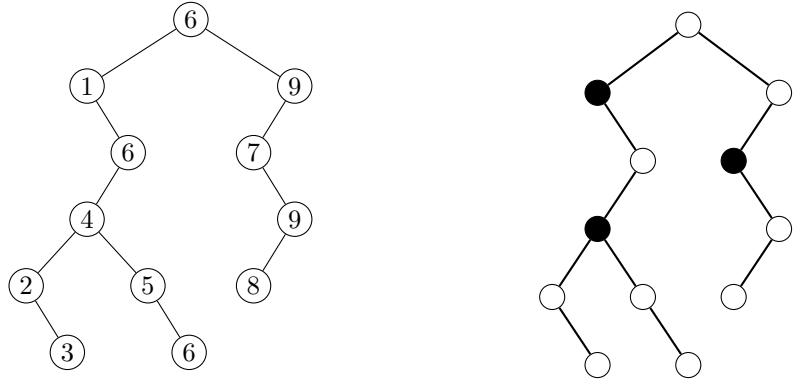

\begin{proof}
Starting with an IT tree, one gets a tree with marked pivots.
Now, the number of marked pivots is exactly the number of values that
have an equal value below in the tree.

Indeed, all nodes with the same value are all together is a path from a leaf
to the root and between two consecutive ones we marked exactly one pivot.
Moreover, any element in a tree has only one parent so a pivot cannot be
marked for two different values.

Now, an initial binary search tree is well-defined by its shape and the
sequence of the values of its nodes in the left to right infix order. Being
initial implies that each value is either the same as the previous one or that
it is one more. And the marked nodes show which values have to be equal, so
that, when applying the reverse algorithm, one directly gets the original
tree. Both algorihthms are injective (and inverses of each other), hence the
bijection.
\end{proof}

So we transformed our IT trees bijectively into trees with marked pivots,
which is quite close in spirit to trees with marked leaves.
The missing link is a standard involution on trees :
the \emph{recursive reversal of the left branch subtrees} denoted by $\psi$.
It is defined in the following way.

\begin{defi}[Recursive reversal of the left branch subtrees]
Let $T$ be a binary tree.
\begin{enumerate}[(1)]
\item Let $k$ be the number of nodes on the left branch of $T$ starting at the
root.  Running down this branch, exchange the right subtrees of the $i^{th}$
and $k\!+\!1\!-\!i^{th}$ nodes of the left branch.
\item Apply the same algorithm to each of these right subtrees.
\end{enumerate}
The tree thus obtained is the recursive reversal of the left branch subtrees
of the original one.
\end{defi}
See \cref{reversal} for an example.

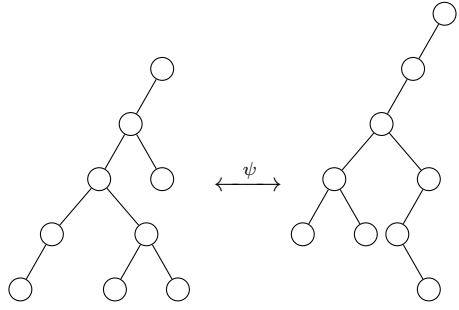
\begin{figure}[ht]
	\centering
	\begin{forest}
		for tree={grow=south, circle, draw, minimum size=2ex, inner sep=0, s sep=1mm, l=0}
		[
		[
		[
		[
		[]
		[,no edge, draw=none]
		[,no edge, draw=none]
		]
		[
		[]
		[,no edge, draw=none]
		[]
		]
		]
		[,no edge, draw=none]
		[]
		]
		[,no edge, draw=none]
		[,no edge, draw=none]
		]
	\end{forest}
	\begin{forest}
		for tree={grow=south, minimum size=2ex, inner sep=0, s sep=1mm, l=0}
		[,no edge[,no edge[$\xleftrightarrow{\;\;\;\psi\;\;\;}$,no edge[,no edge[,no edge]]]]]
	\end{forest}
	\begin{forest}
		for tree={grow=south, circle, draw, minimum size=2ex, inner sep=0, s sep=1mm, l=0}
		[
		[
		[
		[
		[]
		[,no edge, draw=none]
		[]
		]
		[,no edge, draw=none]
		[
		[
		[,no edge, draw=none]
		[,no edge, draw=none]
		[]
		]
		[,no edge, draw=none]
		[,no edge, draw=none]
		]
		]
		[,no edge, draw=none]
		[,no edge, draw=none]
		]
		[,no edge, draw=none]
		[,no edge, draw=none]
		]
\end{forest}
\caption{\label{reversal}The recursive reversal of the left branch subtrees
applied to a binary tree.}
\end{figure}

\begin{prop}
The map $\psi$ induces a bijection between binary trees with $n$ nodes having
$k$ pivots and binary trees with $n$ nodes having $k+1$ leaves.
\end{prop}

\begin{proof}
When applying Step $(i)$ of $\psi$, one reroots a right subtree of a node that
is a left child onto a node that has a left child. And, at the same time, one
reroots a right subtree of a node that has a left child onto a node that is a
left child. This is due to the fact that every node of a left branch is at the
same time a left child and has a left child, except the root and the last one
that respectively has a left child without being a left child and is a left
child without having a left child.

Therefore, there is a bijection between the number of pivots of $T$ and the
number of nodes having 2 children in $\psi(T)$. The number of leaves being 1
plus the number of nodes that have 2 children, this conclude the proof.
\end{proof}

Using both bijections one after the other, we directly conclude that
\begin{cor}
There are $B_{n\!-\!k\!-\!1,k}$ IT trees with $n$ nodes and $n-k$ different
labels.
\end{cor}

Note that the Catalan numbers appear twice in the Borel triangle:
they obviously are the number of IT trees with no pivots since these are
the trees indexed by permutations, hence in bijection with the trees
themselves. And, more surprisingly, they are also the number of IT trees
of size $2n-1$ over an alphabet of cardinality $n$ hence with $n-1$ marked
pivots. Indeed, once applied the involution reverting the left branches, one
gets binary trees with $n$ leaves. But binary trees with $2n-1$ nodes and
$n$ leaves have $n-1$ internal nodes and hence are complete binary trees,
hence the second appearance of the Catalan numbers.


Let us now consider an IT tree $T$ with labels on an alphabet of size $n$.
This tree cannot have any size since repeated letters require pivots between
them. If $T$ has more than $2n$ nodes, it has to have at least $n$ pivots
since it is IT. But the root and every right child of a pivot cannot be a
pivot. Therefore, at least $n+1$ nodes cannot be pivots and we got a
contradiction. So $T$ has at most $2n-1$ nodes.


\begin{thm}
\label{thm:2redTree}
Let $\A$ be a finite alphabet alphabet. Then, the number of 2-reduced binary
search trees labelled by $\A$ is
\begin{equation}
|\Sylv(\A,2)| = 
1 + \sum_{n=1}^{2|\A|-1} \sum_{k=0}^{\lfloor n/ 2\rfloor}
   \binom{|\A|}{n-k} B_{n-k-1,k}.
\end{equation}
\end{thm}

\begin{proof}
The $1$ at the beginning represents the empty tree. The variable $n$ encodes
the size of the tree and $k$ its number of pivots and the bounds in the sum
are the ones we computed before.
We know that there are $B_{n-k-1,k}$ such IT trees, so that the total number
of not necessarily initial trees is the same multiplied by the number of ways
to pick the right number of different values (which is $n-k$) among the total
size of the alphabet, hence the formula.
\end{proof}

Let us check this formula with $\A$ of size $3$ and split the enumeration
depending on the number of nodes of the tree. We have
\begin{itemize}[$\bullet$]
\item $n=0$: one empty tree,
\item $n=1$: one tree of size $1$ with 3 possible labellings,
\item $n=2$: all trees of size 2 with 3 possible labellings,
\item $n=3$: all five trees of size $3$ with one labelling and one such tree
with its pivot marked and three labellings,
\item $n=4$: six trees of size $4$ with one pivot marked and one labelling,
\item $n=5$: two trees of size $5$ with two pivots marked and one labelling.
\end{itemize}
Therefore, one gets that there are $26$ reduced trees in $\Sylv(\{1,2,3\},2)$.

See \cref{f:pivottrees} for all the reduced trees that can be labelled using
$\{1,2,3\}$: on the first row, we have all trees with no pivots; on the second
row, we have all trees with at least 1 marked pivot (in black).

\begin{figure}
\def\taille{.5}
\centering
\begin{tikzpicture}
\node (A) at (0,0) {$\emptyset$};
\end{tikzpicture},
\begin{tikzpicture}
\begin{scope}[every node/.style={circle,scale=.5,fill=white,draw}]
   \node (A) at (0,0) {};
\end{scope}
\end{tikzpicture},
\begin{tikzpicture}
\begin{scope}[every node/.style={circle,scale=.5,fill=white,draw}]
   \node (A) at (0,0) {}; \node (B) at (-1*\taille,-1*\taille) {};
\end{scope}
\begin{scope}[>={Stealth[black]}, every edge/.style={draw=black,thick}]
    \path [-] (A) edge (B);
\end{scope}
\end{tikzpicture},
\begin{tikzpicture}
\begin{scope}[every node/.style={circle,scale=.5,fill=white,draw}]
   \node (A) at (0,0) {}; \node (B) at (1*\taille,-1*\taille) {};
\end{scope}
\begin{scope}[>={Stealth[black]}, every edge/.style={draw=black,thick}]
    \path [-] (A) edge (B);
\end{scope}
\end{tikzpicture},
\begin{tikzpicture}
\begin{scope}[every node/.style={circle,scale=.5,fill=white,draw}]
   \node (A) at (0,0) {}; \node (B) at (-1*\taille,-1*\taille) {};
    \node (C) at (-2*\taille,-2*\taille) {};
\end{scope}
\begin{scope}[>={Stealth[black]}, every edge/.style={draw=black,thick}]
    \path [-] (A) edge (B); \path [-] (B) edge (C);
\end{scope}
\end{tikzpicture},
\begin{tikzpicture}
\begin{scope}[every node/.style={circle,scale=.5,fill=white,draw}]
   \node (A) at (0,0) {}; \node (B) at (-1*\taille,-1*\taille) {};
    \node (C) at (0*\taille,-2*\taille) {};
\end{scope}
\begin{scope}[>={Stealth[black]}, every edge/.style={draw=black,thick}]
    \path [-] (A) edge (B); \path [-] (B) edge (C);
\end{scope}
\end{tikzpicture},
\begin{tikzpicture}
\begin{scope}[every node/.style={circle,scale=.5,fill=white,draw}]
   \node (A) at (0,0) {}; \node (B) at (-1*\taille,-1*\taille) {};
    \node (C) at (1*\taille,-1*\taille) {};
\end{scope}
\begin{scope}[>={Stealth[black]}, every edge/.style={draw=black,thick}]
    \path [-] (A) edge (B); \path [-] (A) edge (C);
\end{scope}
\end{tikzpicture},
\begin{tikzpicture}
\begin{scope}[every node/.style={circle,scale=.5,fill=white,draw}]
   \node (A) at (0,0) {}; \node (B) at (1*\taille,-1*\taille) {};
    \node (C) at (0*\taille,-2*\taille) {};
\end{scope}
\begin{scope}[>={Stealth[black]}, every edge/.style={draw=black,thick}]
    \path [-] (A) edge (B); \path [-] (B) edge (C);
\end{scope}
\end{tikzpicture},
\begin{tikzpicture}
\begin{scope}[every node/.style={circle,scale=.5,fill=white,draw}]
   \node (A) at (0,0) {}; \node (B) at (1*\taille,-1*\taille) {};
    \node (C) at (2*\taille,-2*\taille) {};
\end{scope}
\begin{scope}[>={Stealth[black]}, every edge/.style={draw=black,thick}]
    \path [-] (A) edge (B); \path [-] (B) edge (C);
\end{scope}
\end{tikzpicture},

\begin{tikzpicture}
\begin{scope}[every node/.style={circle,scale=.5,fill=white,draw}]
   \node (A) at (0,0) {}; \node (B)[color=black] at (-1*\taille,-1*\taille) {};
    \node (C) at (0*\taille,-2*\taille) {};
\end{scope}
\begin{scope}[>={Stealth[black]}, every edge/.style={draw=black,thick}]
    \path [-] (A) edge (B); \path [-] (B) edge (C);
\end{scope}
\end{tikzpicture},
%
\begin{tikzpicture}
\begin{scope}[every node/.style={circle,scale=.5,fill=white,draw}]
   \node (A) at (0,0) {}; \node (B)[color=black] at (-1*\taille,-1*\taille) {};
    \node (C) at (0*\taille,-2*\taille) {};
    \node (D) at (1*\taille,-3*\taille) {};
\end{scope}
\begin{scope}[>={Stealth[black]}, every edge/.style={draw=black,thick}]
    \path [-] (A) edge (B); \path [-] (B) edge (C);
    \path [-] (C) edge (D);
\end{scope}
\end{tikzpicture},
%
\begin{tikzpicture}
\begin{scope}[every node/.style={circle,scale=.5,fill=white,draw}]
   \node (A) at (0,0) {}; \node (B)[color=black] at (-1*\taille,-1*\taille) {};
    \node (C) at (0*\taille,-2*\taille) {};
    \node (D) at (-1*\taille,-3*\taille) {};
\end{scope}
\begin{scope}[>={Stealth[black]}, every edge/.style={draw=black,thick}]
    \path [-] (A) edge (B); \path [-] (B) edge (C);
    \path [-] (C) edge (D);
\end{scope}
\end{tikzpicture},
%
\begin{tikzpicture}
\begin{scope}[every node/.style={circle,scale=.5,fill=white,draw}]
   \node (A) at (0,0) {}; \node (B)[color=black] at (-1*\taille,-1*\taille) {};
    \node (C) at (0*\taille,-2*\taille) {};
    \node (D) at (-2*\taille,-2*\taille) {};
\end{scope}
\begin{scope}[>={Stealth[black]}, every edge/.style={draw=black,thick}]
    \path [-] (A) edge (B); \path [-] (B) edge (C);
    \path [-] (B) edge (D);
\end{scope}
\end{tikzpicture},
%
\begin{tikzpicture}
\begin{scope}[every node/.style={circle,scale=.5,fill=white,draw}]
   \node (A) at (0,0) {}; \node (B)[color=black] at (-1*\taille,-1*\taille) {};
    \node (C) at (0*\taille,-2*\taille) {};
    \node (D) at (1*\taille,-1*\taille) {};
\end{scope}
\begin{scope}[>={Stealth[black]}, every edge/.style={draw=black,thick}]
    \path [-] (A) edge (B); \path [-] (B) edge (C);
    \path [-] (A) edge (D);
\end{scope}
\end{tikzpicture},
%
\begin{tikzpicture}
\begin{scope}[every node/.style={circle,scale=.5,fill=white,draw}]
   \node (A) at (0,0) {}; \node (B)[color=black] at (-1*\taille,-1*\taille) {};
    \node (C) at (0*\taille,-2*\taille) {};
    \node (D) at (1*\taille,1*\taille) {};
\end{scope}
\begin{scope}[>={Stealth[black]}, every edge/.style={draw=black,thick}]
    \path [-] (A) edge (B); \path [-] (B) edge (C);
    \path [-] (A) edge (D);
\end{scope}
\end{tikzpicture},
%
\begin{tikzpicture}
\begin{scope}[every node/.style={circle,scale=.5,fill=white,draw}]
   \node (A) at (0,0) {}; \node (B)[color=black] at (-1*\taille,-1*\taille) {};
    \node (C) at (0*\taille,-2*\taille) {};
    \node (D) at (-1*\taille,1*\taille) {};
\end{scope}
\begin{scope}[>={Stealth[black]}, every edge/.style={draw=black,thick}]
    \path [-] (A) edge (B); \path [-] (B) edge (C);
    \path [-] (A) edge (D);
\end{scope}
\end{tikzpicture},
%
\begin{tikzpicture}
\begin{scope}[every node/.style={circle,scale=.5,fill=white,draw}]
   \node (A) at (0,0) {}; \node (B)[color=black] at (-1*\taille,-1*\taille) {};
    \node (C) at (0*\taille,-2*\taille) {};
    \node (D)[color=black] at (-1*\taille,-3*\taille) {};
    \node (E) at (0*\taille,-4*\taille) {};
\end{scope}
\begin{scope}[>={Stealth[black]}, every edge/.style={draw=black,thick}]
    \path [-] (A) edge (B); \path [-] (B) edge (C);
    \path [-] (C) edge (D); \path [-] (D) edge (E);
\end{scope}
\end{tikzpicture},
%
\begin{tikzpicture}
\begin{scope}[every node/.style={circle,scale=.5,fill=white,draw}]
   \node (A) at (0,0) {}; \node (B)[color=black] at (-1*\taille,-1*\taille) {};
    \node (C) at (0*\taille,-2*\taille) {};
    \node (D)[color=black] at (-2*\taille,-2*\taille) {};
    \node (E) at (-1*\taille,-3*\taille) {};
\end{scope}
\begin{scope}[>={Stealth[black]}, every edge/.style={draw=black,thick}]
    \path [-] (A) edge (B); \path [-] (B) edge (C);
    \path [-] (B) edge (D); \path [-] (D) edge (E);
\end{scope}
\end{tikzpicture}.
\caption{The shapes of all 2-reduced trees on an alphabet with 3 letters. On
the first row are all trees with no pivots starting with the empty tree;
on the second row, we have all trees with at least 1 pivot.}
\label{f:pivottrees}
\end{figure}
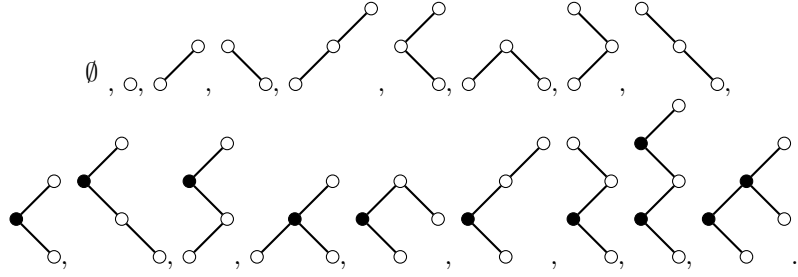
Among those trees, there are $13$ of those that use at least once each value
in $\A$, which is the sum of the row indexed by $2$ in the Borel triangle: the
five trees of size $3$, the six trees of size $4$ with a pivot
(see~\cref{ex-sylv-6-43}) and the two trees of size $5$ with two pivots.

\paragraph{Idempotents\\}\label{p-sylv-idem}

Let us now characterize the idempotent trees among the $2$-reduced trees.
Since the insertion can either leave the tree the same or strictly increase
its size, in order to be idempotent, we need to have that the insertion of any
letter already present in the tree does not change it. This only happens when
the lowest occurrence of that letter is on a node $S$ with an empty left
subtree since a new insertion of that letter would follow exactly the same
path up to $S$ and stop at that exact spot.

\begin{prop}
A tree is idempotent iff the lowest occurrence of every label appearing in $T$
is on a node with an empty left subtree.
\end{prop}

For example, if the alphabet has three letters and we look for the idempotent
trees using all three letters, there are six such trees; see \cref{ex-sylv-idemps}.
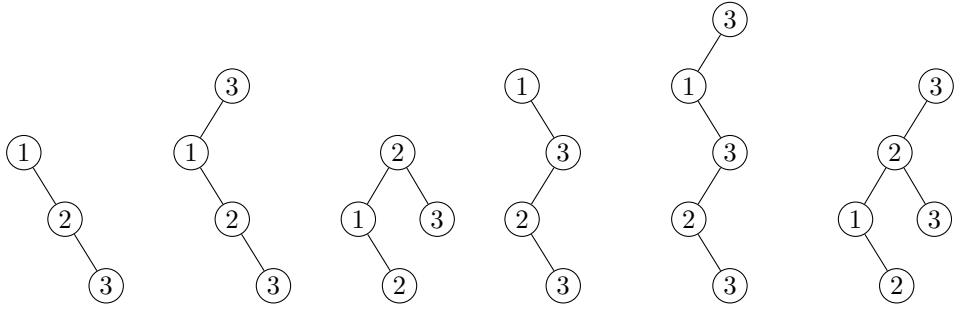
\begin{figure}
\centering
\begin{forest}
for tree={grow=south, circle, draw, minimum size=3ex, inner sep=0, s sep=1mm, l=0}
[1
      [,no edge, draw=none]
      [,no edge, draw=none]
      [2 [,no edge, draw=none] [,no edge, draw=none][3]]
]
\end{forest}
\begin{forest}
for tree={grow=south, circle, draw, minimum size=3ex, inner sep=0, s sep=1mm, l=0}
[3 [1
      [,no edge, draw=none]
      [,no edge, draw=none]
      [2 [,no edge, draw=none] [,no edge, draw=none][3]]
   ]
      [,no edge, draw=none]
      [,no edge, draw=none]
]
\end{forest}
\begin{forest}
for tree={grow=south, circle, draw, minimum size=3ex, inner sep=0, s sep=1mm, l=0}
[2 [1
      [,no edge, draw=none]
      [,no edge, draw=none][2]]
   [,no edge, draw=none]
   [3]
]
\end{forest}
\begin{forest}
for tree={grow=south, circle, draw, minimum size=3ex, inner sep=0, s sep=1mm, l=0}
[1 [,no edge, draw=none] [,no edge, draw=none]
   [3 
      [2 [,no edge, draw=none] [,no edge, draw=none][3]]
      [,no edge, draw=none] [,no edge, draw=none]
   ]
]
\end{forest}
\begin{forest}
for tree={grow=south, circle, draw, minimum size=3ex, inner sep=0, s sep=1mm, l=0}
[3 [1 [,no edge, draw=none] [,no edge, draw=none]
   [3 
      [2 [,no edge, draw=none] [,no edge, draw=none][3]]
      [,no edge, draw=none] [,no edge, draw=none]
   ]
]
      [,no edge, draw=none] [,no edge, draw=none]
]
\end{forest}
\begin{forest}
for tree={grow=south, circle, draw, minimum size=3ex, inner sep=0, s sep=1mm, l=0}
[3 [2
      [1 [,no edge, draw=none][,no edge, draw=none][2]] [,no edge, draw=none] [3]
   ]
      [,no edge, draw=none]
      [,no edge, draw=none]
]
\end{forest}
\caption{
\label{ex-sylv-idemps}The six idempotent trees using exactly 3 letters.}
\end{figure}

\begin{prop}\label{2sylv-idem}
Let $\A$ be a finite alphabet alphabet of size $n$.
Then the monoid $\Sylv(\A,2)$ has  
\begin{equation}
1 + \sum_{k=1}^n \binom{n}{k} S_k
\end{equation}
idempotents, where $S_n$ is the sequence of large Schr\"oder numbers
(see Sequence~A006318 of~\cite{oeis}) shifted by one:
$S_0=S_1=1$, $S_2=2$, $S_3=6$, \emph{etc}.
\end{prop}

\begin{proof}
Let us prove that there are $S_n$ idempotents over $n$ letters using
all those letters, the formula being a direct consequence of that property.
We shall use the alphabet $\A=\{1,\dots,n\}$.

Let us now separate those trees according to the value of their root and
define $I_{n,k}$ as the number of idempotent trees with $n$ different letters
and root $k$.
Then
\begin{equation}
I_{n,k} =
\begin{cases}
\displaystyle\sum_{m=1}^{n-1} I_{n-1,m}   &\text{if $k=1$,} \\
\displaystyle\sum_{\ell=1}^{n-1} I_{n,\ell}   &\text{if $k=n$,} \\
\displaystyle\left(\sum_{\ell=1}^{k-1} I_{k,\ell}\right)\left(\sum_{m=1}^{n-k} I_{n-k,m}\right) &\text{otherwise.} \\
\end{cases}
\end{equation}

Indeed, if the root is $1$, it has no left subtree and its right subtree
is an idempotent that contains all values from $2$ to $n$, hence the first
formula since the root can take any value. If the root is $n$, it has no right
subtree and its left subtree is an idempotent that contains all values from
$1$ to $n$ (value $n$ has to be repeated since the root has a left subtree)
and the root can take any value but $n$ since it is $2$-reduced.
The general case (third formula above) follows from those two cases since the
left and right subtrees of the root are independent.
The second formula above shows that $I_{n,n}=\sum_{\ell=1}^{n-1} I_{n,\ell}$
so that $I_{n,n}=S_n/2$ except when $n=1$ where $I_{1,1}=S_1=1$.
Let us denote by $s_n$ the sequence $s_1=1$ and $s_k=S_k/{2}$ for $k>1$.
Then $I_{n,k}=s_k S_{n-k}$.

Summing all previous equations with a given $n$ yields
\begin{equation}
S_{n} 
 = \sum_{k=1}^n s_k S_{n-k},
\end{equation}
which uniquely defines $S_n$ by induction with $S_0=S_1=1$ and $s_k$ defined
from $S_k$. It is a well-known formula on the large Schr\"oder numbers, see
Sequence A144156 of~\cite{oeis} that represents the triangle of the $I_{n,k}$,
hence the result. Note that the $s_k$ are known in the literature as the
little Schr\"oder numbers.
\end{proof}

\subsection{The \texorpdfstring{$\sigma$}{s}-Sylvester Monoids}
\label{ss-sylv-enum-iso}

If one lets $\sigma(x)=p$ for all $x\in\A$, one can compute the number of
elements in $\Sylv(\A,p)$ using \cref{thm:2redTree}.
Because $\Sylv$ is not of type 1, the formula for general $\sigma$ becomes
significantly more complicated.

\begin{cor}
\label{sylv-enum}
If $\A$ is a finite alphabet,
we have
\begin{equation}
|\Sylv(\A,p)| = 
1 + \sum_{n=1}^{2|\A|-1} (p-1)^n \sum_{k=0}^{\lfloor n/ 2\rfloor}
   \binom{|\A|}{n-k} B_{n-k-1,k} .
\end{equation}
\end{cor}

One then easily computes \cref{table-sigma-sylvester-sizes}.

\begin{table}[ht]
\begin{equation*}
\begin{array}{lrrrrrr}
    \toprule
    \multirow{2}{*}{$p$} & \multicolumn{6}{c@{}}{|\A|}\\
        \cmidrule(lr){2-7}
          &  2 & 3 & 4 & 5 & 6 & 7 \\
        \midrule
        2 & 6   & 26         & 142           & 882             & 5\,910             & 41\,610 \\
        3 & 21  & 255        & 3\,945        & 69\,339         & 1\,286\,589        & 25\,239\,063    \\
        4 & 52  & 1\,252     & 38\,164       & 1\,306\,564     & 47\,967\,220       & 1\,845\,549\,220     \\
        5 & 105 & 4\,205     & 212\,305      & 12\,024\,405    & 730\,040\,505      & 46\,444\,060\,605\\ 
        6 & 186 & 11\,166    & 842\,946      & 71\,345\,526    & 6\,472\,008\,906\\
        7 & 301 & 25\,291    & 2\,668\,057   & 315\,467\,887\\
        8 & 456 & 51\,080    & 7\,176\,520\\
        9 & 657 & 94\,617 \\
    \midrule
         & \text{\href{https://oeis.org/A069778}{A069778}} \\
    \bottomrule
\end{array}
\end{equation*}
\caption{The sizes of $\Sylv(\A, p)$ for some small values of $|\A|$ and small
values of $p\in \NN$.}
\label{table-sigma-sylvester-sizes}
\label{table-2-sylvester-size}
\end{table}

As a corollary of \cref{thm:isom}.(ii), one has

\begin{cor}\label{sylv-iso}
Let $\sigma, \tau:\A \to \NN_{\geq2}$.
Then $\Sylv(\A,\sigma)\cong\Sylv(\A,\tau)$ if and only if $\sigma=\tau$.
\end{cor}

\section{The \texorpdfstring{$\sigma$}{s}-Hyposylvester Monoids}
\label{section-hyposylv}


\subsection{The Hyposylvester Monoid}\label{ss-hyposylv}

The \emph{hyposylvester congruence} $\equiv_{\HS}$ is the least congruence
containing the relations
\begin{equation}\label{relhs1}
U\,ac\,V = U\,ca\,V
\end{equation}
if there exists $b\in \cont(UV)$ such that $a<b<c$, and
\begin{equation}\label{relhs2}
ab\,U\,a = ba\,U\,a,
\end{equation}
for $a<b$, where $U,V$ are arbitrary words.

Recently defined in \cite{NT}, the \emph{hyposylvester monoid} $\HS(\A)$
is defined as a quotient of the free monoid $\A^*$ by $\equiv_{\HS}$.

As in \cref{def:quotient-monoid}, for $\sigma: \A \to \NN_{\geq 2}$,
we define the \emph{$\sigma$-hyposylvester congruence} to be the least
congruence on $\A^*$ containing the hyposylvester congruence and
$a^{\sigma(a)} = a$ for every $a\in \A$; we denote this by
$\equiv_{\HS(\sigma)}$. The $\sigma$-hyposylvester monoid on $\A$ is the
quotient $\A ^ * / {\equiv_{\HS(\sigma)}}$ and is denoted by $\HS(\A, \sigma)$.

\paragraph{An Invariant of the Hyposylvester Monoid\\}

Let $W$ be a word.
For $x\in \cont(W)$ different from the maximal one, define
$x^+=\min\set{z\in W}{z>x}$. Define similarly $x^-=\max\set{z\in W}{z<x}$. Let
$\ell_x(W)$ be the number of occurrences of the letter $x^+$ to the left of
the rightmost $x$ in $W$ and let $r_x(W)$ be the number of occurrences of the
letter $x^+$ to the right of the rightmost $x$ in $W$.

By Theorem 3.4 of \cite{NT}, two words $U$ and $V$ are
hyposylvester-equivalent if and only if $\ev(U)=\ev(V)$ and
$\ell_x(U)=\ell_x(V)$ for all $x\in \cont(U)$.
The following definition gives an invariant on hyposylvester classes in form of a tree.

\begin{defi}[Hyposylvester tree building]
Let $W\in\A^*$ with $\cont(W)=\{w_1<w_2<\cdots<w_n\}$.
The \emph{hyposylvester tree of} $W$, denoted $\H(W)$, is obtained by the following procedure:
\begin{enumerate}[(1)]
\item Put a node labelled by $w_n$ with $|W|_{w_n}+1$ leaves. Let $i=n-1$.
\item Label by $w_i$ the $\ell_{w_i}(W)\!+\!1$-st child of $w_{i+1}$.
\item Add $|W|_{w_i}\!+\!1$ children to that child. Go back to step 2 with
$i=i-1$ if $i\not=1$ or stop if $i=1$.
\end{enumerate}
\end{defi}

We will use the same notation $\ell_{w_{i}^-}(T)$ and $r_{w_{i}^-}(T)$ for the
number of children of the node $w_i$ that are, respectively, to the left and
to the right of the node $w_{i-1}$.

From such a tree $T$, one can build the canonical word of its class by
labeling each sector between two consecutive edges of a given node with the
label of this node and read the sectors from left to right. This canonical
word is the maximal lexicographic element of its class and is denoted by
$\RW(T)$.
See \cref{43124} for an example of both a hyposylvester tree and how to build
its canonical word.

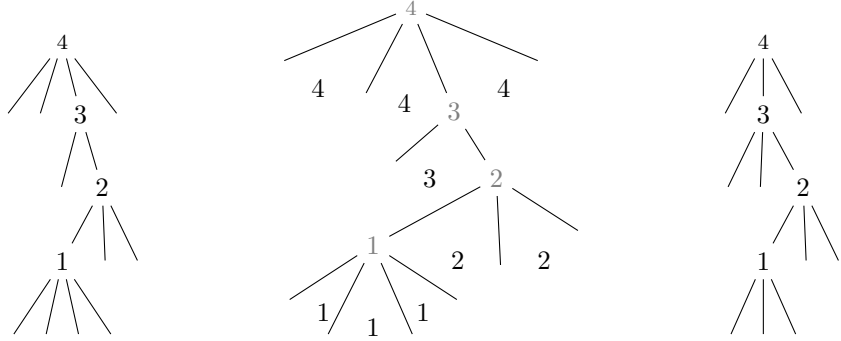
\begin{figure}[ht]
\centering
\forestset{
default preamble={
font=\footnotesize,
}
}
\begin{forest}
[4[][][3[][2[1[][][][]][][]]][]]
\end{forest}
\qquad\qquad
\begin{forest}
[4,color=gray
	[,l=2 mm, anchor=north east]
	[4,no edge,l=4 mm, anchor=north east]
	[,l=7 mm, anchor=north east]
	[4,no edge, l=6 mm, anchor=north east]
	[3,color=gray, l=7 mm, anchor=north west
		[,l=2 mm, anchor=north east]
		[3,no edge,l=2mm, anchor=north]
		[2,color=gray, l=2 mm, anchor=north west
			[1,color=gray, l=2mm, anchor=north
				[,l=2mm]
				[1,l=4mm, no edge]
				[,l=7mm]
				[1,l=6mm, no edge]
				[,l=7mm]
				[1,l=4mm, no edge]
				[,l=2mm]
			]
			[2,no edge,l=4 mm, anchor=north]
			[,l=7 mm, anchor=north]
			[2,no edge,l=4 mm, anchor=north]
			[,l=2 mm, anchor=north]
		]
	]
	[4,no edge,l=4mm, anchor=north west]
	[,l=2mm, anchor=north west]
]
\end{forest}
\qquad\qquad
\begin{forest}
[4[][3[][][2[1[][][]][][]]][]]
\end{forest}
\caption{\label{43124}On the left, the hyposylvester tree $T$ associated to
the word $141124324$. In the middle, the algorithm to obtain
$\RW(T)=443111224$. And on the right, the reduced hyposylvester tree of $T$,
its 2-reduced canonical word being $43124$.}
\end{figure}

\subsection{Quotient Type}\label{ss-hyposylv-type}

For $|\A|\geq 2$, $\HS(\A)$ is not a type 1 monoid. This is for the same
reason as it is with $\Sylv$: for $a<b$, using \cref{ss:Type1Monoids}
notation, one has $\theta(bbab)=\theta(babb)=bab$ and
$\ev_\sigma(bbab)=\ev_\sigma(babb)=abbb$ but they are not
$\sigma$-hyposylvester equivalent if $\sigma(b)> 2$.
We shall see that it is of type 2.

The following lemma comes directly from Theorem 3.4 of~\cite{NT} and the
definition of $\H(W)$.

\begin{lemma}
\label{HypoSylvWord}
Let $T$ be a hyposylvester tree.
The $\ell_{x^-}(T)$ first occurrences of $x$ are consecutive in $\RW(T)$ and
the same goes for the $r_{x^-}(T)$ last occurrences of $x$.

Moreover, there will always be an occurrence of $x^-$ in between the
$\ell_x(W)$-th occurrence and the $\ell_x(W)\!+\!1$-st occurrence of $x$ in
$W$.
\end{lemma}

\begin{prop}
\label{cor-hyposylv-gathered}
The hyposylvester monoid is a \gathered monoid.
\end{prop}

\begin{proof}
Thanks to the previous lemma, we know that, for a hyposylvester tree $T$,
$\RW(T)$ is a \gathered element of its class.
It is proven in~\cite{NT} that $\RW(T)$ is the maximal element of its
class since the class is an interval of the weak order and $\RW(T)$ avoids the
patterns $ab\cdots a$, $b\cdots ac$, and $ac\cdots b$.

In other words, going from $W$ to its canonical element amounts to
orienting the rewritings in the following way,
for any $U,V\in\A^*$
\newpage

\begin{enumerate}[(1)]
\item $U\,ac\,V \to U\,ca\,V$, \quad if there exists $b\in \cont(UV)$ such that $a<b<c$,
\item $ab\,U\,a \to ba\,U\,a$, \quad if $a<b$.
\end{enumerate}

So $\RW(T)\in\G_{\HS}$, i.e. it is the canonical \gathered word of its class
as defined in \cref{ss:GatheredMonoids}.

Now, let $W, W'\in\A^*$ such that $W'$ is obtained by expanding a letter $x$
in a word $W$ and see what are their respective canonical elements.
Let us follow the path of elementary rewritings from $W$ to $G(W)$, its
canonical \gathered element.
If a rewriting does not involve that $x$, then apply the same on $W'$.
If it involves $x$, then apply it to both $x$ in $W'$ so that they end up next
to each other again. In all cases, the elementary rewritings commute with the
expansion of $x$, so that $G(W')$ is the expansion of the same $x$ of $G(W)$.

The proof of the contraction is the same and works since there are no
rewritings requiring two equal consecutive letters, so that all rewritings
implicating both $x$ in $W$ can be mimicked with only one of these.
\end{proof}

Following the notation of \cref{ss:GatheredMonoids}, denote by $G(W)$ the
canonical word of the hyposylvester class of $W$. Following \cref{def:type2} and letting $m=\min(\cont(W))$, $G(W)$ is \emph{$\sigma$-reduced} if
\begin{enumerate}[(1)]
\item for $a\in \cont(W)$,
      $\ell_{a}(W)<\sigma(a)$ and $r_{a}(W)<\sigma(a)$,
\item $|W|_{m}<\sigma(m)$.
\end{enumerate}

\begin{thm}
\label{hypotree-w}
Let $\A$ be an alphabet. The set of $\sigma$-reduced canonical words is a
section of $\HS(\A,\sigma)$.
\end{thm}

\begin{proof}
The proof is straightforward using the fact that the hyposylvester monoid is
\gathered.
\end{proof}

As defined in \cref{ss:GatheredMonoids}, denote by $G_\sigma(W)$ the
$\sigma$-reduced word of $G(W)\in\G_{\HS}$ and denote by $\G_{\HS(\sigma)}$
the set of $\sigma$-reduced \gathered words.

\subsection{The 2-Hyposylvester Monoid}

Let us now study the 2-quotient of the hyposylvester monoid.
\paragraph{Reduced Hyposylvester trees\\}
\label{hyposylv-reductree}
A hyposylvester tree $T$ is called \emph{2-reduced} if its internal nodes have
exactly 3 children.
With a similar algorithm as for hyposylvester trees, one can build a 2-reduced
hyposylvester tree from a word $W\in\A^*$.

\begin{defi}[Construction of the reduced hyposylvester tree]
\label{def:redHStree}
Let $W\in\A^*$ be a word with content \mbox{$\cont(W)=\{w_1<w_2<\cdots <w_n\}$}.
\begin{enumerate}[(1)]
\item Build a node labelled $w_n$ with three children. Let $i=n-1$.
\item Label by $w_i$ the left child, the right child or the
middle child of the node labelled $w_{i+1}$, if there are occurrences of
$w_{i+1}$ respectively only to the right, only to the left or on both sides of
the rightmost $w_i$ in $W$.
\item Add 3 children to that child. Stop if $i=1$. Otherwise, go back to step 2 with $i=i-1$.
\end{enumerate}
This tree is denoted by $\H_2(W)$.
\end{defi}

From a 2-reduced hyposylvester tree, one builds its 2-reduced \gathered word
the following way.
\newpage

\begin{defi}[2-reduced hyposylvester reading]
\label{def:redHSgathered}
Let $T$ be a reduced hyposylvester tree with labels $\{w_1<\cdots<w_n\}$. 
\begin{enumerate}[(1)]
\item Start with the word $W=w_n$ and $i=n-1$.
\item The node $w_i$ is one of the three children of $w_{i+1}$.
  \begin{enumerate}[(i)]
  \item If $w_i$ is the left child of $w_{i+1}$, insert $w_i$ to the left of
        $w_{i+1}$ in $W$,
  \item if $w_i$ is the right child of $w_{i+1}$, insert $w_i$ to the right of
        $w_{i+1}$ in $W$,
  \item if $w_i$ is the middle child of $w_{i+1}$, replace $w_{i+1}$ by the
        word $w_iw_{i+1}w_i$ in $W$.
  \end{enumerate}
\item  If $i=1$ stop, otherwise go back to step (ii) with $i=i-1$.
\end{enumerate}
This word is the 2-reduced \gathered word of $T$, denoted by $\RW_2(T)$.
\end{defi}

Note that $\H_{2}(\RW_2(T))=T$ and also that for $W$,
$\RW_2(\H_2(W))=G_{\HS(2)}(W)$.
See the right tree of \cref{43124} for an example of a 2-reduced hyposylvester
tree.
Note that any $2$-reduced \gathered word is always decreasing and increasing in
values.

\begin{thm}
\label{hypotree-t}
Let $\A$ be an alphabet. The set of reduced hyposylvester trees over $\A$ is a
family of representatives of $\HS(\A,2)$.
\end{thm}

\begin{proof}
The proof is straightforward using the fact that the hyposylvester monoid is
\gathered.
\end{proof}

Note that this monoid is not $\J$-trivial since $3212\cdot 1=321$
and $321\cdot 2=3212$.
It is well known that $\HS(\A)$ is a quotient of $\Sylv(\A)$, the same goes for $\HS(\A,2)$ and $\Sylv(\A,2)$.
Because the class of $\L$-trivial monoids is also a pseudovariety (see Chapter 3 Section 4 of \cite{P}) which is, by definition, closed under quotienting, $\HS(\A,2)$ is $\L$-trivial.


\paragraph{Enumeration\\}

The number of 2-hyposylvester classes comes naturally from
\cref{hypotree-t}.

\begin{prop}
Let $\A$ be a finite alphabet. We have
\begin{equation}
|\HS(\A,2)|=1+\sum_{k=1}^{|\A|} \binom{|\A|}{k} 3^{k-1}=\frac{4^{|\A|}+2}{3}.
\end{equation}
\end{prop}

This is Sequence~A047849 of~\cite{oeis}.

\begin{prop}
Given an alphabet $\A$ of cardinality $n$, let
\begin{equation}
\alpha_{n,k}=|\set{W\in\G_{\HS(2)}}{\cont(W)=\A,|W|=k}|,
\end{equation}
be the cardinality of the set of 2-reduced canonical \gathered
words of length $k$ with content $\A$.
One has
\begin{equation}
\alpha_{n+1,k+1}=\alpha_{n,k-1}+2\alpha_{n,k},
\end{equation}
with $\alpha_{1,1}=1$ and $\alpha_{1,0}=0$.
\end{prop}

\begin{proof}
This is a direct consequence of \cref{def:redHSgathered}.
\end{proof}

This is Sequence A038207
 of~\cite{oeis}
which enumerates the number of $k$-faces of a $n$-cube. Choosing a face of a
$n$-cube amounts to choosing, for each coordinate, either $\{0\}$, $\{1\}$ or
the whole interval $[0,1]$.  Thus, the bijection between the faces of a
$n$-cube and the 2-hyposylvester classes for $|\A|=n+1$ appears clearly thanks
to \cref{def:redHStree} by either labelling the left, the right, or the middle
child of the $i$th internal node.

\paragraph{Idempotents\\}

Given $\B\subseteq\A$, $\B=\{b_1<b_2<\cdots<b_j\}$ and
$\C\subseteq\B\setminus b_1$,
$\C=\{c_1<\cdots< c_i\}$, let $$e_{\B,\C}:=b_jb_{j-1}\cdots b_1c_1\cdots
c_i.$$

\begin{prop}
\label{prop:HS2idem}
The idempotents of $\HS(\A, 2)$ are precisely the $e_{\B,\C}$ for every 
$\B\subseteq \A$ and $\C\subseteq\B\setminus\min(\B)$.
In particular, there are
\begin{equation}
1+\sum_{k=1}^{|\A|}\binom{|\A|}{k}2^{k-1}
\end{equation}
idempotents in $\HS(\A, 2)$.
\end{prop}

\begin{proof}
First, all the $e_{\B,\C}$ are idempotents: let $W$ be such an element.
Then for any letter $x$ not maximal, both $\ell_x(WW)$ and $\ell_x(W)$ are
nonzero and $r_x(WW)>0$ iff $r_x(W)>0$. So both the reduced hyposylvester
trees of $W$ and $WW$ are the same (see~\cref{def:redHStree}).

Now, given a word $W$ and a letter $x\in\cont(W)$ not maximal, we have
$\ell_x(WW)>0$. Thus in order for $W$ to be an idempotent, we need
$\ell_x(W)>0$ for all $x\in\cont(W)$.
Since the $2$-reduced \gathered words are decreasing then increasing (see
\cref{def:redHSgathered}), the previous condition implies that all values of
$W$ belong to the decreasing part, hence are a $e_{\B,\C}$.
\end{proof}

\subsection{The \texorpdfstring{$\sigma$}{s}-Hyposylvester Monoids}
\label{ss-hyposyl-enum-idem-iso}

Because it is not of type 1 as the sylvester monoid, one cannot use
\cref{thm:Type1-Enum} to compute the cardinality of the $\sigma$-hyposylvester
monoids using the 2-hyposylvester monoid.
Fortunately, there is a relatively easy way to compute these cardinalities.

\begin{cor}
For $\B\subseteq\A$, let $V_{\B}=\{(v_1,\ell_2,r_2,\cdots,\ell_{n},r_n)\}$
with the following conditions:
\begin{itemize}
\item $1\leq v_1<\sigma(b_1)$,
\item $0\leq \ell_k,r_k <\sigma(b_k)$ for all $1<k\leq n$,
\item $0<\ell_k + r_k$ for all $1<k\leq n$.
\end{itemize}
There is a bijection between $V_\B$ and the set of $\sigma$-reduced word of
content $\B\subset\A$.
\end{cor}

\begin{proof}
Let $G_\sigma\in\G_{\HS(\sigma)}$ with
$\cont(G_\sigma)=\B=\{b_1<b_2<\cdots<b_n\}$.
Because the hyposylvester monoid is of type 2, $G_\sigma$ is an inflation of a $G_2\in\G_{\HS(2)}$.
Then $\ell_k=\ell_{b_k}(G_\sigma)$, $r_k=r_{b_k}(G_\sigma)$ for all
$2\leq k\leq n$ and $v_1=|G_\sigma|_{b_1}$ completely determine $G_\sigma$.
\end{proof}

\begin{cor}\label{hyposylv-enum}
Let $\A$ a finite alphabet and $\sigma : \A \to \NN$ a map such that
$\sigma(a) \geq 2$ for every $a\in \A$. We have
\begin{equation}
|\HS(\A,\sigma)|= 1 +
  \sum_{\B\subseteq \A;\B\not=\emptyset}
     (\sigma(\min(\B))-1)
     \prod_{\substack{b\in\B\\b\neq \min{\B}}}(\sigma(b)^2-1).
\end{equation}

If $\sigma$ is constant equal to $p$, one gets
\begin{equation}
|\HS(\A,p)|=1+(p-1)\sum_{k=1}^{|\A|} \binom{|\A|}{i} (p^2-1)^{k-1}=
\frac{p^{2{|\A|}}+p}{p+1}.
\end{equation}

If we let $S_p(n)=|\HS(\A,p)|$ for any $\A$ of size $n$, then the exponential
generating function of $S_p(n)$ is given by
\begin{align}
	\sum_{n=0}^\infty S_p(n)\frac{z^n}{n!}=\frac{(p\exp(z)+\exp(p^2z)}{p+1}.
\end{align}
\end{cor}

\begin{prop}
\label{hyposylv-idem}
The monoid $\HS(\A,\sigma)$ contains exactly
\begin{align}\label{eq:idemHS}
1+\sum_{\substack{\B\subseteq\A\\\B\neq\emptyset}} \prod_{\substack{b\in\B\\b\neq\min(\B)}}(\sigma(b)-1)
\end{align}
idempotents. If $\sigma(x)=p$ for all $x\in\A$, then \cref{eq:idemHS}
simplifies into
\begin{align}
1+\sum_{k=1}^{|\A|} \binom{|\A|}{k}(p-1)^{k-1}.
\end{align}
\end{prop}
\begin{proof}
An idempotent of $\HS(\A,\sigma)$ is an inflation $e^\#$ of an idempotent $e$
of $\HS(\A,2)$ such that $\ev_{\sigma}(e^\#)=\prod_{b_i\in\cont(e)}
b_i^{\sigma(b_i)-1}$.
Using \cref{prop:HS2idem}, one directly obtains the result.
\end{proof}

As a corollary of \cref{thm:isom}.(ii), one has

\begin{cor}\label{hyposylv-iso}
Let $\sigma, \tau:\A \to \NN_{\geq2}$, then
$\HPl(\A,\sigma)\cong\HPl(\A,\tau)$ if and only if $\sigma=\tau$.
\end{cor}

\section{The \texorpdfstring{$\sigma$}{s}-Metasylvester Monoids}
\label{section-metasylv}
\label{metasylvester}

\subsection{The Metasylvester Monoid}\label{ss-metasylv}

Let the \emph{metasylvester congruence}, $\equiv_{\MS}$, be the least
congruence containing the relations
\begin{equation}\label{eq:meatRel1}
ab\,U\,a = ba\,U\,a,
\end{equation}
for $a<b$ and
\begin{equation}
b\,U\,ac\,V\,b = b\,U\,ca\,V\,b,
\end{equation}
for $a<b<c$, for any $U,V\in\A^*$.

The \emph{metasylvester} monoid, denoted $\MS(\A)$, is the quotient
of $A^*$ by $\equiv_{\MS}$.

As in \cref{def:quotient-monoid}, for $\sigma: \A \to \NN_{\geq 2}$,
we define the \emph{$\sigma$-metasylvester congruence} to be the least
congruence on $\A ^ *$ containing the metasylvester congruence $\equiv_{\MS}$
and $a ^ {\sigma(a)} = a$ for every $a\in \A$; denoted $\equiv_{\MS(\sigma)}$.
The \emph{$\sigma$-metasylvester monoid over $\A$} is the quotient
$\A^* / {\equiv_{\MS(\sigma)}}$ and is denoted by $\MS(\A, \sigma)$.

\paragraph{An Invariant of the Metasylvester Monoid\\}

In \cite{NT}, the fourth author and Thibon introduced an algorithm that
provides an invariant of metasylvester classes.

\begin{defi}[Metasylvester tree building]\label{MetaTreeBuild}
Let \mbox{$W\in\A^*$} with \mbox{$\cont(W)=\{w_1<w_2<\cdots<w_n\}$}, and let $k_i=|W|_{w_i}$ for all $i$.
The metasylvester tree of $W$, denoted $\Mm(W)$, is built recursively.

\begin{enumerate}[(1)]
\item Label the root with the maximal letter $w_n$ of $W$ and give it $k_n$
children.  Set $i=n-1$.

\item Consider the smallest $j>i$ such that the rightmost occurrence of $w_i$
lies \emph{between} two occurrences of $w_j$.

\begin{enumerate}[(i)]
\item If such $j$ exists, the $(k+1)$-st subtree of $w_j$, where $k$
is the number of occurrences of $w_j$ to the left of the rightmost occurrence of
$w_i$ becomes the tree obtained by inserting the restriction of $W$ to $w_i$
and all letters that are already in the $(k+1)$-st subtree of $w_j$.
End the algorithm if $i=1$. Otherwise, start again Step 2 with $i=i-1$.

\item If no such $j$ exists, if this occurrence is to the left of
all maximal letters of $W$, then the leftmost node of the root becomes the
tree obtained by inserting the restriction of $W$ to $w_i$ and all letters
that are already in the leftmost subtree of the root. The same happens on the
rightmost node if the rightmost occurrence of $w_i$ is to the right of all
maximal letters of $W$.
End the algorithm if $i=1$. Otherwise, start again Step 2 with $i=i-1$.
\end{enumerate}
\end{enumerate}
\end{defi}

As described in Section 3.3.1 of \cite{NT}, from such tree $T$, one can build
the canonical word $\RW(T)$ of its class by labelling each sector between
two edges of a given node by the label of this node and read the sectors from
left to right. Furthermore, this element is the maximal element of its class
which is an interval of the weak order.
See \cref{14324} for an example of both the metasylvester
tree and how to build its associated canonical word.

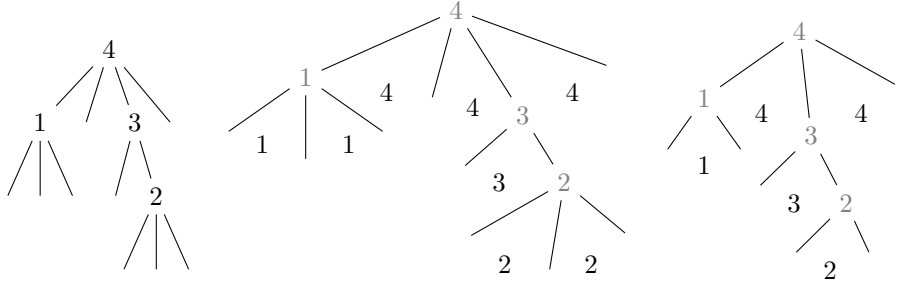
\begin{figure}[ht]
\centering
\begin{forest}
[4[1[][][]][][3[][2[][][]]][]]
\end{forest}
\begin{forest}
[4,color=gray
	[1,color=gray, l=2mm, anchor=north
				[,l=2mm]
				[1,l=4mm, no edge]
				[,l=6mm]
				[1,l=4mm, no edge]
				[,l=2mm]
			]
	[4,no edge,l=4 mm, anchor=north east]
	[,l=7 mm, anchor=north east]
	[4,no edge, l=6 mm, anchor=north east]
	[3,color=gray, l=7 mm, anchor=north west
		[,l=2 mm, anchor=north east]
		[3,no edge,l=2mm, anchor=north]
		[2,color=gray, l=2 mm, anchor=north west
			[,l=2 mm, anchor=north]			
			[2,no edge,l=4 mm, anchor=north]
			[,l=7 mm, anchor=north]
			[2,no edge,l=4 mm, anchor=north]
			[,l=2 mm, anchor=north]
		]
	]
	[4,no edge,l=4mm, anchor=north west]
	[,l=2mm, anchor=north west]
]
\end{forest}
\begin{forest}
[4,color=gray
	[1,color=gray, l=2mm, anchor=north
				[,l=2mm]
				[1,l=4mm, no edge]
				[,l=2mm]
			]
	[4,no edge,l=4 mm, anchor=north]
	[3,color=gray, l=7 mm, anchor=north west
		[,l=2 mm, anchor=north east]
		[3,no edge,l=2mm, anchor=north]
		[2,color=gray, l=2 mm, anchor=north west
			[,l=2 mm, anchor=north]
			[2,no edge,l=2 mm, anchor=north]
			[,l=2 mm, anchor=north]
		]
	]
	[4,no edge,l=4mm, anchor=north west]
	[,l=2mm, anchor=north west]
]
\end{forest}
\caption{\label{14324}On the left, the metasylvester tree $T$ of $112424324$.
In the middle, the algorithm to get $\RW(T)=11443224$. On the right, the
2-reduction of $T$ associated with the 2-canonical word $14324$.}
\end{figure}

\subsection{Quotient Type}\label{ss-metasylv-type}

For $|\A|\geq 2$, $\MS(\A)$ is not a type 1 monoid. This is for the same
reason as for $\Sylv$; for $a<b$, using \cref{ss:Type1Monoids} notation,
one has $\theta(bbab)=\theta(babb)=bab$ and
$\ev_\sigma(bbab)=\ev_\sigma(babb)=abbb$ but they are not
$\sigma$-metasylvester equivalent if $\sigma(b)\neq 2$.
We shall see that it is of type 2.

\begin{prop}
\label{cor-metasylv-gathered}
The metasylvester monoid is a \gathered monoid.
\end{prop}

The proof is very similar to the proof in the hyposylvester case.

\begin{proof}
First, given any word $W$, one obtains $G(W)=\RW(\Mm(W))$ by orienting the
relations as
\begin{enumerate}[(1)]
\item $ab\,U\,a \to ba\,U\,a$, \quad for $a<b$ and $U\in\A^*$, and,
\item $b\,U\,ac\,V\,b \to b\,U\,ca\,V\,b$, \quad for $a<b<c$, for any
$U,V\in\A^*$.
\end{enumerate}

Then if two values are equal and consecutive and, in the path from $W$ to
$G(W)$, they get separated by another value, one easily sees that one can
apply the same rewriting to the other equal value, putting these back
consecutively. So if two equal letters are consecutive in $W$, they still are
consecutive in $G(W)$. Since $G(W)$ is the maximal element of its class,
$G(W)\in\G_{\MS}$.

Now, the argument is the same as in the proof of the hyposylvester case: there
are no rewritings forbidden by consecutive equal letters or requiring
consecutive equal letters so the canonical element of an expansion (resp. a
contraction) is the expansion (resp. the contraction) of its canonical
element.
\end{proof}

The previous insertion algorithm gave us a canonical word for each
metasylvester class. When quotienting, one can shave these trees by
replacing any $\sigma(x)$ consecutive unlabelled leaves that have $x$ as their
parents by only one unlabelled leaf.
Call such a tree a $\sigma$-\emph{reduced metasylvester tree}. One can obtain
a word from a $\sigma$-reduced metasylvester tree by the same procedure as
in the metasylvester monoid case; denote this word by $G_\sigma(W)$ (or
$G_\sigma(T)$ if $G(W)=T$).

For example, if $\sigma(x)=2,$ for all values $x\in\A$, the corresponding tree
is a tree where for each labelled node, only its rightmost and leftmost edges
can be unlabelled.
See the right tree of \cref{14324} for an example of a 2-reduced metasylvester
tree.

Then $W$ and $G_\sigma(W)$ are in the same $\sigma$-reduced metasylvester
class because each sector separated by an unlabelled edge are labelled the
same and are consecutive in the canonical word of the metasylvester class.

Therefore, if $T$ is a $\sigma$-reduced metasylvester tree,
$\RW(T)\in\G_{\MS(\sigma)}$.

\begin{thm}
\label{metasylv-s}
Let $\A$ be an alphabet. The set of $\sigma$-reduced metasylvester trees is a
section of $\MS(\A,\sigma)$.
\end{thm}

\begin{proof}
The proof is straightforward using the fact that the metasylvester monoid is
\gathered.
\end{proof}

\subsection{The 2-Metasylvester Monoid}

As we did with the other monoids, we first study the 2-metasylvester monoid.

\paragraph{$\L$-triviality\\}

Clearly, $\MS(\A,2)$ is not $\J$-trivial since otherwise, $\Sylv(\A,2)$,
which is a quotient, would be as well, which is not the case. We shall
show that it is $\L$-trivial. In order to do so, we first need a rather
technical lemma and a notation. Denote by $p(i,W)$ the label of the
parent of $i$ in $\Mm(W)$.

\begin{lemma}\label{metaparents}
    For $i,x\in\A$ and a word $W$, $p(i,xW)\leq p(i,W)$.
\end{lemma}

\begin{proof}
Let $W$ be the 2-reduced canonical word of its class.
For $x\in\A$, let $T$ be the 2-reduced metasylvester tree of $xW$.
   
First, if $x\in\cont(W)$ and $p(i,xW)\neq p(i,W)$, then, by
\cref{MetaTreeBuild}, $i<x$ and the rightmost occurrence of $i$ needs to lie
at the left of the first $x$.

For such an $i$, assume that there exists $i<j$ such that the right occurrence
of $i$ lies between two occurrences of $j$. Take this $j$ to be minimal and
these occurrences consecutive. If $j<x$, by \cref{MetaTreeBuild}, one has
$p(i,xW)=p(i,W)$.

Now, if $j>x$ or if such $j$ does not exist, $i$ is inserted in $T$ on the
second child of $x$, its first child being an empty leaf. If $p(i,W)\geq x$,
then $p(i,xW)\leq x\leq p(i,W)$. And if $p(i,W)<x$, let $k=p(i,W)$. The
only letters between the rightmost $i$ and the rightmost $k$ are smaller than
$k$.

Thus, for any $k<l\leq x$, if $k$ is in a subtree of $l$ in $T$, so does $i$.
So $p(i,xW)\leq k=p(i,W)$, which concludes the case where $x\in\cont(W)$.

Now, if $x\notin\cont(W)$ and $p(i,xW)\neq p(i,W)$, then $i<x$ and the
rightmost occurrence of $i$ needs to lie at the left of all the $y>x$ in $W$.
In the 2-reduced metasylvester tree $T'$ of $W$, they all are in the left
subtree of a $k>x$. Take such $k$ minimal.

Then, by \cref{MetaTreeBuild}, $T$ is obtained from $T'$ by simply cutting out
the left subtree $L$ of $k$, labeling the created leaf $x$ and, placing $L$ as
the right subtree of $x$, its left subtree being a leaf. This concludes the
proof.
\end{proof}

\begin{thm}
    The monoid $\MS(\A,2)$ is $\L$-trivial.
\end{thm}
\begin{proof}
    Let $x\in\A$ and $W$ be the 2-reduced canonical word of its class.
    
    By \cref{metaparents}, one has $p(i,xW)\leq p(i,W)$ for all $i\in \cont(W)$.
    
    And if $p(i,xW)=p(i,W)$ for all $i\in\cont(W)$, clearly, there are no letters smaller than $x$ at the left of the leftmost $x$, except the ones that are in the left subtree of $x$ in $\Mm(W)$.
    Thus, in that case, $\Mm(xW)$ is obtained from $\Mm(W)$ by adding a new left subtree to $x$ if its left subtree contains labeled elements.

    Therefore, adding a letter to the left of $W$ cannot create a non-trivial cycle through elements of $\MS(\A,2)$.
\end{proof}

\paragraph{Enumeration}

We can do a finer enumeration by counting the number of 2-reduced canonical
words of length $k$ and content of size $n$.
For an alphabet $\A$ of size $n$,
let $$\Met_{n,k}:=\set{W\in\G_{\MS(2)}}{\cont(W)=\A,|W|=k}.$$

\begin{prop}\label{p:2metaenum}
Let
\begin{equation}
m_{n,k}=|\Met_{n,k}|.
\end{equation}
One has
\begin{equation}
m_{n+1,k+1}=(k+1)m_{n,k}+(k-1)m_{n,k-1},
\end{equation}
with $m_{1,1}=1$ and $m_{1,0}=0$.
\end{prop}

\begin{proof}
Let $W\in\G_{\MS(2)}$ and $T=\Mm(W)$ its corresponding tree. In order to add a
new letter 0 smaller than all the letters in $W$, one has two choices.
Either add $0$ in between two distinct letters $j<k$ (or at the beginning or
the end of $W$) or duplicate a letter $j$ and insert $0$ in between these two
occurrences of $j$. This can be seen on the tree $T$ by adding $0$
respectively on the rightmost or leftmost child of $j$ or by adding $0$ as an
extra middle child of $j$.

One can retrieve the former $T$ by simply removing $0$ and maybe shave one of
its edge off if needed (latter case).

The equation comes naturally.
\end{proof}

Using the finer enumeration, one can easily build \cref{table-metasylv}:
the number of reduced 2-metasylvester trees with $k$ sectors using all
values from $1$ to $|\A|$.

\begin{table}[ht]
\begin{equation*}
\begin{array}{lrrrrrrr}
\toprule
\multirow{2}{*}{k} & \multicolumn{7}{c@{}}{|\A|}\\
    \cmidrule(lr){2-8}
      &  1 & 2 & 3 & 4 & 5 & 6 & 7 \\
\midrule
     1 & 1 &   &   \\
     2 &   & 2 &   \\
     3 &   & 1 & 6 \\
     4 &   &   & 8 & 24 \\
     5 &   &   & 3 & 58 & 120 \\
     6 &   &   &   & 50 & 444 & 720  \\
     7 &   &   &   & 15 & 640 & 3708 & 5040 \\
\bottomrule
\end{array}
\end{equation*}
\caption{The number of reduced 2-metasylvester trees with $k$ sectors
using all values from $1$ to $|\A|$.}
\label{table-metasylv}
\end{table}

Note that the main diagonal is the factorial since it counts the number of
(usual) decreasing trees. More surprisingly, the numbers where
$k=2|\A|-1$ are the products of the odd numbers, $|\A|-1!!$.
In that case, the node containing the sectors of all letters but $n$ is
between two sectors of the same value. So by induction, to add a new smallest
value, one has to split any of the already existing nodes, that are in total
number $2|\A|-1$, hence the product formula.

\begin{cor}
\label{cor:enum-2MS}
If $\A$ is a finite alphabet of size $n$, we have,
\begin{equation}
|\MS(\A,2)| = \sum_{i=0}^{n}\binom{n}{i}\sum_{k=i}^{2i-1}m_{i,k}.
\end{equation}
\end{cor}

\paragraph{Idempotents}
\label{ss-meta-idem}

In order to compute the idempotents of $\MS(\A,2)$ we need this technical
lemma, which is a corollary of \cref{eq:meatRel1}.

\begin{lemma}\label{lem:MetaDeletion}
Let $x\in\A$ and $W\in\A^*$. If one writes $W=UxV$ such that
$x=\min(\cont(U))$ then $W\equiv_{\MS(2)}U'xV$ where $U'$ is obtained from $U$
by removing all its $x$'s.
\end{lemma}

\begin{prop}\label{prop:MS2idem}
The monoid $\MS(\A, 2)$ contains
\begin{equation}
1+\sum_{k=1}^{|\A|}\binom{|\A|}{k}L_k
\end{equation}
idempotents, where $L_{k}$ is the $k$-th large Schröder number.
\end{prop}

\begin{proof}
Let $T$ be a 2-reduced metasylvester tree and $W=\RW(T)$.
Assume $W$ is idempotent, that is $WW\equiv_{\MS(2)} W$.
Let $\cont(W)=\{w_1<w_2<\cdots<w_n\}$.

First, every labelled node of $T$ has a left unlabelled child.
Indeed, if a node $y$ has a left child labeled $x<y$,
that means that the rightmost $x$ is at the left of any $y$ in $W$.
But this would not be the case in $WW$, which contradicts $WW\equiv_{\MS(2)}
W$.
Because of that, the left to right prefix reading of $T$ has to be
$w_nw_{n-1}\cdots w_1$.
Indeed, assume there exists $i<j$ such that $w_i$ appears before $w_j$
in the left to right prefix reading of $T$.
Because every left subtrees are trivial, there exists $k>j$ such that $w_i$ is
in one of the middle subtrees of $w_k$ and $w_j$ is in the right subtree of
$w_k$.
Thus the rightmost occurrence of $w_i$ is at the left of any occurrences of
$w_j$ in $W$.
But this would not be the case in $WW$, which again contradicts
$WW\equiv_{\MS(2)} W$.

By recursively applying \cref{lem:MetaDeletion} from $w_1$ to $w_n$, one
proves that these two properties are enough to conclude that
$WW\equiv_{\MS(2)}W$.

Therefore, from an unlabelled tree $t$ with $k$ nodes and $j$ leaves, one can
build exactly $2^{k-j}$ distinct idempotents.
First, label its nodes such that its left to right prefix word is strictly
decreasing.
Then, add two children to all leaves and a left child to every internal node.
Finally, for each internal node, either add a right child or do not.

Clearly, for an idempotent $e$, shaving its leaves creates the tree $t$ from
which $e$ is built.

By recalling that the Narayana number $T_{k,j}$ counts the number of ordered
trees with $k$ nodes and $j$ leaves and that the large Schröder numbers can be
defined as $L_k=\sum_{j=1}^{k-1} 2^{k-j}T_{k,j}$,
one gets the statement.
\end{proof}

\subsection{The \texorpdfstring{$\sigma$}{s}-Metasylvester Monoids}
\label{ss-meta-enum-iso}

As for the sylvester monoid, because the metasylvester monoid is not of type
$1$, the formula for general $\sigma$ becomes significantly more complicated.
But one can compute the number of elements in $\MS(\A,p)$ using
\cref{cor:enum-2MS}.

\begin{cor}\label{meta-enum}
If $\A$ is a finite alphabet, we have,
\begin{equation}
|\MS(\A,p)| = \sum_{i=0}^{|\A|}\binom{|\A|}{i}\sum_{k=i}^{2i-1}(p-1)^km_{i,k}.
\end{equation}
\end{cor}

As a corollary of \cref{thm:isom}.(ii), one has

\begin{cor}\label{meta-iso}
Let $\sigma, \tau:\A \to \NN_{\geq2}$, then $\MS(\A,\sigma)\cong\MS(\A,\tau)$
if and only if $\sigma=\tau$.
\end{cor}

\section{References}
\nocite{oeis}
\printbibliography[heading=none]
\end{document}